\documentclass[reqno]{amsart}
\usepackage{a4wide}
\usepackage{stmaryrd}
\usepackage{graphicx}
\usepackage{amssymb}

\usepackage{mathtools}
\usepackage{subfigure}

\numberwithin{equation}{section}

\usepackage{hyperref}

\usepackage{tikz-cd}
\usetikzlibrary{decorations.pathreplacing,decorations.markings}

\tikzset{
  on each segment/.style={
    decorate,
    decoration={
      show path construction,
      moveto code={},
      lineto code={
        \path [#1]
        (\tikzinputsegmentfirst) -- (\tikzinputsegmentlast);
      },
      curveto code={
        \path [#1] (\tikzinputsegmentfirst)
        .. controls
        (\tikzinputsegmentsupporta) and (\tikzinputsegmentsupportb)
        ..
        (\tikzinputsegmentlast);
      },
      closepath code={
        \path [#1]
        (\tikzinputsegmentfirst) -- (\tikzinputsegmentlast);
      },
    },
  },
  mid arrow/.style={postaction={decorate,decoration={
        markings,
        mark=at position .5 with {\arrow[#1]{stealth}}
      }}},
}

\theoremstyle{plain}
\newtheorem{theorem}{Theorem}[section]
\newtheorem{lemma}[theorem]{Lemma}
\newtheorem{corollary}[theorem]{Corollary}
\newtheorem{proposition}[theorem]{Proposition}

\theoremstyle{definition}
\newtheorem{definition}[theorem]{Definition}
\newtheorem{example}[theorem]{Example}
\newtheorem{assumption}[theorem]{Assumption}

\theoremstyle{remark}
\newtheorem{remark}[theorem]{Remark}

\DeclareMathOperator{\Aut}{Aut}
\DeclareMathOperator{\Dehn}{Dehn}
\DeclareMathOperator{\Hom}{Hom}
\DeclareMathOperator{\Inn}{Inn}
\DeclareMathOperator{\Isom}{Isom}
\DeclareMathOperator{\Iso}{Twist}

\DeclareMathOperator{\Out}{Out}
\DeclareMathOperator{\sgn}{sgn}

\usepackage[mathscr]{euscript}
\usepackage{ragged2e}

\renewcommand{\tilde}{\widetilde}
\renewcommand{\bar}{\overline}

\newcommand\Ch{\mathcal{C}\mathsf{h}}

\newcommand\even{{\mathsf{even}}}
\newcommand\out{{\mathsf{out}}}
\newcommand\odd{{\mathsf{odd}}}
\newcommand\CLTTF{\textsf{CLTTF} }

\title{The automorphism groups of Artin groups of edge-separated CLTTF graphs}
\author{Byung Hee An}
\email{anbyhee@knu.ac.kr}
\address{Department of Mathematics Education, Kyungpook National University, Republic of Korea}

\author{Youngjin Cho}
\email{y\_cho@knu.ac.kr}
\address{Department of Mathematics Education, Kyungpook National University, Republic of Korea}

\begin{document}
\begin{abstract}
We provide an explicit presentation of the automorphism group of an edge-separated CLTTF Artin group.
\end{abstract}

\maketitle

\tableofcontents

\section{Introduction}

Let $\Gamma$ be a simple graph such that every edge $e$ carries an integer label $m(e)\ge2$. An {\em Artin group} $A_\Gamma$ with a defining graph $\Gamma$ is generated by vertices of $\Gamma$ and related by
$$\underbrace{sts\cdots}_{m(e)}=\underbrace{tst\cdots}_{m(e)}$$
for each edge $e$ joining $s$ and $t$. A set of generators is called that of {\em Artin generators} if a defining graph can be recovered by using them as vertices. For example, the 4-strand braid group is an Artin group defined by the triangle with edge labels 2, 3, 3. If all edge labels are 2, $A_\Gamma$ is called a {\em right-angled Artin group}. An Artin group is {\em rigid} if it has a unique defining graph, or equivalently, if a set of Artin generators is sent to any other set of Artin generators by an automorphism of the Artin group. Right-angled Artin groups \cite{Droms87} and Artin groups of finite type \cite{BMcMN02} are known to be rigid. In general, Artin groups need not be rigid.

From now on, we fix a finite set $V$ of vertices and assume that a graph $\Gamma$ is edge-labeled whose set of edges is denoted by $E(\Gamma)$. 
Suppose that a graph $\Gamma$ has two subgraphs $\Gamma_1$ and $\Gamma_2$ with intersection $\Gamma_0$ such that $A_{\Gamma_0}$ is an Artin subgroup of finite type.
% and each vertex in $\Gamma\setminus(\Gamma_1\cup \Gamma_2)$ that is adjacent to a vertex in $\Gamma_1$ is adjacent to every vertex in $\Gamma_2$ by an edge labeled $2$. 
In \cite{Cr05}, the author proposes a typical way of obtaining a new defining graph from $\Gamma$ under this circumstance. Recall that there is a unique element $\lambda$ in $A_{\Gamma_0}$, which is the longest element in the associated Coxeter group, such that the conjugation by $\lambda$ permutes elements of $\Gamma_2$.
A new set $S'$ obtained from $V$ by replacing elements of $V(\Gamma_2)$ by their conjugates by $\lambda$ generates $A_\Gamma$ and then $S'$ determines a new defining graph $\Delta$ that is called an {\em edge-twist} of $\Gamma$ with respec to the triple $(\Gamma_1, \Gamma_0, \Gamma_2)$. In fact $\Delta$ is obtained from $\Gamma$ by replacing each edges joining a vertex $v$ in $\Gamma_0$ and a vertex $w$ in $\Gamma_2$ by a new edge joining $v$ and $\lambda w\lambda^{-1}$. We may identify $V(\Delta)$ with $V$ since only edges are altered. 
There is an obvious isomorphism $:A_\Gamma\to A_{\Delta}$ called a {\em twist isomorphism}, that sends each $v\in V(\Gamma_2)$ to $\lambda v\lambda^{-1}$ and fixes other generators. It is a conjecture that two defining graphs of an Artin group are {\em twist-equivalent}, that is, related via a series of twists.

There have been extensive researches on automorphism groups of free abelian groups, free groups, and more generally, right-angled Artin groups. In particular isometric actions on appropriate spaces by outer automorphisms are studied to understand geometric structures of groups of outer automorphisms. There are also many complete results on automorphism groups of some Artin groups of finite type. Nielsen automorphisms or Whitehead automorphisms on free groups can be adapted to form a set of generators of automorphism groups when they are appropriate. They are usually classified as one of the following types: permutations of generators, inversions, transvections, and partial conjugations. For right-angled Artin groups, peak reduction arguments can be employed to obtain a complete set of relations among generators \cite{Day09,Day14}.

On the other hand, there are very few results on automorphism groups of non-rigid Artin groups. John Crisp gave the first noticeable result in \cite{Cr05}. He considered \CLTTF Artin groups defined by graphs that are Connected, has edge labels $\ge3$ (Large Type), and is Triagle Free. \CLTTF Artin groups form a somewhat manageable family of non-rigid Artin groups in studying their automorphism groups. In fact there are no transvections and twists occur only along edges with odd labels. Furthermore two defining graphs of a \CLTTF group are twist-equivalent \cite{Cr05}. He showed that the isomorphism groupoid of a \CLTTF Artin group is generated by graph automorphisms, inversions, partial conjugations, and twist isomorphisms. However given a  \CLTTF Artin group, it is not feasible to obtain a presentation of its automorphism group by using the groupoid since its automorphism can be given by any circuit including loops in the graph of groupoid.

In this paper, we provide concrete and explicit description and group presentations of the (outer) automorphism group, whose generators are vertices, edge-twists, and certain graph isomorphisms.

\begin{theorem}[Theorem~\ref{theorem:main theorem}]
Let $\Gamma$ be a \CLTTF graph without separating vertices.
Then the automorphism group $\Aut(A_\Gamma)$ and outer automorphism group admit the following finite group presentations:
\begin{align*}
\Aut(A_\Gamma)&\cong
\left\langle V, S(\Gamma), \iota ~\middle|~ R_0(\Gamma), R_1(\Gamma), R_2(\Gamma), R_3(\Gamma), R_4(\Gamma), \tilde R_\Phi(\Gamma)
\right\rangle,\\
\Out(A_\Gamma)&\cong\left\langle
S(\Gamma), \iota ~\middle|~
R_1(\Gamma), R_2(\Gamma), R_3(\Gamma), R_4(\Gamma), R_\Phi
\right\rangle,
\end{align*}
where the sets $S(\Gamma)$, $R_i(\Gamma)$ for $0\le i\le 4$, $R_\Phi(\Gamma)$ and $\tilde R_\Phi(\Gamma)$ are given in Section~\ref{section:presentation}.
\end{theorem}

The rest of the paper is organized as follows.
In Section~\ref{section:CLTTF graphs}, we review basics on \CLTTF graphs including chunk trees, graph isomorphisms, edge-twists and their pull-backs and push-forwards. We also define the subgroup $\Iso(\Gamma)$ of the permutation group $\mathfrak{S}_V$ consisting of graph isomorphisms whose source and target are edge-twist equivalent. We further define the category $\mathscr{G}$ of \CLTTF graphs whose morphisms are isomorphisms freely generated by graph isomorphisms and edge-twists.

In Section~\ref{section:CLTTF Artin groups}, we review \CLTTF Artin groups and their isomorphisms, and define the category $\mathscr{A}$ of \CLTTF Artin groups whose morphisms are generated by graph isomorphisms and partial conjugations. We prove the equivalence between categories $\mathscr{G}$ and $\mathscr{A}$, and the relationship between our and Crisp's categories of \CLTTF Artin groups is briefly explained.

In Section~\ref{section:automorphism groups}, we introduce the \emph{twisted intersection product} between graph isomorphisms in $\Iso(\Gamma)$ and finally we provide the group presentations for both $\Aut(A_\Gamma)$ and $\Out(A_\Gamma)$.

\subsection*{Acknowledgements}
The first author was supported by Kyungpook National University Research Fund, 2020.

\section{\CLTTF Graphs}\label{section:CLTTF graphs}
Throughout this paper, we fix a finite set $V$ with $\#(V)\ge 3$ and assume that every graph has the set $V$ of vertices unless mentioned otherwise.

\subsection{\CLTTF graphs and chunk trees}
Let $\Gamma=(V, E, m)$ be a simple graph with an edge-label $m:E \rightarrow \mathbb{Z}_{\geq 2}$. We call $\Gamma$ \CLTTF if it satisfies the following:
\begin{itemize}
\item it is \textsf{C}onnected,
\item it is of \textsf{L}arge \textsf{T}ype, i.e., $m(e)\ge 3$ for every edge $e\in E$, 
\item it is \textsf{T}riangle-\textsf{F}ree, i.e., there are no full subgraphs of three vertices which look like a triangle.
\end{itemize}

A \emph{decomposition} of $\Gamma$ along a subgraph $\Gamma_0$ is a triple $(\Gamma_1, \Gamma_0, \Gamma_2)$ such that $\Gamma_1$ and $\Gamma_2$ are full subgraphs different from $\Gamma_0$ whose union and intersection are $\Gamma$ and $\Gamma_0$, respectively,
\begin{align*}
\Gamma &= \Gamma_1 \cup \Gamma_2,& 
\Gamma_0 &= \Gamma_1 \cap \Gamma_2.
\end{align*}
We call a vertex $v$ or an edge $e$ \emph{seprating} if there exists a decomposition with $\Gamma_0=v$ or $e$ and we say that a \CLTTF graph $\Gamma$ is \emph{edge-separated} if there are no separating vertices.

\begin{example}[edge-separated \CLTTF graphs]
Let us consider the graph $\Gamma$ below
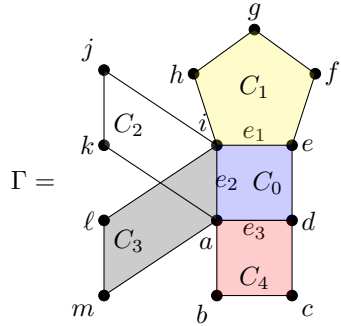
\[
\Gamma=\begin{tikzpicture}[baseline=-.5ex]
\draw[fill] (0, -0.5) circle (2pt) node[below=2ex, left=-.5ex] {$a$};
\draw[fill] (1, -0.5) circle (2pt) node[right] {$d$};
\draw[fill] (0, 0.5) circle (2pt) node[above=2ex, left] {$i$};
\draw[fill] (1, 0.5) circle (2pt) node[right] {$e$};
\draw[fill] (0, -1.5) circle (2pt) node[below left] {$b$};
\draw[fill] (1, -1.5) circle (2pt) node[below right] {$c$};
\draw[fill](1, 0.5) ++(72:1) circle (2pt) node[right] {$f$} ++(144:1) circle(2pt) node[above] {$g$} ++(216:1) circle(2pt) node[left] {$h$};
\draw[fill] (-1.5,1.5) circle (2pt) node[above left] {$j$};
\draw[fill] (-1.5,0.5) circle (2pt) node[left] {$k$};
\draw[fill] (-1.5,-0.5) circle (2pt) node[left] {$\ell$};
\draw[fill] (-1.5,-1.5) circle (2pt) node[below left] {$m$};
\draw (0.5,0.5) node[above=-.5ex] {$e_1$} (0,0) node[right=-1ex] {$e_2$} (0.5,-0.5) node[below=-.5ex] {$e_3$};
\draw[color=black, fill=red, fill opacity=0.2](0, -1.5) rectangle node[opacity=1, below] {$C_4$} (1,-0.5);
\draw[color=black, fill=yellow, fill opacity=0.2](1, 0.5) -- ++(72:1) -- ++(144:1) -- ++(216:1) -- ++(288:1);
\draw (0.5, 1) node[opacity=1, above] {$C_1$};
\draw[color=black, fill=white, fill opacity=0.2](0, 0.5) -- ++(-1.5, 1) -- ++(-90:1) node[opacity=1, above right] {$C_2$} -- ++(1.5, -1);
\draw[color=black, fill=black, fill opacity=0.2](0, -0.5) -- ++(-1.5, -1) -- ++(90:1) node[opacity=1, below right] {$C_3$} -- ++(1.5, 1);
\draw[color=black, fill=blue, fill opacity=0.2](0, -0.5) rectangle node[opacity=1, right=-1ex] {$C_0$} (1,0.5);
\end{tikzpicture}
\]
whose edges are all labeled as $3$.
Then this is an edge-separated \CLTTF graph.
\end{example}

For simplicity, we will assume the following:
\begin{assumption}\label{assumption:edge-separated}
Every \CLTTF graph is edge-separated.
\end{assumption}

\begin{definition}[Chunk]
Let $C$ be a connected full subgraph of $\Gamma$. We shall say that $C$ is \emph{indecomposable} if, for every decomposition $(\Gamma_1,e,\Gamma_2)$ 
%of $\Delta$,
%(by whice we imply that $\Delta_1, \Delta_2$ are full subgraphs of $\Delta$ such that $\Delta_1 \cup \Delta_2 = \Delta$ and $\Delta_1 \cap \Delta_2 = T$),
of $\Gamma$ over a separating edge $e$, either $C \subset \Gamma_1$ or $C \subset \Gamma_2$.

By a \emph{chunk} of $\Gamma$ we mean a maximal indecomposable (connected and full) subgraph of $\Gamma$.
\end{definition}

Notice that a chunk $C$ of a \CLTTF graph $\Gamma$ is again a \CLTTF graph with at least 3 vertices.
Moreover, any two chunks of $\Gamma$ intersect, if at all, along a single separating edge.
Hence we can constuct a new graph from $\Gamma$ consisting of chunks and separating edges as follows:

\begin{definition}[Chunk graph]
Let $\Gamma$ be an edge-separated \CLTTF graph.
The \emph{chunk graph} $\Ch_\Gamma$ is a graph constructed as follows:
\begin{itemize}
\item The set $V(\Ch_\Gamma)$ of vertices consists of chunks and separating edges
\[
V(\Ch_\Gamma)=\{C\subset \Gamma\mid C\text{ is a chunk}\}\cup
\{e\subset \Gamma\mid e\text{ is a separating edge}\}.
\]
\item The set $E(\Ch_\Gamma)$ of edges consists of the pairs $(e,C)$ of a separating edge $e$ and a chunk $C$ whenever $e\subset C$
\[
E(\Ch_\Gamma)=\{
(e,C)\mid e\subset C, e\text{ is a separating edge}, C\text{ is a chunk}
\}.
\]
\end{itemize}
\end{definition}

\begin{example}\label{example:chunk graph}
An example of a chunk tree is as follows:
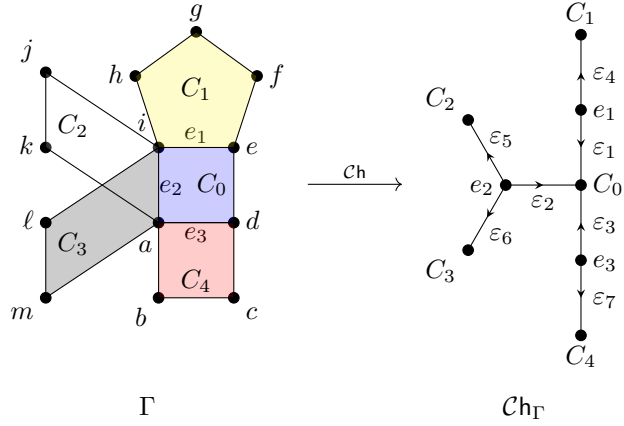
\[
\begin{tikzcd}[row sep=0pc, column sep=3pc]
\begin{tikzpicture}[baseline=-.5ex]
\draw[fill] (0, -0.5) circle (2pt) node[below=2ex, left=-.5ex] {$a$};
\draw[fill] (1, -0.5) circle (2pt) node[right] {$d$};
\draw[fill] (0, 0.5) circle (2pt) node[above=2ex, left] {$i$};
\draw[fill] (1, 0.5) circle (2pt) node[right] {$e$};
\draw[fill] (0, -1.5) circle (2pt) node[below left] {$b$};
\draw[fill] (1, -1.5) circle (2pt) node[below right] {$c$};
\draw[fill](1, 0.5) ++(72:1) circle (2pt) node[right] {$f$} ++(144:1) circle(2pt) node[above] {$g$} ++(216:1) circle(2pt) node[left] {$h$};
\draw[fill] (-1.5,1.5) circle (2pt) node[above left] {$j$};
\draw[fill] (-1.5,0.5) circle (2pt) node[left] {$k$};
\draw[fill] (-1.5,-0.5) circle (2pt) node[left] {$\ell$};
\draw[fill] (-1.5,-1.5) circle (2pt) node[below left] {$m$};
\draw (0.5,0.5) node[above=-.5ex] {$e_1$} (0,0) node[right=-1ex] {$e_2$} (0.5,-0.5) node[below=-.5ex] {$e_3$};
\draw[color=black, fill=red, fill opacity=0.2](0, -1.5) rectangle node[opacity=1, below] {$C_4$} (1,-0.5);
\draw[color=black, fill=yellow, fill opacity=0.2](1, 0.5) -- ++(72:1) -- ++(144:1) -- ++(216:1) -- ++(288:1);
\draw (0.5, 1) node[opacity=1, above] {$C_1$};
\draw[color=black, fill=white, fill opacity=0.2](0, 0.5) -- ++(-1.5, 1) -- ++(-90:1) node[opacity=1, above right] {$C_2$} -- ++(1.5, -1);
\draw[color=black, fill=black, fill opacity=0.2](0, -0.5) -- ++(-1.5, -1) -- ++(90:1) node[opacity=1, below right] {$C_3$} -- ++(1.5, 1);
\draw[color=black, fill=blue, fill opacity=0.2](0, -0.5) rectangle node[opacity=1, right=-1ex] {$C_0$} (1,0.5);
\end{tikzpicture}
\arrow[r,"\Ch"]&
\begin{tikzpicture}[baseline=-.5ex]
\foreach \i in {-2,-1,0,1,2} {
\draw[fill] (1,\i) circle (2pt); 
}
\draw[fill] (0, 0) circle (2pt) (240:1) circle (2pt) (120:1) circle (2pt);
\draw[postaction={on each segment={mid arrow}}] (0,0) node[left] {$e_2$} -- node[pos=0.75, right] {$\varepsilon_6$} (240:1) node[below left] {$C_3$};
\draw[postaction={on each segment={mid arrow}}] (0,0) -- node[midway, below] {$\varepsilon_2$} (1,0);
\draw[postaction={on each segment={mid arrow}}] (0,0) -- node[pos=0.75, right] {$\varepsilon_5$} (120:1) node[above left] {$C_2$};
\draw[postaction={on each segment={mid arrow}}] (1,1) -- node[midway, right] {$\varepsilon_4$} (1,2) node[above] {$C_1$};
\draw[postaction={on each segment={mid arrow}}] (1,1) node[right] {$e_1$} -- node[midway, right] {$\varepsilon_1$} (1,0);
\draw[postaction={on each segment={mid arrow}}] (1,-1) -- node[midway, right] {$\varepsilon_3$} (1,0) node[right] {$C_0$};
\draw[postaction={on each segment={mid arrow}}] (1,-1) node[right] {$e_3$} -- node[midway, right] {$\varepsilon_7$} (1,-2) node[below] {$C_4$};
\end{tikzpicture}\\
\Gamma & \Ch_\Gamma
\end{tikzcd}
\]
\end{example}

As seen in this example, one can easily observe the following: the chunk graph $\Ch_\Gamma$ is
\begin{itemize}
\item simple and connected,
\item bipartite with respect to being a separating edge and being a chunk, and
\item a tree whose leaves are chunks.
\end{itemize}

The first two observations are obviously true for any edge-separated \CLTTF graphs by the construction of the chunk graph. 
Indeed, by the definition of the chunk graph, $\Ch_\Gamma$ has no multiple edges, loops and edges which is connecting two chunks or two separating edges, respectively, and the connectivity of $\Gamma$ implies the connectivity of $\Ch_\Gamma$.
Moreover, since each separating edge of $\Gamma$ should be contained in at least two chunks of $\Gamma$, all univalent vertices of $\Ch_\Gamma$ are chunks.

\begin{theorem}
The chunk graph $\Ch_\Gamma$ is a tree whose leaves are chunks of $\Gamma$, and we will call $\Ch_\Gamma$ the \emph{chunk tree} for $\Gamma$.
\end{theorem}
\begin{proof}
Suppose that $\Ch_\Gamma$ is not a tree.
Since $\Ch_\Gamma$ is simple and bipartite, any embedded cycle in $\Ch_\Gamma$ has four or more verices, of which more than one vertex correspond to separating edges.
Therefore $\Ch_\Gamma$ can not be disconnected removing one separating edge of the cycle, which is a contradiction.
\end{proof}

One of the direct consequence of the theorem is that $\Ch_\Gamma$ has a unique vertex $*_\Gamma$ such that any vertex in $\Ch_\Gamma$ is far from $*_\Gamma$ at most $\operatorname{Diam}(\Ch_\Gamma)/2$, where $\operatorname{Diam}(\Ch_\Gamma)$ is the diameter of $\Ch_\Gamma$ with respect to the edge-length.
Hence the vertex $*_\Gamma$ plays the role of the \emph{center} of $\Ch_\Gamma$.

\begin{definition}[Center of the chunk tree]
We call the vertex $*_\Gamma$ the \emph{center} of the chunk tree $\Ch_\Gamma$.
\end{definition}

Let $\varepsilon=(e,C)$ be an edge of the chunk tree $\Ch_\Gamma$.
By cutting $\varepsilon$ in $\Ch_\Gamma$, we have two disjoint subgraphs $\Ch_{\Gamma,1}(\varepsilon)$ and $\Ch_{\Gamma,2}(\varepsilon)$ containing $e$ and $C$, respectively.
Then it induces a decomposition $(\Gamma_{1}(\varepsilon), e, \Gamma_{2}(\varepsilon))$ such that each $\Gamma_i(\epsilon)$ is the union of all chunks corresponding to vertices in $\Ch_{\Gamma,i}(\varepsilon)$. See Figure~\ref{figure:edges in chunk tree and decompositions} for example.

\begin{figure}[ht]
\[
\begin{tikzcd}[row sep=0pc, column sep=3pc]
\begin{tikzpicture}[baseline=-.5ex]
\foreach \i in {-2,-1,0,1,2} {
\draw[fill] (1,\i) circle (2pt); 
}
\draw[fill] (0, 0) circle (2pt) (240:1) circle (2pt) (120:1) circle (2pt);
\draw[postaction={on each segment={mid arrow}}] (0,0) node[left] {$e_2$} -- node[pos=0.75, right] {$\varepsilon_6$} (240:1) node[below left] {$C_3$};
\draw[red, thick, postaction={on each segment={mid arrow}}] (0,0) -- node[midway, below] {$\varepsilon_2$} (1,0);
\draw[postaction={on each segment={mid arrow}}] (0,0) -- node[pos=0.75, right] {$\varepsilon_5$} (120:1) node[above left] {$C_2$};
\draw[postaction={on each segment={mid arrow}}] (1,1) -- node[midway, right] {$\varepsilon_4$} (1,2) node[above] {$C_1$};
\draw[postaction={on each segment={mid arrow}}] (1,1) node[right] {$e_1$} -- node[midway, right] {$\varepsilon_1$} (1,0);
\draw[postaction={on each segment={mid arrow}}] (1,-1) -- node[midway, right] {$\varepsilon_3$} (1,0) node[right] {$C_0$};
\draw[postaction={on each segment={mid arrow}}] (1,-1) node[right] {$e_3$} -- node[midway, right] {$\varepsilon_7$} (1,-2) node[below] {$C_4$};
\end{tikzpicture}
\arrow[r,"\text{Cutting}"]&
\begin{tikzpicture}[baseline=-.5ex]
\draw[fill] (0, 0) circle (2pt) (240:1) circle (2pt) (120:1) circle (2pt);
\draw[postaction={on each segment={mid arrow}}] (0,0) node[left] {$e_2$} -- node[pos=0.75, right] {$\varepsilon_6$} (240:1) node[below left] {$C_3$};
\draw[postaction={on each segment={mid arrow}}] (0,0) -- node[pos=0.75, right] {$\varepsilon_5$} (120:1) node[above left] {$C_2$};
\end{tikzpicture}\displaystyle{\coprod}
\begin{tikzpicture}[baseline=-.5ex]
\foreach \i in {-2,-1,0,1,2} {
\draw[fill] (1,\i) circle (2pt); 
}
\draw[postaction={on each segment={mid arrow}}] (1,1) -- node[midway, right] {$\varepsilon_4$} (1,2) node[above] {$C_1$};
\draw[postaction={on each segment={mid arrow}}] (1,1) node[right] {$e_1$} -- node[midway, right] {$\varepsilon_1$} (1,0);
\draw[postaction={on each segment={mid arrow}}] (1,-1) -- node[midway, right] {$\varepsilon_3$} (1,0) node[right] {$C_0$};
\draw[postaction={on each segment={mid arrow}}] (1,-1) node[right] {$e_3$} -- node[midway, right] {$\varepsilon_7$} (1,-2) node[below] {$C_4$};
\end{tikzpicture}\\
\Ch_\Gamma & \Ch_{\Gamma,1}\quad \displaystyle\coprod \quad \Ch_{\Gamma,2}\\
\begin{tikzpicture}[baseline=-.5ex]
\draw[fill, red] (0, -0.5) circle (2pt) node[below=2ex, left=-.5ex] {$a$};
\draw[fill, red] (1, -0.5) circle (2pt) node[right] {$d$};
\draw[fill, red] (0, 0.5) circle (2pt) node[above=2ex, left] {$i$};
\draw[fill, red] (1, 0.5) circle (2pt) node[right] {$e$};
\draw[fill] (0, -1.5) circle (2pt) node[below left] {$b$};
\draw[fill] (1, -1.5) circle (2pt) node[below right] {$c$};
\draw[fill](1, 0.5) ++(72:1) circle (2pt) node[right] {$f$} ++(144:1) circle(2pt) node[above] {$g$} ++(216:1) circle(2pt) node[left] {$h$};
\draw[fill] (-1.5,1.5) circle (2pt) node[above left] {$j$};
\draw[fill] (-1.5,0.5) circle (2pt) node[left] {$k$};
\draw[fill] (-1.5,-0.5) circle (2pt) node[left] {$\ell$};
\draw[fill] (-1.5,-1.5) circle (2pt) node[below left] {$m$};
\draw (0.5,0.5) node[above=-.5ex] {$e_1$} (0,0) node[red, right=-1ex] {$e_2$} (0.5,-0.5) node[below=-.5ex] {$e_3$};
\draw[color=black, fill=red, fill opacity=0.2](0, -1.5) rectangle node[opacity=1, below] {$C_4$} (1,-0.5);
\draw[color=black, fill=yellow, fill opacity=0.2](1, 0.5) -- ++(72:1) -- ++(144:1) -- ++(216:1) -- ++(288:1);
\draw (0.5, 1) node[opacity=1, above] {$C_1$};
\draw[color=black, fill=white, fill opacity=0.2](0, 0.5) -- ++(-1.5, 1) -- ++(-90:1) node[opacity=1, above right] {$C_2$} -- ++(1.5, -1);
\draw[color=black, fill=black, fill opacity=0.2](0, -0.5) -- ++(-1.5, -1) -- ++(90:1) node[opacity=1, below right] {$C_3$} -- ++(1.5, 1);
\draw[color=red, thick, fill=blue, fill opacity=0.2](0, -0.5) rectangle node[opacity=1, right=-1ex] {$C_0$} (1,0.5);
\end{tikzpicture}
\arrow[r,"\text{Decomp.}"]&
\begin{tikzpicture}[baseline=-.5ex]
\draw[fill] (0, -0.5) circle (2pt) node[right] {$a$};
\draw[fill] (0, 0.5) circle (2pt) node[right] {$i$};
\draw[fill] (-1.5,1.5) circle (2pt) node[above left] {$j$};
\draw[fill] (-1.5,0.5) circle (2pt) node[left] {$k$};
\draw[fill] (-1.5,-0.5) circle (2pt) node[left] {$\ell$};
\draw[fill] (-1.5,-1.5) circle (2pt) node[below left] {$m$};
\draw (0,0) node[right] {$e_2$};
\draw[color=black, fill=white, fill opacity=0.2](0, 0.5) -- ++(-1.5, 1) -- ++(-90:1) node[opacity=1, above right] {$C_2$} -- ++(1.5, -1);
\draw[color=black, fill=black, fill opacity=0.2](0, -0.5) -- ++(-1.5, -1) -- ++(90:1) node[opacity=1, below right] {$C_3$} -- ++(1.5, 1) -- ++(0,-1);
\end{tikzpicture}
\displaystyle{\coprod_{
\begin{tikzpicture}[baseline=-.5ex, scale=0.5]
\draw[fill] (0,0.5) circle (2pt) node[above] {$i$} -- node[midway, right] {$e_2$} (0,-0.5) circle (2pt) node[below] {$a$};
\end{tikzpicture}
}}
\begin{tikzpicture}[baseline=-.5ex]
\draw[fill] (0, -0.5) circle (2pt) node[left] {$a$};
\draw[fill] (1, -0.5) circle (2pt) node[right] {$d$};
\draw[fill] (0, 0.5) circle (2pt) node[left] {$i$};
\draw[fill] (1, 0.5) circle (2pt) node[right] {$e$};
\draw[fill] (0, -1.5) circle (2pt) node[below left] {$b$};
\draw[fill] (1, -1.5) circle (2pt) node[below right] {$c$};
\draw[fill](1, 0.5) ++(72:1) circle (2pt) node[right] {$f$} ++(144:1) circle(2pt) node[above] {$g$} ++(216:1) circle(2pt) node[left] {$h$};
\draw[color=black, fill=red, fill opacity=0.2](0, -1.5) rectangle node[opacity=1, below] {$C_4$} (1,-0.5);
\draw[color=black, fill=blue, fill opacity=0.2](0, -0.5) rectangle node[opacity=1] {$C_0$} (1,0.5);
\draw[color=black, fill=yellow, fill opacity=0.2](1, 0.5) -- ++(72:1) -- ++(144:1) -- ++(216:1) -- ++(288:1);
\draw (0.5, 1) node[opacity=1, above] {$C_1$};
\end{tikzpicture}\\
\Gamma & \Gamma_1\qquad \displaystyle\coprod_{e_2} \qquad\Gamma_2
\end{tikzcd}
\]
\caption{An edge in $\Ch_\Gamma$ and a decomposition}
\label{figure:edges in chunk tree and decompositions}
\end{figure}
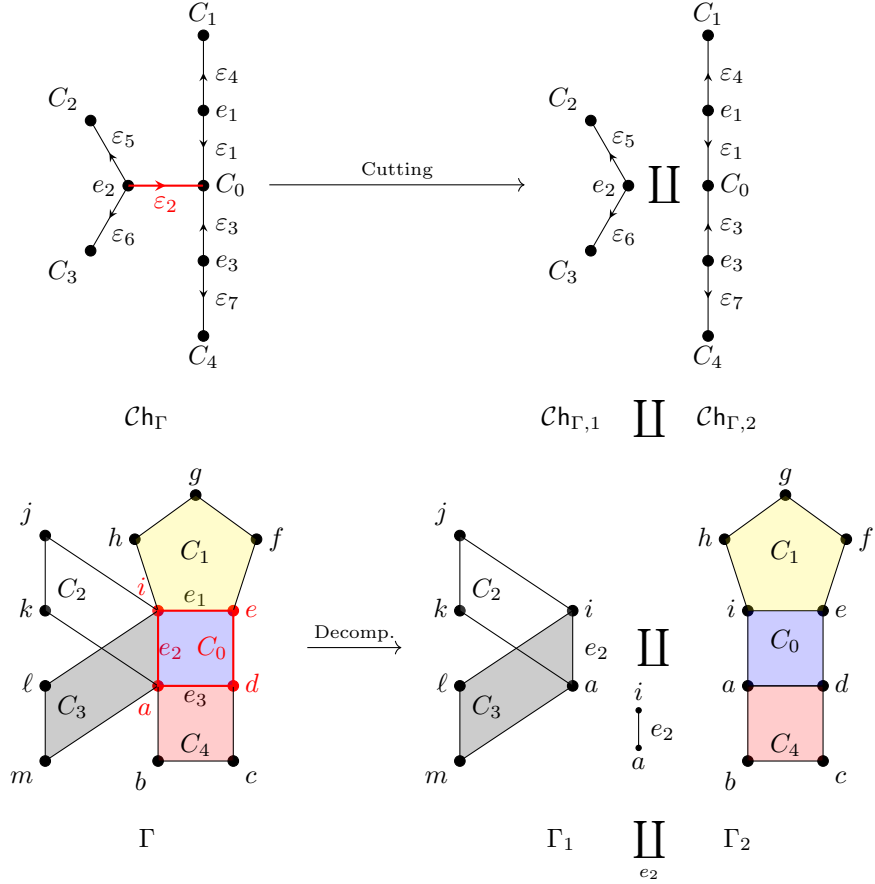

\begin{remark}\label{remark:not every decomposition comes from edges of chunk tree}
Each edge in $\Ch_\Gamma$ induces a decomposition of $\Gamma$, but not every decomposition of $\Gamma$ comes from an edge in $\Ch_\Gamma$.
\end{remark}

\subsection{Modifications of \CLTTF graphs}
We introduce two ways of modifications which are graph isomorphisms and edge-twists.

\subsubsection{Graph isomorphisms}
Let $\mathfrak{S}_V$ be the group of permutations on $V$ and let $(\alpha:V\to V)\in\mathfrak{S}_V$.
Then for each graph $\Gamma=(V, E_\Gamma, m_\Gamma)$, there exists a unique graph $\Delta=(V, E_\Delta, m_\Delta)$ such that the permutation $\alpha$ induces a graph isomorphism $\Gamma\to\Delta$, denoted by $\alpha$ again.
Here we mean by a \emph{graph isomorphism} $\alpha$ from $\Gamma$ to $\Delta$ a permutation on $V$ which preserves edges with labels, i.e., for every pair $\{s,t\}\subset V, s\neq t$,
\[
\{s,t\}\in E_\Gamma \Longleftrightarrow \{\alpha(s), \alpha(t)\}\in E_\Delta
\]
and for each $e\in E_\Gamma$, $m_\Gamma(e)=m_\Delta(\alpha(e))$.
We also denote by $\Gamma\cong\Delta$ if $\Gamma$ and $\Delta$ are isomorphic as CLTTF graphs.

Since each graph isomorphism preserves the connectivity of subgraphs, it maps chunks and separating edges to themselves, respectively. In other words, it preserves the chunk tree.

\begin{theorem}\label{theorem:invariance of the chunk tree under graph isomorphisms}
Let $\alpha\in\mathfrak{S}_V$.
For each $\Gamma$ and $\Delta=\alpha(\Gamma)$, there is an induced isomorphism between rooted trees
\[
\Ch(\alpha):(\Ch_\Gamma,*_\Gamma)\rightarrow (\Ch_\Delta,*_\Delta).
\]
\end{theorem}

\subsubsection{Edge-twists}
Another way to obtain a new \CLTTF graph is an \emph{edge-twist}.

\begin{definition}[Edge-twists]
Let $(\Gamma_1, e=\{s,t\}, \Gamma_2)$ be a decomposition of $\Gamma=(V,E_\Gamma, m_\Gamma)$.
The \emph{edge-twist} of $\Gamma$ with respect to the decomposition $(\Gamma_1, e, \Gamma_2)$ is the graph $\Delta=(V, E_\Delta, m_\Delta)$ with the label $m_\Delta:E_\Delta\to\mathbb{Z}_{\ge2}$ obtained as follows:
\begin{itemize}
\item if $m_\Gamma(e)$ is even, then $(E_\Delta, m_\Delta)\coloneqq (E_\Gamma,m_\Gamma)$,
% : do nothing (i.e. $\Gamma=\varepsilon(\Gamma)$).
\item if $m_\Gamma(e)$ is odd, then 
\begin{align*}
E_\Delta &\coloneqq \{f \mid f\in E_\Gamma, f\not\subset \Gamma_2\text{ or }f\cap e=\varnothing\}
\cup\{\{v, s\}, \{w,t\} \mid v,w\in \Gamma_2, \{v, t\}, \{w,s\}\in E_\Gamma\},\\
m_\Delta(f) &\coloneqq \begin{cases}
m_\Gamma(f) & f\not\subset \Gamma_2\text{ or }f\cap e=\varnothing;\\
m_\Gamma(\{v,t\}) & f=\{v,s\}, v\in \Gamma_2;\\
m_\Gamma(\{v,s\}) & f=\{v,t\}, v\in \Gamma_2.
\end{cases}
\end{align*}
\end{itemize}

For each edge $\varepsilon=(e,C)\in E(\Ch_\Gamma)$, the edge-twist of $\Gamma$ with respect to the decomposition $(\Gamma_1(\varepsilon), e, \Gamma_2(\varepsilon))$ will be denoted by 
\[
\Delta=\bar\varepsilon(\Gamma),\quad\text{ or }\quad\bar\varepsilon:\Gamma\to \Delta.
\]

We say that $\Gamma$ and $\Delta$ are \emph{edge-twist equivalent} if $\Delta$ is obtained from $\Gamma$ by a sequence of edge-twists, denoted by $\Gamma\sim \Delta$.
\end{definition}

Roughly speaking, the edge-twist along $(\Gamma_1, e=\{s,t\}, \Gamma_2)$ with $m(e)$ odd will interchange the connectivities with $s$ and $t$ only for vertices in $\Gamma_2$.
An intuitive example is depicted as follows:
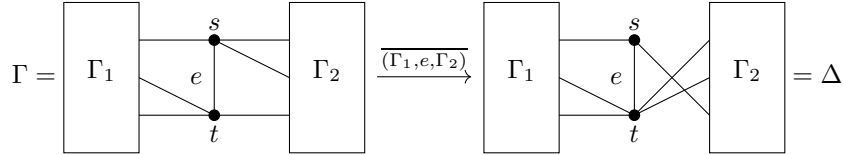
\[
\begin{tikzcd}[row sep=1.5pc, column sep=3pc]
\Gamma=\begin{tikzpicture}[baseline=-.5ex]
\draw[fill] (0,-0.5) circle (2pt) (0, 0.5) circle (2pt);
\draw (-2,-1) rectangle node {$\Gamma_1$} (-1,1);
\draw (2,-1) rectangle node {$\Gamma_2$} (1,1);
\draw (0,-0.5) node[below] {$t$} -- node[midway, left] {$e$} (0, 0.5) node[above] {$s$};
\draw (-1,0) -- (0, -0.5);
\draw (-1,0.5) -- (0, 0.5);
\draw (-1, -0.5) -- (0,-0.5);
\draw (0,-0.5) -- (1, -0.5);
\draw (0,0.5) -- (1, 0);
\draw (0,0.5) -- (1, 0.5);
\end{tikzpicture}
\arrow[r,"\overline{(\Gamma_1, e, \Gamma_2)}"]&
\begin{tikzpicture}[baseline=-.5ex]
\draw[fill] (0,-0.5) circle (2pt) (0, 0.5) circle (2pt);
\draw (-2,-1) rectangle node {$\Gamma_1$} (-1,1);
\draw (2,-1) rectangle node {$\Gamma_2$} (1,1);
\draw (0,-0.5) node[below] {$t$} -- node[midway, left] {$e$} (0, 0.5) node[above] {$s$};
\draw (-1,0) -- (0, -0.5);
\draw (-1,0.5) -- (0, 0.5);
\draw (-1, -0.5) -- (0,-0.5);
\draw (0,0.5) -- (1, -0.5);
\draw (0,-0.5) -- (1, 0);
\draw (0,-0.5) -- (1, 0.5);
\end{tikzpicture}=\Delta
\end{tikzcd}
\]
Therefore when $m(e)$ is odd, then the result $\Delta$ never be the same as $\Gamma$ since $\Gamma$ is triangle-free.

Notice that the decomposition $(\Gamma_1, e, \Gamma_2)$ of $\Gamma$ can be regarded as a decomposition of $\Delta$ as well.
Therefore, chunks and separating edges in $\Gamma_i\subset \Gamma$ are again chunks and separating edges in $\Gamma_i\subset\Delta$.
Namely, we have an induced isomorphism between chunk trees.
\begin{theorem}\label{theorem:invariance of the chunk tree under edge-twists}
Let $(\Gamma_1, e,\Gamma_2)$ be a decomposition of $\Gamma$.
The edge-twist $\Gamma\to \Delta$ with respect to the decomposition $(\Gamma_1, e,\Gamma_2)$ induces an isomorphism $(\Ch_\Gamma, *_\Gamma)\rightarrow (\Ch_\Delta,*_\Delta)$ between rooted trees.

In particular, for each $\varepsilon\in E(\Ch_\Gamma)$, the induced isomorphism will be denoted by $\Ch(\varepsilon)$.
\[
\Ch(\varepsilon):(\Ch_\Gamma,*_\Gamma)\rightarrow (\Ch_\Delta,*_\Delta)
\]
\end{theorem}
\begin{proof}
This follows obviously from the above discussion and we omit the proof.
\end{proof}

The direct consequence of this theorem is as follows:
\begin{corollary}
Let $\llbracket\Gamma\rrbracket$ be an edge-twist equivalent class of an edge-separated \CLTTF graph $\Gamma$.
Then the pair $(\Ch_{\llbracket\Gamma\rrbracket},*_{\llbracket\Gamma\rrbracket})$ of the chunk tree $\Ch_{\llbracket\Gamma\rrbracket}$ and the central vertex $*_{\llbracket\Gamma\rrbracket}$ for the class $\llbracket\Gamma\rrbracket$ is well-defined.

In other words, there is a canonical identification $(\Ch_\Gamma,*_\Gamma)\cong (\Ch_\Delta,*_\Delta)$ if $\Gamma\sim \Delta$.
\end{corollary}

Furthermore, each edge-twist $\bar\varepsilon:\Gamma\to \Delta$ induces a label-preserving bijection $\bar\varepsilon:(E_\Gamma, m_\Gamma)\to (E_\Delta, m_\Delta)$. 
That is, for each edge $e\in E_\Gamma$, we have an edge $\bar\varepsilon(e)$ in $\Delta$ so that $m_\Gamma(e)=m_\Delta(\bar\varepsilon(e))$.
In particular, for each separating edge $e\subset E_\Gamma$, then its label remains the same in its edge-twist equivalence class.

\begin{definition}
For each edge $\varepsilon=(e,C)\in E(\Ch_{\llbracket\Gamma\rrbracket})$, the label $m_{\llbracket\Gamma\rrbracket}(\varepsilon)$ is defined as $m_\Gamma(e)$.
\end{definition}

Due to Remark~\ref{remark:not every decomposition comes from edges of chunk tree}, edge-twists with respect to arbitrary decompositions of $\Gamma$ form a strictly larger class than those with respect to decompositions coming from edges in the chunk tree $\Ch_\Gamma$.
However, one can easily check that every edge-twist is actually a composition of the latter edge-twists.
See Figure~\ref{figure:composition of edge-twists} for example.
Note that in the Figure~\ref{figure:composition of edge-twists}, the chunk trees are identified in the obvious way.

%\begin{remark}
%The edge twist in Crisp's paper can be considered finite edge twist operations using our definition. 
%\end{remark}

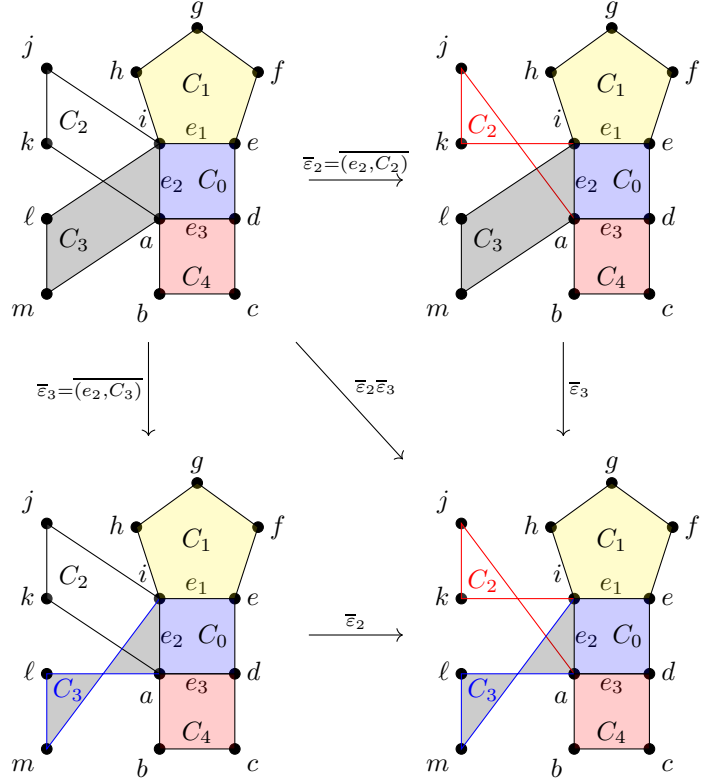
\begin{figure}[ht]
\[
\begin{tikzcd}[row sep=3pc, column sep=3pc]
\begin{tikzpicture}[baseline=-.5ex]
\draw[fill] (0, -0.5) circle (2pt) node[below=2ex, left=-.5ex] {$a$};
\draw[fill] (1, -0.5) circle (2pt) node[right] {$d$};
\draw[fill] (0, 0.5) circle (2pt) node[above=2ex, left] {$i$};
\draw[fill] (1, 0.5) circle (2pt) node[right] {$e$};
\draw[fill] (0, -1.5) circle (2pt) node[below left] {$b$};
\draw[fill] (1, -1.5) circle (2pt) node[below right] {$c$};
\draw[fill](1, 0.5) ++(72:1) circle (2pt) node[right] {$f$} ++(144:1) circle(2pt) node[above] {$g$} ++(216:1) circle(2pt) node[left] {$h$};
\draw[fill] (-1.5,1.5) circle (2pt) node[above left] {$j$};
\draw[fill] (-1.5,0.5) circle (2pt) node[left] {$k$};
\draw[fill] (-1.5,-0.5) circle (2pt) node[left] {$\ell$};
\draw[fill] (-1.5,-1.5) circle (2pt) node[below left] {$m$};
\draw (0.5,0.5) node[above=-.5ex] {$e_1$} (0,0) node[right=-1ex] {$e_2$} (0.5,-0.5) node[below=-.5ex] {$e_3$};
\draw[color=black, fill=red, fill opacity=0.2](0, -1.5) rectangle node[opacity=1, below] {$C_4$} (1,-0.5);
\draw[color=black, fill=yellow, fill opacity=0.2](1, 0.5) -- ++(72:1) -- ++(144:1) -- ++(216:1) -- ++(288:1);
\draw (0.5, 1) node[opacity=1, above] {$C_1$};
\draw[color=black, fill=white, fill opacity=0.2](0, 0.5) -- ++(-1.5, 1) -- ++(-90:1) node[opacity=1, above right] {$C_2$} -- ++(1.5, -1);
\draw[color=black, fill=black, fill opacity=0.2](0, -0.5) -- ++(-1.5, -1) -- ++(90:1) node[opacity=1, below right] {$C_3$} -- ++(1.5, 1);
\draw[color=black, fill=blue, fill opacity=0.2](0, -0.5) rectangle node[opacity=1, right=-1ex] {$C_0$} (1,0.5);
\end{tikzpicture}
\arrow[r,"{\bar\varepsilon_2=\overline{(e_2,C_2)}}"]\arrow[d,"{\bar\varepsilon_3=\overline{(e_2,C_3)}}"']\arrow[rd,"\bar\varepsilon_2\bar\varepsilon_3"] &
\begin{tikzpicture}[baseline=-.5ex]
\draw[fill] (0, -0.5) circle (2pt) node[below=2ex, left=-.5ex] {$a$};
\draw[fill] (1, -0.5) circle (2pt) node[right] {$d$};
\draw[fill] (0, 0.5) circle (2pt) node[above=2ex, left] {$i$};
\draw[fill] (1, 0.5) circle (2pt) node[right] {$e$};
\draw[fill] (0, -1.5) circle (2pt) node[below left] {$b$};
\draw[fill] (1, -1.5) circle (2pt) node[below right] {$c$};
\draw[fill](1, 0.5) ++(72:1) circle (2pt) node[right] {$f$} ++(144:1) circle(2pt) node[above] {$g$} ++(216:1) circle(2pt) node[left] {$h$};
\draw[fill] (-1.5,1.5) circle (2pt) node[above left] {$j$};
\draw[fill] (-1.5,0.5) circle (2pt) node[left] {$k$};
\draw[fill] (-1.5,-0.5) circle (2pt) node[left] {$\ell$};
\draw[fill] (-1.5,-1.5) circle (2pt) node[below left] {$m$};
\draw (0.5,0.5) node[above=-.5ex] {$e_1$} (0,0) node[right=-1ex] {$e_2$} (0.5,-0.5) node[below=-.5ex] {$e_3$};
\draw[color=black, fill=red, fill opacity=0.2](0, -1.5) rectangle node[opacity=1, below] {$C_4$} (1,-0.5);
\draw[color=black, fill=yellow, fill opacity=0.2](1, 0.5) -- ++(72:1) -- ++(144:1) -- ++(216:1) -- ++(288:1);
\draw (0.5, 1) node[opacity=1, above] {$C_1$};
\draw[color=red, fill=white, fill opacity=0.2](0, -0.5) -- (-1.5, 1.5) -- ++(0,-1) node[opacity=1, above=1.5ex, right=-.5ex] {$C_2$} -- (0, 0.5);
\draw[color=black, fill=black, fill opacity=0.2](0, -0.5) -- ++(-1.5, -1) -- ++(90:1) node[opacity=1, below right] {$C_3$} -- ++(1.5, 1);
\draw[color=black, fill=blue, fill opacity=0.2](0, -0.5) rectangle node[opacity=1, right=-1ex] {$C_0$} (1,0.5);
\end{tikzpicture}
\arrow[d,"\bar\varepsilon_3"]\\
\begin{tikzpicture}[baseline=-.5ex]
\draw[fill] (0, -0.5) circle (2pt) node[below=2ex, left=-.5ex] {$a$};
\draw[fill] (1, -0.5) circle (2pt) node[right] {$d$};
\draw[fill] (0, 0.5) circle (2pt) node[above=2ex, left] {$i$};
\draw[fill] (1, 0.5) circle (2pt) node[right] {$e$};
\draw[fill] (0, -1.5) circle (2pt) node[below left] {$b$};
\draw[fill] (1, -1.5) circle (2pt) node[below right] {$c$};
\draw[fill](1, 0.5) ++(72:1) circle (2pt) node[right] {$f$} ++(144:1) circle(2pt) node[above] {$g$} ++(216:1) circle(2pt) node[left] {$h$};
\draw[fill] (-1.5,1.5) circle (2pt) node[above left] {$j$};
\draw[fill] (-1.5,0.5) circle (2pt) node[left] {$k$};
\draw[fill] (-1.5,-0.5) circle (2pt) node[left] {$\ell$};
\draw[fill] (-1.5,-1.5) circle (2pt) node[below left] {$m$};
\draw (0.5,0.5) node[above=-.5ex] {$e_1$} (0,0) node[right=-1ex] {$e_2$} (0.5,-0.5) node[below=-.5ex] {$e_3$};
\draw[color=black, fill=red, fill opacity=0.2](0, -1.5) rectangle node[opacity=1, below] {$C_4$} (1,-0.5);
\draw[color=black, fill=yellow, fill opacity=0.2](1, 0.5) -- ++(72:1) -- ++(144:1) -- ++(216:1) -- ++(288:1);
\draw (0.5, 1) node[opacity=1, above] {$C_1$};
\draw[color=black, fill=white, fill opacity=0.2](0, 0.5) -- ++(-1.5, 1) -- ++(-90:1) node[opacity=1, above right] {$C_2$} -- ++(1.5, -1);
\draw[color=blue, fill=black, fill opacity=0.2](0, 0.5) -- (-1.5, -1.5) -- ++(0,1) node[opacity=1, below=1.5ex, right=-0.5ex] {$C_3$} -- (0, -0.5);
\draw[color=black, fill=blue, fill opacity=0.2](0, -0.5) rectangle node[opacity=1, right=-1ex] {$C_0$} (1,0.5);
\end{tikzpicture}
\arrow[r,"\bar\varepsilon_2"]&
\begin{tikzpicture}[baseline=-.5ex]
\draw[fill] (0, -0.5) circle (2pt) node[below=2ex, left=-.5ex] {$a$};
\draw[fill] (1, -0.5) circle (2pt) node[right] {$d$};
\draw[fill] (0, 0.5) circle (2pt) node[above=2ex, left] {$i$};
\draw[fill] (1, 0.5) circle (2pt) node[right] {$e$};
\draw[fill] (0, -1.5) circle (2pt) node[below left] {$b$};
\draw[fill] (1, -1.5) circle (2pt) node[below right] {$c$};
\draw[fill](1, 0.5) ++(72:1) circle (2pt) node[right] {$f$} ++(144:1) circle(2pt) node[above] {$g$} ++(216:1) circle(2pt) node[left] {$h$};
\draw[fill] (-1.5,1.5) circle (2pt) node[above left] {$j$};
\draw[fill] (-1.5,0.5) circle (2pt) node[left] {$k$};
\draw[fill] (-1.5,-0.5) circle (2pt) node[left] {$\ell$};
\draw[fill] (-1.5,-1.5) circle (2pt) node[below left] {$m$};
\draw (0.5,0.5) node[above=-.5ex] {$e_1$} (0,0) node[right=-1ex] {$e_2$} (0.5,-0.5) node[below=-.5ex] {$e_3$};
\draw[color=black, fill=red, fill opacity=0.2](0, -1.5) rectangle node[opacity=1, below] {$C_4$} (1,-0.5);
\draw[color=black, fill=yellow, fill opacity=0.2](1, 0.5) -- ++(72:1) -- ++(144:1) -- ++(216:1) -- ++(288:1);
\draw (0.5, 1) node[opacity=1, above] {$C_1$};
\draw[color=red, fill=white, fill opacity=0.2](0, -0.5) -- (-1.5, 1.5) -- ++(0,-1) node[opacity=1, above=1.5ex, right=-.5ex] {$C_2$} -- (0, 0.5);
\draw[color=blue, fill=black, fill opacity=0.2](0, 0.5) -- (-1.5, -1.5) -- ++(0,1) node[opacity=1, below=1.5ex, right=-0.5ex] {$C_3$} -- (0, -0.5);
\draw[color=black, fill=blue, fill opacity=0.2](0, -0.5) rectangle node[opacity=1, right=-1ex] {$C_0$} (1,0.5);
\end{tikzpicture}
\end{tikzcd}
\]
\caption{A composition of edge-twists}
\label{figure:composition of edge-twists}
\end{figure}

More precisely, let $e\subset \Gamma$ be an odd-labeled separating edge and let $C_1,\dots, C_n\subset \Gamma$ be all chunks containing $e$.
In the chunk tree $\Ch_\Gamma$, the vertex $e$ is of $n$-valent and edges $\varepsilon_i=(e,C_i)$ are adjacent to $e$.
Suppose that a decomposition $(\Gamma_1, e, \Gamma_2)$ is given such that for some $\ell<n$. That is,
\begin{align*}
C_1,\dots, C_\ell \subset \Gamma_1\quad\text{ and }\quad
C_{\ell+1},\dots, C_n \subset \Gamma_2.
\end{align*}
Then the edge-twist with respect to $(\Gamma_1,e,\Gamma_2)$ is nothing but the composition
\[
\bar\varepsilon_{\ell+1}\bar\varepsilon_\ell\cdots\bar\varepsilon_{n}:\Gamma \to \Delta.\footnotemark
\]
In this sense, it is enough to consider edge-twists along edges in the chunk tree.
Then one can easily check that for every $i,j\in\{1,\dots, n\}$,
\begin{equation}\label{equation:commutativity of edge-twists}
\bar\varepsilon_i(\bar\varepsilon_j(\Gamma))=\bar\varepsilon_j(\bar\varepsilon_i(\Gamma))\quad\text{ and }\quad
\bar\varepsilon_i(\bar\varepsilon_i(\Gamma))=\Gamma.
\end{equation}
\footnotetext{Here we omit the notation $\circ$ for compositions.}

On the other hand, edge-twists along all edges in $\Ch_\Gamma$ are sometimes too many.
If we take edge-twists on $\Gamma$ with respect to all $\varepsilon_i$'s, then the result graph
\[
\Delta=(\bar\varepsilon_1\bar\varepsilon_2\cdots\bar\varepsilon_n)(\Gamma)
\]
is obtained from $\Gamma$ by interchanging the roles of two vertices of $e$.
That is, there is a graph isomorphism $\alpha\in\mathfrak{S}_V$ with $\alpha(\Gamma)=\Delta$ defined by
\[
\alpha(v) = \begin{cases}
v & v\neq s, t;\\
t & v=s;\\
s & v=t.
\end{cases}
\]
Therefore, up to graph isomorphisms, one may reduce one of edge-twists $\bar\varepsilon_1,\bar\varepsilon_2,\dots, \bar\varepsilon_n$.

We will provide a systematical way of doing this as follows:
recall that $\Ch_\Gamma$ has the central vertex $*_\Gamma$.
Then we have another orientation on edges of $\Ch_\Gamma$ given by the \emph{away-from-center} convention.
We say that an edge $\varepsilon=(e,C)\in E(\Ch_\Gamma)$ is \emph{outward} or \emph{inward} if $C$ is farther or closer than $e$ from $*_\Gamma$, respectively.
We denote the subset of outward and inward edges in $\Ch_\Gamma$ by $E^\out(\Ch_\Gamma)$ and $E^{\mathsf{in}}(\Ch_\Gamma)$, respectively.

For example, the chunk tree $\Ch_\Gamma$ in Example~\ref{example:chunk graph} has three inward edges $\varepsilon_1, \varepsilon_2, \varepsilon_3$ and four outward edges $\varepsilon_4, \varepsilon_5, \varepsilon_6, \varepsilon_7$.
\begin{align*}
E^{\mathsf{in}}(\Ch_\Gamma)&=\{\varepsilon_1,\varepsilon_2,\varepsilon_3\},&
E^{\out}(\Ch_\Gamma)&=\{\varepsilon_4,\varepsilon_5,\varepsilon_6,\varepsilon_7\}.
\end{align*}

One simple but important observation is as follows: in the chunk tree $\Ch_\Gamma$ and a separating edge $e$, there are at most one inward edge adjacent to $e$.
Indeed, the case without inward edge is when $*_\Gamma=e$ is the central vertex of $\Ch_\Gamma$.

In summary, it suffices to consider edge-twists with respect to outward edges in $\Ch_\Gamma$ since every edge-twist is a combination of those and graph isomorphisms.

\begin{remark}\label{remark:iso}
Even though we consider outward edges only in $\Ch_\Gamma$, there are still possibilities that compositions of edge-twists involving different separating edges become graph isomorphisms as well.
We will consider these cases later in Section~\ref{section:iso}.
\end{remark}

Let $\varepsilon\in E(\Ch_{\Gamma})$ and $(\Gamma_1, e, \Gamma_2)$ be the decomposition corresponding to $\varepsilon$.
For a sake of convenience, we define for each $i=1,2$,
\begin{align*}
\Gamma_i(\varepsilon)\coloneqq \Gamma_i\subset\Gamma\quad\text{ and }\quad
V_i(\varepsilon)\coloneqq V(\Gamma_i) \subset V.
\end{align*}

\begin{lemma}\label{lemma:commutative and involutive}
Edge-twists along outward edges in the chunk tree are commutative and involutive.
\end{lemma}
\begin{proof}
By definition of edge-twists, it is obvious that for $\varepsilon\in E^{\out}(\Ch_{\llbracket\Gamma\rrbracket})$ and any $\Delta\sim \Gamma$,
\begin{equation}\label{equation:involutive}
\bar\varepsilon^2(\Delta) = (\bar\varepsilon\bar\varepsilon)(\Delta) = \Delta
\end{equation}
and so edge-twists along outward edges are involutive.

Let $\varepsilon=(e,C), \varepsilon'=(e',C') \in E^{\out}(\Ch_{\llbracket\Gamma\rrbracket})$.
If $e=e'$ or $e\cap e'=\emptyset$, then corresponding edge-twists will commute.
Otherwise, there are three cases 
(i) $C\not\subset \Gamma_2(\varepsilon')\text{ and }C'\not\subset \Gamma_2(\varepsilon)$,
(ii) $C\not\subset \Gamma_2(\varepsilon')\text{ and }C'\subset \Gamma_2(\varepsilon)$, and
(iii) $C\subset \Gamma_2(\varepsilon')\text{ and }C'\not\subset \Gamma_2(\varepsilon)$,
according to whether
$C\subset\Gamma_2(\varepsilon')$ or $C'\subset\Gamma_2(\varepsilon)$ for decomposition $(\Gamma_1(\varepsilon), e, \Gamma_2(\varepsilon))$ and $(\Gamma_1(\varepsilon'), e', \Gamma_2(\varepsilon'))$ induced from $\varepsilon$ and $\epsilon'$, respectively.

Notice that when $C\subset\Gamma_2(\varepsilon')$ and $C'\subset\Gamma_2(\varepsilon)$, then $C=C'$ and at least one of $\varepsilon$ and $\varepsilon'$ should be \emph{inward}.

For each case, edge-twists are commutative as shown in Figure~\ref{figure:commutative edge-twists} and we omit the detail.
\end{proof}

\begin{remark}\label{remark:scopes are disjoint or nested}
We remark that for any $\varepsilon, \varepsilon'\in E^{\out}(\Ch_{\llbracket\Gamma\rrbracket})$, two sets $V\setminus V_1(\varepsilon)$ and $V\setminus V_1(\varepsilon')$ are either disjoint or nested.
\end{remark}

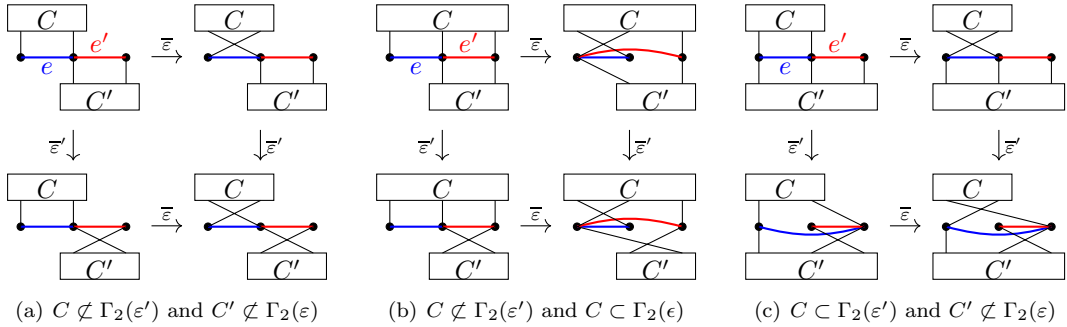
\begin{figure}[ht]
\subfigure[\label{figure:commutative edge-twists 1}$C\not\subset \Gamma_2(\varepsilon')$ and $C'\not\subset \Gamma_2(\varepsilon)$]{
$
\begin{tikzcd}[ampersand replacement=\&, column sep=1pc, row sep=1pc]
\begin{tikzpicture}[baseline=-.5ex,scale=0.7]
\draw[fill] (-1,0) circle (2pt) (0,0) circle (2pt) (1,0) circle (2pt);
\draw[blue, thick] (-1,0) -- node[midway, below=-.5ex] {$e$} (0,0);
\draw[red, thick] (0,0) -- node[midway, above=-.5ex] {$e'$} (1,0);
\draw (-1,0) -- +(0,0.5) (0,0) -- +(0,0.5) (0,0) -- +(0,-0.5) (1,0) -- +(0,-0.5);
\draw (-1.25,0.5) rectangle node[below=-1.5ex] {$C$} (0.25, 1);
\draw (1.25,-0.5) rectangle node[below=-1.5ex] {$C'$} (-0.25, -1);
\end{tikzpicture}
\arrow[r, "\bar\varepsilon"]
\arrow[d, "\bar\varepsilon'"']
\&
\begin{tikzpicture}[baseline=-.5ex,scale=0.7]
\draw[fill] (-1,0) circle (2pt) (0,0) circle (2pt) (1,0) circle (2pt);
\draw[blue, thick] (-1,0) -- (0,0);
\draw[red, thick] (0,0) -- (1,0);
\draw (-1,0) -- +(1,0.5) (0,0) -- +(-1,0.5) (0,0) -- +(0,-0.5) (1,0) -- +(0,-0.5);
\draw (-1.25,0.5) rectangle node[below=-1.5ex] {$C$} (0.25, 1);
\draw (1.25,-0.5) rectangle node[below=-1.5ex] {$C'$} (-0.25, -1);
\end{tikzpicture}
\arrow[d, "\bar\varepsilon'"]
\\
\begin{tikzpicture}[baseline=-.5ex,scale=0.7]
\draw[fill] (-1,0) circle (2pt) (0,0) circle (2pt) (1,0) circle (2pt);
\draw[blue, thick] (-1,0) -- (0,0);
\draw[red, thick] (0,0) -- (1,0);
\draw (-1,0) -- +(0,0.5) (0,0) -- +(0,0.5) (0,0) -- +(1,-0.5) (1,0) -- +(-1,-0.5);
\draw (-1.25,0.5) rectangle node[below=-1.5ex] {$C$} (0.25, 1);
\draw (1.25,-0.5) rectangle node[below=-1.5ex] {$C'$} (-0.25, -1);
\end{tikzpicture}
\arrow[r, "\bar\varepsilon"]
\&
\begin{tikzpicture}[baseline=-.5ex,scale=0.7]
\draw[fill] (-1,0) circle (2pt) (0,0) circle (2pt) (1,0) circle (2pt);
\draw[blue, thick] (-1,0) -- (0,0);
\draw[red, thick] (0,0) -- (1,0);
\draw (-1,0) -- +(1,0.5) (0,0) -- +(-1,0.5) (0,0) -- +(1,-0.5) (1,0) -- +(-1,-0.5);
\draw (-1.25,0.5) rectangle node[below=-1.5ex] {$C$} (0.25, 1);
\draw (1.25,-0.5) rectangle node[below=-1.5ex] {$C'$} (-0.25, -1);
\end{tikzpicture}
\end{tikzcd}
$
}
\subfigure[\label{figure:commutative edge-twists 2}$C\not\subset \Gamma_2(\varepsilon')$ and $C\subset\Gamma_2(\epsilon)$]{
$
\begin{tikzcd}[ampersand replacement=\&, column sep=1pc, row sep=1pc]
\begin{tikzpicture}[baseline=-.5ex,scale=0.7]
\draw[fill] (-1,0) circle (2pt) (0,0) circle (2pt) (1,0) circle (2pt);
\draw[blue, thick] (-1,0) -- node[midway, below=-.5ex] {$e$} (0,0);
\draw[red, thick] (0,0) -- node[midway, above=-.5ex] {$e'$} (1,0);
\draw (-1,0) -- +(0,0.5) (0,0) -- +(0,0.5) (0,0) -- +(0,-0.5) (1,0) -- +(0,0.5) (1,0) -- +(0,-0.5);
\draw (-1.25,0.5) rectangle node[below=-1.5ex] {$C$} (1.25, 1);
\draw (1.25,-0.5) rectangle node[below=-1.5ex] {$C'$} (-0.25, -1);
\end{tikzpicture}
\arrow[r, "\bar\varepsilon"]
\arrow[d, "\bar\varepsilon'"']
\&
\begin{tikzpicture}[baseline=-.5ex,scale=0.7]
\draw[fill] (-1,0) circle (2pt) (0,0) circle (2pt) (1,0) circle (2pt);
\draw[blue, thick] (-1,0) -- (0,0);
\draw[red, thick] (-1,0) to[out=15,in=165] (1,0);
\draw (-1,0) -- +(1,0.5) (0,0) -- +(-1,0.5) (-1,0) -- +(1,-0.5) (1,0) -- +(0,0.5) (1,0) -- +(0,-0.5);
\draw (-1.25,0.5) rectangle node[below=-1.5ex] {$C$} (1.25, 1);
\draw (1.25,-0.5) rectangle node[below=-1.5ex] {$C'$} (-0.25, -1);
\end{tikzpicture}
\arrow[d, "\bar\varepsilon'"]
\\
\begin{tikzpicture}[baseline=-.5ex,scale=0.7]
\draw[fill] (-1,0) circle (2pt) (0,0) circle (2pt) (1,0) circle (2pt);
\draw[blue, thick] (-1,0) -- (0,0);
\draw[red, thick] (0,0) -- (1,0);
\draw (-1,0) -- +(0,0.5) (0,0) -- +(0,0.5) (0,0) -- +(1,-0.5) (1,0) -- +(0,0.5) (1,0) -- +(-1,-0.5);
\draw (-1.25,0.5) rectangle node[below=-1.5ex] {$C$} (1.25, 1);
\draw (1.25,-0.5) rectangle node[below=-1.5ex] {$C'$} (-0.25, -1);
\end{tikzpicture}
\arrow[r, "\bar\varepsilon"]
\&
\begin{tikzpicture}[baseline=-.5ex,scale=0.7]
\draw[fill] (-1,0) circle (2pt) (0,0) circle (2pt) (1,0) circle (2pt);
\draw[blue, thick] (-1,0) -- (0,0);
\draw[red, thick] (-1,0) to[out=15,in=165] (1,0);
\draw (-1,0) -- +(1,0.5) (0,0) -- +(-1,0.5) (-1,0) -- +(2,-0.5) (1,0) -- +(0,0.5) (1,0) -- +(-1,-0.5);
\draw (-1.25,0.5) rectangle node[below=-1.5ex] {$C$} (1.25, 1);
\draw (1.25,-0.5) rectangle node[below=-1.5ex] {$C'$} (-0.25, -1);
\end{tikzpicture}
\end{tikzcd}
$
}
\subfigure[\label{figure:commutative edge-twists 3}$C\subset \Gamma_2(\varepsilon')$ and $C'\not\subset \Gamma_2(\varepsilon)$]{
$
\begin{tikzcd}[ampersand replacement=\&, column sep=1pc, row sep=1pc]
\begin{tikzpicture}[baseline=-.5ex,scale=0.7]
\draw[fill] (-1,0) circle (2pt) (0,0) circle (2pt) (1,0) circle (2pt);
\draw[blue, thick] (-1,0) -- node[midway, below=-.5ex] {$e$} (0,0);
\draw[red, thick] (0,0) -- node[midway, above=-.5ex] {$e'$} (1,0);
\draw (-1,0) -- +(0,0.5) (-1,0) -- +(0,-0.5) (0,0) -- +(0,0.5) (0,0) -- +(0,-0.5) (1,0) -- +(0,-0.5);
\draw (-1.25,0.5) rectangle node[below=-1.5ex] {$C$} (0.25, 1);
\draw (1.25,-0.5) rectangle node[below=-1.5ex] {$C'$} (-1.25, -1);
\end{tikzpicture}
\arrow[r, "\bar\varepsilon"]
\arrow[d, "\bar\varepsilon'"']
\&
\begin{tikzpicture}[baseline=-.5ex,scale=0.7]
\draw[fill] (-1,0) circle (2pt) (0,0) circle (2pt) (1,0) circle (2pt);
\draw[blue, thick] (-1,0) -- (0,0);
\draw[red, thick] (0,0) -- (1,0);
\draw (-1,0) -- +(1,0.5) (-1,0) -- +(0,-0.5) (0,0) -- +(-1,0.5) (0,0) -- +(0,-0.5) (1,0) -- +(0,-0.5);
\draw (-1.25,0.5) rectangle node[below=-1.5ex] {$C$} (0.25, 1);
\draw (1.25,-0.5) rectangle node[below=-1.5ex] {$C'$} (-1.25, -1);
\end{tikzpicture}
\arrow[d, "\bar\varepsilon'"]
\\
\begin{tikzpicture}[baseline=-.5ex,scale=0.7]
\draw[fill] (-1,0) circle (2pt) (0,0) circle (2pt) (1,0) circle (2pt);
\draw[blue, thick] (-1,0) to[out=-15,in=-165] (1,0);
\draw[red, thick] (0,0) -- (1,0);
\draw (-1,0) -- +(0,0.5) (-1,0) -- +(0,-0.5) (1,0) -- +(-1,0.5) (0,0) -- +(1,-0.5) (1,0) -- +(-1,-0.5);
\draw (-1.25,0.5) rectangle node[below=-1.5ex] {$C$} (0.25, 1);
\draw (1.25,-0.5) rectangle node[below=-1.5ex] {$C'$} (-1.25, -1);
\end{tikzpicture}
\arrow[r, "\bar\varepsilon"]
\&
\begin{tikzpicture}[baseline=-.5ex,scale=0.7]
\draw[fill] (-1,0) circle (2pt) (0,0) circle (2pt) (1,0) circle (2pt);
\draw[blue, thick] (-1,0) to[out=-15,in=-165] (1,0);
\draw[red, thick] (0,0) -- (1,0);
\draw (-1,0) -- +(1,0.5) (-1,0) -- +(0,-0.5) (1,0) -- +(-2,0.5) (0,0) -- +(1,-0.5) (1,0) -- +(-1,-0.5);
\draw (-1.25,0.5) rectangle node[below=-1.5ex] {$C$} (0.25, 1);
\draw (1.25,-0.5) rectangle node[below=-1.5ex] {$C'$} (-1.25, -1);
\end{tikzpicture}
\end{tikzcd}
$
}
\caption{Commutativity of edge-twists}
\label{figure:commutative edge-twists}
\end{figure}

Let us define two sets $E^\out_{\even}(\Ch_{\llbracket\Gamma\rrbracket})$ and $E^\out_{\odd}(\Ch_{\llbracket\Gamma\rrbracket})$ as the subsets of $E^\out(\Ch_{\llbracket\Gamma\rrbracket})$ consisting of $\varepsilon=(e, C)$ with $m(e)$ even and odd, respectively.
For each function $\bar\eta:E^\out_\odd(\Ch_{\llbracket\Gamma\rrbracket})\to\mathbb{Z}_2$, we define the composition $\bar{\mathcal{E}}(\bar\eta)$ of edge-twists as
\begin{align*}
\bar{\mathcal{E}}(\bar\eta) &= \prod_{\varepsilon\in E^{\out}_\odd(\Ch_{\llbracket\Gamma\rrbracket})} \bar\varepsilon^{\bar\eta(\varepsilon)}.
\end{align*}

\begin{proposition}\label{proposition:normal form for graphs}
Let $\Delta$ be a \CLTTF graph edge-twist equivalent to $\Gamma$.
Then there exists a unique function $\bar\eta:E^{\out}_\odd(\Ch_{\llbracket\Gamma\rrbracket})\to \mathbb{Z}_2$ such that 
$\bar{\mathcal{E}} = \bar{\mathcal{E}}(\bar\eta):\Gamma\to\Delta$.
\end{proposition}
\begin{proof}
By definition of edge-twist equivalence, the existence is obvious.
Let $\bar\eta$ and $\bar\eta'$ be two such functions. That is,
\[
\bar{\mathcal{E}}=\bar{\mathcal{E}}(\bar\eta):\Gamma\to \Delta\quad\text{ and }\quad
\bar{\mathcal{E}}'=\bar{\mathcal{E}}(\bar\eta'):\Gamma\to \Delta.
\]

We define $(\bar\eta+\bar\eta'):E^{\out}_\odd(\Ch_{\llbracket\Gamma\rrbracket})\to\mathbb{Z}_2$ as the point-wise addition over $\mathbb{Z}_2$
\[
(\bar\eta+\bar\eta')(\varepsilon) \coloneqq 
\bar\eta(\varepsilon)+\bar\eta'(\varepsilon)=
\begin{cases}
0 & \bar\eta(\varepsilon) = \bar\eta'(\varepsilon);\\
1 & \bar\eta(\varepsilon) \neq \bar\eta'(\varepsilon),
\end{cases}
\]
and then by Lemma~\ref{lemma:commutative and involutive}, we have
\[
\bar{\mathcal{E}}'^{-1}\bar{\mathcal{E}} = \bar{\mathcal{E}}(\bar\eta+\bar\eta'):\Gamma\to\Gamma.
\]

Now let $\varepsilon=(e=\{s,t\},C)\in E^{\out}_\odd(\Ch_{\llbracket\Gamma\rrbracket})$ be one of the farthest edge from the center $*_{\llbracket\Gamma\rrbracket}$ with $(\bar\eta+\bar\eta')(\varepsilon)=1$.
Pick an edge $\{v,s\}\in E_C$ with $v\neq s,t$.
Then by definition of edge-twist, $\{v,t\}\in E_C$.
This is a contradiction since the set $\{v,s,t\}$ of vertices forms a triangle but $\Gamma$ is triangle-free.
Hence there are no such $\varepsilon$ and so $(\bar\eta+\bar\eta')\equiv 0$, or equivalently, $\bar\eta$ coincides with $\bar\eta'$.
\end{proof}

We consider the following rigidities of \CLTTF graphs.

\begin{definition}[Rigidity of \CLTTF graphs]
A \CLTTF graph $\Gamma$ is said to be 
\begin{enumerate}
\item \emph{rigid} if %the edge-twist equivalent class $[\Gamma]$ consists of graphs isomorphic to $\Gamma$. That is,
\[
\Gamma\sim \Delta \Longrightarrow \Gamma\cong \Delta,
\]
\item \emph{discretely rigid} if
\[
\Gamma\sim \Delta\text{ and }\Gamma\cong\Delta\Longrightarrow \Gamma=\Delta.
\]
\end{enumerate}
\end{definition}

\begin{remark}
One can see these rigidity as follows: rigid if and only if $\llbracket\Gamma\rrbracket$ up to graph isomorphism is a singleton, and discretely rigid if $\llbracket\Gamma\rrbracket$ up to graph isomorphism is the same as $\llbracket\Gamma\rrbracket$ itself.
\end{remark}

There are examples of rigid but not discretely rigid \CLTTF graphs, and \textit{vice versa}.

\begin{example}
The \CLTTF graph $\Gamma$ below is rigid but not discretely rigid, while $\Delta$ is discretely rigid but not rigid.
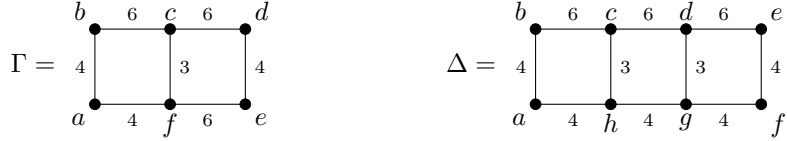
\begin{align*}
\Gamma&=\begin{tikzpicture}[baseline=-.5ex]
\foreach \i in {-1,0,1} {
\draw[fill] (\i,-0.5) circle (2pt) (\i, 0.5) circle (2pt);
}
\draw (-1,-0.5) node[below left] {$a$} -- node[midway, left] {$\scriptstyle4$} (-1, 0.5) node[above left] {$b$};
\draw (0,-0.5) node[below] {$f$} -- node[midway, right] {$\scriptstyle3$} (0, 0.5) node[above] {$c$};
\draw (1,-0.5) node[below right] {$e$} -- node[midway, right] {$\scriptstyle4$} (1, 0.5) node[above right] {$d$};
\draw (-1,-0.5) -- node[midway, below] {$\scriptstyle4$} (0, -0.5);
\draw (-1,0.5) -- node[midway, above] {$\scriptstyle6$} (0, 0.5);
\draw (0,-0.5) -- node[midway, below] {$\scriptstyle6$} (1, -0.5);
\draw (0,0.5) -- node[midway, above] {$\scriptstyle6$} (1, 0.5);
\end{tikzpicture}&
\Delta&=\begin{tikzpicture}[baseline=-.5ex]
\foreach \i in {-1.5,-0.5,0.5,1.5} {
\draw[fill] (\i,-0.5) circle (2pt) (\i, 0.5) circle (2pt);
}
\draw (-1.5,-0.5) node[below left] {$a$} -- node[midway, left] {$\scriptstyle4$} (-1.5, 0.5) node[above left] {$b$};
\draw (-0.5,-0.5) node[below] {$h$} -- node[midway, right] {$\scriptstyle3$} (-0.5, 0.5) node[above] {$c$};
\draw (0.5,-0.5) node[below] {$g$} -- node[midway, right] {$\scriptstyle3$} (0.5, 0.5) node[above] {$d$};
\draw (1.5,-0.5) node[below right] {$f$} -- node[midway, right] {$\scriptstyle4$} (1.5, 0.5) node[above right] {$e$};
\draw (-1.5,-0.5) -- node[midway, below] {$\scriptstyle4$} (-0.5, -0.5);
\draw (-1.5,0.5) -- node[midway, above] {$\scriptstyle6$} (-0.5, 0.5);
\draw (-0.5,-0.5) -- node[midway, below] {$\scriptstyle4$} (0.5, -0.5);
\draw (-0.5,0.5) -- node[midway, above] {$\scriptstyle6$} (0.5, 0.5);
\draw (0.5,-0.5) -- node[midway, below] {$\scriptstyle4$} (1.5, -0.5);
\draw (0.5,0.5) -- node[midway, above] {$\scriptstyle6$} (1.5, 0.5);
\end{tikzpicture}
\end{align*}
\end{example}

\begin{lemma}
Let $\Gamma=(V,E,m)$ be a rigid and discretely rigid \CLTTF graph.
Then for each separating edge $e$, the label $m(e)$ is even.
\end{lemma}
\begin{proof}
Since $\Gamma$ is rigid and discretely rigid, $\Gamma\sim\Delta$ implies $\Gamma=\Delta$.
However, if there is an odd-labelled edge $e\in E$, then an edge-twist involving $e$ yields an edge-twist equivalent graph $\Delta$ different from $\Gamma$ as mentioned earlier. This contradiction completes the proof.
\end{proof}

\subsection{pull-backs and push-forwards}
(a) Let $\alpha:\Gamma\to\Delta$ be a graph isomorphism and $\varepsilon=(e,C)\in E^\out(\Ch_{\Gamma})$ with $\Gamma'=\bar\varepsilon(\Gamma)$.
We define the \emph{push-forwards} $\alpha_*(\bar\varepsilon):\Delta\to\Delta'$ of $\bar\varepsilon$ via $\alpha$ as the edge-twist $\alpha_*(\bar\varepsilon)\coloneqq\overline{\alpha_*(\varepsilon)}$, where
\[
\alpha_*(\varepsilon)\coloneqq\Ch(\alpha)(\varepsilon) = (\alpha(e), \alpha(C))\in E^{\out}(\Ch_\Delta).
\]
Then it is obvious that $\alpha\in\mathfrak{S}_V$ induces a graph isomorphism $\alpha:\Gamma'\to\Delta'$, which fits into the diagram in Figure~\ref{figure:push-forward 1}.

\noindent (b) Let $\bar\varepsilon'=\overline{(e',C')}:\Delta\to \Delta'$ be an edge-twist.
Then the \emph{pull-back} $\alpha^*(\bar\varepsilon')$ of $\varepsilon'$ via $\alpha$ is defined as the edge-twist $\alpha^*(\bar\varepsilon)\coloneqq\overline{\alpha^*(\varepsilon)}$, where
\[
\alpha^*(\varepsilon') = \alpha^{-1}_*(\varepsilon')=(\alpha^{-1}(e'), \alpha^{-1}(C))\in E^\out(\Ch_\Gamma).
\]
Then as before, the permutation $\alpha'\in\mathfrak{S}_V$ induces a graph isomorphism $\alpha:\Gamma'\to\Delta'$. See Figure~\ref{figure:push-forward 2}.

\noindent (c) For a graph isomorphism $\alpha':\Gamma'\to\Delta'$ and an edge-twist $\bar\varepsilon:\Gamma\to\Gamma'$, we define $\Delta = \alpha'(\Gamma)$.
Then the push-forward $\alpha_*(\bar\varepsilon):\Delta\to\Delta'$ is the edge-twist, which fits into the diagram in Figure~\ref{figure:push-forward 3}.

\noindent (d) For an edge-twist $\bar\varepsilon':\Delta\to\Delta'$, by using the inverse $\alpha'^{-1}$ as before, there exist a graph $\Gamma=\alpha^{-1}(\Delta')$ and an edge-twist $\alpha^*(\bar\varepsilon'):\Gamma\to\Gamma'$. See Figure~\ref{figure:figure:push-forward 4}

\begin{figure}[ht]
\subfigure[\label{figure:push-forward 1}]{
$\begin{tikzcd}[ampersand replacement=\&]
\Gamma\arrow[r, "\alpha"]\arrow[d, "\bar\varepsilon"'] \& \Delta\\
\Gamma'
\end{tikzcd}
\Longrightarrow
\begin{tikzcd}[ampersand replacement=\&]
\Gamma\arrow[r, "\alpha"] \arrow[d, "\bar\varepsilon"']\& \Delta \arrow[d, "\alpha_*(\bar\varepsilon)"]\\
\Gamma'\arrow[r, "\alpha"] \& \alpha(\Gamma')
\end{tikzcd}$
}
\subfigure[\label{figure:push-forward 2}]{
$\begin{tikzcd}[ampersand replacement=\&]
\Gamma\arrow[r, "\alpha"] \& \Delta\arrow[d, "\bar\varepsilon'"]\\
\& \Delta'
\end{tikzcd}
\Longrightarrow
\begin{tikzcd}[ampersand replacement=\&]
\Gamma\arrow[r, "\alpha"] \arrow[d, "\alpha^*(\bar\varepsilon')"']\& \Delta \arrow[d, "\bar\varepsilon'"]\\
\alpha^{-1}(\Delta')\arrow[r, "\alpha"] \& \Delta'
\end{tikzcd}$
}
\subfigure[\label{figure:push-forward 3}]{
$\begin{tikzcd}[ampersand replacement=\&]
\Gamma\arrow[d, "\varepsilon"'] \& \\
\Gamma'\arrow[r, "\alpha'"] \& \Delta'
\end{tikzcd}
\Longrightarrow
\begin{tikzcd}[ampersand replacement=\&]
\Gamma\arrow[r, "\alpha'"] \arrow[d, "\bar\varepsilon"']\& \alpha'(\Gamma) \arrow[d, "\alpha'_*(\bar\varepsilon)"]\\
\Gamma'\arrow[r, "\alpha'"] \& \Delta'
\end{tikzcd}$
}
\subfigure[\label{figure:figure:push-forward 4}]{
$\begin{tikzcd}[ampersand replacement=\&]
\& \Delta\arrow[d, "\bar\varepsilon'"] \\
\Gamma'\arrow[r, "\alpha'"] \& \Delta'
\end{tikzcd}
\Longrightarrow
\begin{tikzcd}[ampersand replacement=\&]
\alpha'^{-1}(\Delta)\arrow[r, "\alpha'"] \arrow[d, "\alpha^*(\bar\varepsilon')"']\& \Delta \arrow[d, "\bar\varepsilon'"]\\
\Gamma'\arrow[r, "\alpha'"] \& \Delta'
\end{tikzcd}$
}
\caption{Push-forwards and pull-backs}
\label{figure:push-forwards and pull-backs}
\end{figure}

\begin{example}\label{example:push-forward}
Recall the graph $\Gamma$ in Example~\ref{example:chunk graph}.
Let $(\alpha_0:\Gamma\to\Gamma)\in\Aut(\Gamma)$ be a graph automorphism which swiches vertices $j$ and $k$ with $\ell$ and $m$, 
\begin{align*}
\alpha_0(j)&=\ell,&
\alpha_0(\ell)&=j,&
\alpha_0(k)&=m,&
\alpha_0(m)&=k,&
\end{align*}
and let $\bar\varepsilon=\overline{(e_2,C_2)}:\Gamma\to\Gamma'$ be an edge-twist.
Then we have a \CLTTF graph $\Delta'=\alpha_0(\Gamma')$ and an edge-twist $(\alpha_0)_*(\bar\varepsilon)=\overline{(e_2, C_3)}:\Delta\to\Delta'$ as depicted in Figure~\ref{figure:example of push-forward}.
\end{example}

\begin{figure}[ht]
\[
\begin{tikzcd}[row sep=1.5pc, column sep=3pc]
\Gamma=\begin{tikzpicture}[baseline=-.5ex]
\draw[fill] (0, -0.5) circle (2pt) node[below=2ex, left=-.5ex] {$a$};
\draw[fill] (1, -0.5) circle (2pt) node[right] {$d$};
\draw[fill] (0, 0.5) circle (2pt) node[above=2ex, left] {$i$};
\draw[fill] (1, 0.5) circle (2pt) node[right] {$e$};
\draw[fill] (0, -1.5) circle (2pt) node[below left] {$b$};
\draw[fill] (1, -1.5) circle (2pt) node[below right] {$c$};
\draw[fill](1, 0.5) ++(72:1) circle (2pt) node[right] {$f$} ++(144:1) circle(2pt) node[above] {$g$} ++(216:1) circle(2pt) node[left] {$h$};
\draw[fill] (-1.5,1.5) circle (2pt) node[above left] {$j$};
\draw[fill] (-1.5,0.5) circle (2pt) node[left] {$k$};
\draw[fill] (-1.5,-0.5) circle (2pt) node[left] {$\ell$};
\draw[fill] (-1.5,-1.5) circle (2pt) node[below left] {$m$};
\draw (0.5,0.5) node[above=-.5ex] {$e_1$} (0,0) node[right=-1ex] {$e_2$} (0.5,-0.5) node[below=-.5ex] {$e_3$};
\draw[color=black, fill=red, fill opacity=0.2](0, -1.5) rectangle node[opacity=1, below] {$C_4$} (1,-0.5);
\draw[color=black, fill=yellow, fill opacity=0.2](1, 0.5) -- ++(72:1) -- ++(144:1) -- ++(216:1) -- ++(288:1);
\draw (0.5, 1) node[opacity=1, above] {$C_1$};
\draw[color=black, fill=white, fill opacity=0.2](0, 0.5) -- ++(-1.5, 1) -- ++(-90:1) node[opacity=1, above right] {$C_2$} -- ++(1.5, -1);
\draw[color=black, fill=black, fill opacity=0.2](0, -0.5) -- ++(-1.5, -1) -- ++(90:1) node[opacity=1, below right] {$C_3$} -- ++(1.5, 1);
\draw[color=black, fill=blue, fill opacity=0.2](0, -0.5) rectangle node[opacity=1, right=-1ex] {$C_0$} (1,0.5);
\end{tikzpicture}
\arrow[r,"\alpha_0"]\arrow[d,"\bar\varepsilon"'] &
\begin{tikzpicture}[baseline=-.5ex]
\draw[fill] (0, -0.5) circle (2pt) node[below=2ex, left=-.5ex] {$a$};
\draw[fill] (1, -0.5) circle (2pt) node[right] {$d$};
\draw[fill] (0, 0.5) circle (2pt) node[above=2ex, left] {$i$};
\draw[fill] (1, 0.5) circle (2pt) node[right] {$e$};
\draw[fill] (0, -1.5) circle (2pt) node[below left] {$b$};
\draw[fill] (1, -1.5) circle (2pt) node[below right] {$c$};
\draw[fill](1, 0.5) ++(72:1) circle (2pt) node[right] {$f$} ++(144:1) circle(2pt) node[above] {$g$} ++(216:1) circle(2pt) node[left] {$h$};
\draw[fill] (-1.5,1.5) circle (2pt) node[above left] {$j$};
\draw[fill] (-1.5,0.5) circle (2pt) node[left] {$k$};
\draw[fill] (-1.5,-0.5) circle (2pt) node[left] {$\ell$};
\draw[fill] (-1.5,-1.5) circle (2pt) node[below left] {$m$};
\draw (0.5,0.5) node[above=-.5ex] {$e_1$} (0,0) node[right=-1ex] {$e_2$} (0.5,-0.5) node[below=-.5ex] {$e_3$};
\draw[color=black, fill=red, fill opacity=0.2](0, -1.5) rectangle node[opacity=1, below] {$C_4$} (1,-0.5);
\draw[color=black, fill=yellow, fill opacity=0.2](1, 0.5) -- ++(72:1) -- ++(144:1) -- ++(216:1) -- ++(288:1);
\draw (0.5, 1) node[opacity=1, above] {$C_1$};
\draw[color=black, fill=white, fill opacity=0.2](0, 0.5) -- ++(-1.5, 1) -- ++(-90:1) node[opacity=1, above right] {$C_2$} -- ++(1.5, -1);
\draw[color=black, fill=black, fill opacity=0.2](0, -0.5) -- ++(-1.5, -1) -- ++(90:1) node[opacity=1, below right] {$C_3$} -- ++(1.5, 1);
\draw[color=black, fill=blue, fill opacity=0.2](0, -0.5) rectangle node[opacity=1, right=-1ex] {$C_0$} (1,0.5);
\end{tikzpicture}=\Gamma
\arrow[d, "(\alpha_0)_*(\bar\varepsilon)"]\\
\Gamma'=\begin{tikzpicture}[baseline=-.5ex]
\draw[fill] (0, -0.5) circle (2pt) node[below=2ex, left=-.5ex] {$a$};
\draw[fill] (1, -0.5) circle (2pt) node[right] {$d$};
\draw[fill] (0, 0.5) circle (2pt) node[above=2ex, left] {$i$};
\draw[fill] (1, 0.5) circle (2pt) node[right] {$e$};
\draw[fill] (0, -1.5) circle (2pt) node[below left] {$b$};
\draw[fill] (1, -1.5) circle (2pt) node[below right] {$c$};
\draw[fill](1, 0.5) ++(72:1) circle (2pt) node[right] {$f$} ++(144:1) circle(2pt) node[above] {$g$} ++(216:1) circle(2pt) node[left] {$h$};
\draw[fill] (-1.5,1.5) circle (2pt) node[above left] {$j$};
\draw[fill] (-1.5,0.5) circle (2pt) node[left] {$k$};
\draw[fill] (-1.5,-0.5) circle (2pt) node[left] {$\ell$};
\draw[fill] (-1.5,-1.5) circle (2pt) node[below left] {$m$};
\draw (0.5,0.5) node[above=-.5ex] {$e_1'$} (0,0) node[right=-1ex] {$e_2'$} (0.5,-0.5) node[below=-.5ex] {$e_3'$};
\draw[color=black, fill=red, fill opacity=0.2](0, -1.5) rectangle node[opacity=1, below] {$C_4'$} (1,-0.5);
\draw[color=black, fill=yellow, fill opacity=0.2](1, 0.5) -- ++(72:1) -- ++(144:1) -- ++(216:1) -- ++(288:1);
\draw (0.5, 1) node[opacity=1, above] {$C_1'$};
\draw[color=black, fill=white, fill opacity=0.2](0, -0.5) -- (-1.5, 1.5) -- ++(0,-1) node[opacity=1, above=1.5ex, right=-.5ex] {$C_2'$} -- (0, 0.5);
\draw[color=black, fill=black, fill opacity=0.2](0, -0.5) -- ++(-1.5, -1) -- ++(90:1) node[opacity=1, below right] {$C_3'$} -- ++(1.5, 1);
\draw[color=black, fill=blue, fill opacity=0.2](0, -0.5) rectangle node[opacity=1, right=-1ex] {$C_0'$} (1,0.5);
\end{tikzpicture}\arrow[r,"\alpha_0"] &
\begin{tikzpicture}[baseline=-.5ex]
\draw[fill] (0, -0.5) circle (2pt) node[below=2ex, left=-.5ex] {$a$};
\draw[fill] (1, -0.5) circle (2pt) node[right] {$d$};
\draw[fill] (0, 0.5) circle (2pt) node[above=2ex, left] {$i$};
\draw[fill] (1, 0.5) circle (2pt) node[right] {$e$};
\draw[fill] (0, -1.5) circle (2pt) node[below left] {$b$};
\draw[fill] (1, -1.5) circle (2pt) node[below right] {$c$};
\draw[fill](1, 0.5) ++(72:1) circle (2pt) node[right] {$f$} ++(144:1) circle(2pt) node[above] {$g$} ++(216:1) circle(2pt) node[left] {$h$};
\draw[fill] (-1.5,1.5) circle (2pt) node[above left] {$j$};
\draw[fill] (-1.5,0.5) circle (2pt) node[left] {$k$};
\draw[fill] (-1.5,-0.5) circle (2pt) node[left] {$\ell$};
\draw[fill] (-1.5,-1.5) circle (2pt) node[below left] {$m$};
\draw (0.5,0.5) node[above=-.5ex] {$e_1''$} (0,0) node[right=-1ex] {$e_2''$} (0.5,-0.5) node[below=-.5ex] {$e_3''$};
\draw[color=black, fill=red, fill opacity=0.2](0, -1.5) rectangle node[opacity=1, below] {$C_4$} (1,-0.5);
\draw[color=black, fill=yellow, fill opacity=0.2](1, 0.5) -- ++(72:1) -- ++(144:1) -- ++(216:1) -- ++(288:1);
\draw (0.5, 1) node[opacity=1, above] {$C_1''$};
\draw[color=black, fill=black, fill opacity=0.2](0, 0.5) -- ++(-1.5, 1) -- ++(-90:1) node[opacity=1, above right] {$C_2''$} -- ++(1.5, -1);
\draw[color=black, fill=white, fill opacity=0.2](0, 0.5) -- (-1.5, -1.5) -- ++(0,1) node[opacity=1, below=1.5ex, right=-0.5ex] {$C_3''$} -- (0, -0.5);
\draw[color=black, fill=blue, fill opacity=0.2](0, -0.5) rectangle node[opacity=1, right=-1ex] {$C_0''$} (1,0.5);
\end{tikzpicture}=\Delta'
\end{tikzcd}
\]
\caption{An example of push-forward}
\label{figure:example of push-forward}
\end{figure}
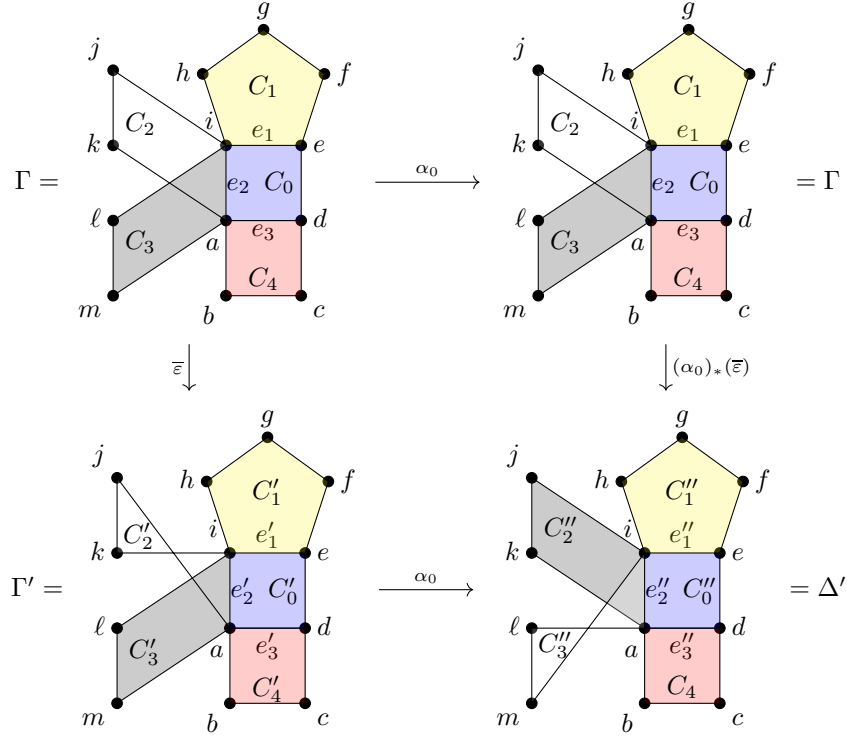

\begin{remark}
Notice that in the previous example, the graph isomorphism $\alpha_0:\Gamma'\to\Delta'$ is not an automorphism anymore.
\end{remark}

\subsection{The group \texorpdfstring{$\Iso(\Gamma)$}{Twist}}\label{section:iso}
As mentioned earlier in Remark~\ref{remark:iso}, we will consider graph isomorphisms which can be expressed as composititions of edge-twists as well in this section. Namely, those are graph isomorphisms between edge-twist equivalent graphs.
Let $\Delta$ be a graph which is edge-twist equivalent to and isomorphic to $\Gamma$. Namely,
\[
\Delta\sim\Gamma\quad\text{ and }\quad\Delta\cong\Gamma.
\]
In other words, there are two ways of obtaining $\Delta$ from $\Gamma$ such that for some $\alpha\in\mathfrak{S}_V$ and $\bar\eta:E^\out_\odd(\Ch_{\llbracket\Gamma\rrbracket})\to\mathbb{Z}_2$, 
\[
\bar{\mathcal{E}}=\bar{\mathcal{E}}(\bar\eta):\Gamma\to \Delta=\alpha(\Gamma),
\]
Then the set $\Iso(\Gamma)$ consists of such graph isomorphisms
\[
\Iso(\Gamma) = \{\alpha\in\mathfrak{S}_V\mid \Gamma\sim \alpha(\Gamma)\}\subset\mathfrak{S}_V.
\]
\begin{remark}
By Proposition~\ref{proposition:normal form for graphs}, the composition $\bar{\mathcal{E}}$ is uniquely determined only by $\Delta$.
Hence we will not lose any information even though we throw out $\bar{\mathcal{E}}$.
\end{remark}

\begin{lemma}
The set $\Iso(\Gamma)$ is closed under the composition.
\end{lemma}
\begin{proof}
For $i=1,2$, let $\alpha_i\in\mathfrak{S}_V$ with $\alpha_i(\Gamma)=\Delta_i$ be elements in $\Iso(\Gamma)$.
Then we need to show that for the composition $\alpha=\alpha_2\alpha_1$,
\[
\Gamma\sim \Delta=\alpha(\Gamma).
\]

By definition of $\Iso(\Gamma)$ and Proposition~\ref{proposition:normal form for graphs}, there exist unique compositions $\mathcal{E}_{\alpha_1}$ and $\mathcal{E}_{\alpha_2}$ of edge-twists such that for $i=1,2$,
\[
\bar{\mathcal{E}}_{\alpha_i}:\Gamma\to \Delta_i.
\]
Then as seen in the previous section, we have the push-forward $\bar{\mathcal{E}}_{\alpha_1}'\coloneqq(\alpha_2)_*(\bar{\mathcal{E}}_{\alpha_1}):\Delta_2\to\Delta$ for $\Delta=\alpha_2(\Delta_1)$, which fits into the diagram
\[
\begin{tikzcd}
\Gamma\arrow[r, "\alpha_1"] & \Delta_1 \arrow[from=d, "\bar{\mathcal{E}}_{\alpha_1}"] \arrow[r,"\alpha_2"] & \Delta \arrow[from=d,"(\alpha_2)_*(\bar{\mathcal{E}}_{\alpha_1})"']\\
& \Gamma\arrow[r, "\alpha_2"] & \Delta_2 \arrow[from=d, "\bar{\mathcal{E}}_{\alpha_2}"']\\
& & \Gamma
\end{tikzcd}
\]
Therefore the graph $\Delta$ is edge-twist equivalent to $\Gamma$ via the composition 
\[
\bar{\mathcal{E}}_\alpha\coloneqq\bar{\mathcal{E}}_{\alpha_1}'\bar{\mathcal{E}}_{\alpha_2}:\Gamma\to\Delta
\]
and we are done.
\end{proof}

\begin{theorem}
The set $\Iso(\Gamma)$ has a group structure with respect to the composition.
\end{theorem}
\begin{proof}
Obviously, the identity isomorphism $\operatorname{Id}:\Gamma\to\Gamma$ is in $\Iso(\Gamma)$ and plays the role of the identity under the composition.

Let $\alpha\in\mathfrak{S}_V$ be in $\Iso(\Gamma)$ and let $\bar{\mathcal{E}}_\alpha:\Gamma\to \alpha(\Gamma)$ be a unique composition of edge-twists by Proposition~\ref{proposition:normal form for graphs}.
We will consider the inverse $\alpha^{-1}\in\mathfrak{S}_V$.
Since $\alpha^{-1}\alpha=\operatorname{Id}$, it suffices to show that $\alpha^{-1}(\Gamma)\sim \Gamma$.
Indeed, we have the diagram
\[
\begin{tikzcd}
\Gamma\arrow[r, "\alpha"] & \Delta \arrow[from=d, "\bar{\mathcal{E}}_\alpha"] \arrow[r,"\alpha^{-1}"] & \Gamma \arrow[from=d,"\alpha^{-1}_*(\bar{\mathcal{E}}_\alpha)"']\\
& \Gamma\arrow[r, "\alpha^{-1}"] & \alpha^{-1}(\Gamma)
\end{tikzcd}
\]
and therefore $\alpha^{-1}(\Gamma)$ is edge-twist equivalent to $\Gamma$ via $\alpha_*^{-1}(\bar{\mathcal{E}}_\alpha)$ and so $\alpha^{-1}\in\Iso(\Gamma)$.

Finally, since the composition is associative, the set $\Iso(\Gamma)$ has a group structure as claimed.
\end{proof}

\begin{remark}
The group $\Iso(\Gamma)$ is isomorphic to the hom-set of the twist equivalence groupoid $Twist(\mathcal{G})(\Gamma,\Gamma)$ described in \cite{Cr05}.
\end{remark}

One can check easily that the (graph) automorphism group $\Aut(\Gamma)$ of $\Gamma$ is a subgroup of $\Iso(\Gamma)$ so that
\[
[\Iso(\Gamma):\Aut(\Gamma)]=\#(\llbracket \Gamma\rrbracket).
\]
However, it is not necessarily normal in general.

\begin{example}\label{example:generators of Iso}
Recall the graph $\Gamma$ depicted in Example~\ref{example:chunk graph}.
One can easily check that 
\[
\Aut(\Gamma) = \langle \alpha_0\mid \alpha_0^2\rangle\cong \mathbb{Z}_2,
\]
where $\alpha_0$ is given in Example~\ref{example:push-forward}.
Then the group $\Iso(\Gamma)$ is generated by graph isomorphisms
\[
\{\alpha_0, \alpha_1,\dots, \alpha_4\}\subset\mathfrak{S}_V
\]
as depicted in Figure~\ref{figure:example of Iso}.
\begin{figure}[ht]
\[
\begin{tikzcd}[row sep=-1pc, column sep=1pc]
\begin{tikzpicture}[baseline=-.5ex]
\draw[fill] (0, -0.5) circle (2pt) node[below=2ex, left=-.5ex] {$a$};
\draw[fill] (1, -0.5) circle (2pt) node[right] {$d$};
\draw[fill] (0, 0.5) circle (2pt) node[above=2ex, left] {$i$};
\draw[fill] (1, 0.5) circle (2pt) node[right] {$e$};
\draw[fill] (0, -1.5) circle (2pt) node[below left] {$b$};
\draw[fill] (1, -1.5) circle (2pt) node[below right] {$c$};
\draw[fill](1, 0.5) ++(72:1) circle (2pt) node[right] {$f$} ++(144:1) circle(2pt) node[above] {$g$} ++(216:1) circle(2pt) node[left] {$h$};
\draw[fill] (-1.5,1.5) circle (2pt) node[above left] {$j$};
\draw[fill] (-1.5,0.5) circle (2pt) node[left] {$k$};
\draw[fill] (-1.5,-0.5) circle (2pt) node[left] {$\ell$};
\draw[fill] (-1.5,-1.5) circle (2pt) node[below left] {$m$};
\draw (0.5,0.5) node[above=-.5ex] {$e_1$} (0,0) node[right=-1ex] {$e_2$} (0.5,-0.5) node[below=-.5ex] {$e_3$};
\draw[color=black, fill=red, fill opacity=0.2](0, -1.5) rectangle node[opacity=1, below] {$C_4$} (1,-0.5);
\draw[color=black, fill=yellow, fill opacity=0.2](1, 0.5) ++(72:1) -- ++(144:1) -- ++(216:1) -- (1,0.5) -- (0,0.5) -- cycle;
\draw (0.5, 1) node[opacity=1, above] {$C_1$};
\draw[color=black, fill=white, fill opacity=0.2](0, 0.5) -- ++(-1.5, 1) -- ++(-90:1) node[opacity=1, above right] {$C_2$} -- ++(1.5, -1);
\draw[color=black, fill=black, fill opacity=0.2](0, -0.5) -- ++(-1.5, -1) -- ++(90:1) node[opacity=1, below right] {$C_3$} -- ++(1.5, 1);
\draw[color=black, fill=blue, fill opacity=0.2](0, -0.5) rectangle node[opacity=1, right=-1ex] {$C_0$} (1,0.5);
\end{tikzpicture}
\arrow[from=rd, bend right, "\alpha_1:h\leftrightarrow f"']
& & 
\begin{tikzpicture}[baseline=-.5ex]
\draw[fill] (0, -0.5) circle (2pt) node[below=2ex, left=-.5ex] {$a$};
\draw[fill] (1, -0.5) circle (2pt) node[right] {$d$};
\draw[fill] (0, 0.5) circle (2pt) node[above=2ex, left] {$i$};
\draw[fill] (1, 0.5) circle (2pt) node[right] {$e$};
\draw[fill] (0, -1.5) circle (2pt) node[below left] {$b$};
\draw[fill] (1, -1.5) circle (2pt) node[below right] {$c$};
\draw[fill](1, 0.5) ++(72:1) circle (2pt) node[right] {$f$} ++(144:1) circle(2pt) node[above] {$g$} ++(216:1) circle(2pt) node[left] {$h$};
\draw[fill] (-1.5,1.5) circle (2pt) node[above left] {$j$};
\draw[fill] (-1.5,0.5) circle (2pt) node[left] {$k$};
\draw[fill] (-1.5,-0.5) circle (2pt) node[left] {$\ell$};
\draw[fill] (-1.5,-1.5) circle (2pt) node[below left] {$m$};
\draw (0.5,0.5) node[above=-.5ex] {$e_1$} (0,0) node[right=-1ex] {$e_2$} (0.5,-0.5) node[below=-.5ex] {$e_3$};
\draw[color=black, fill=red, fill opacity=0.2](0, -1.5) rectangle node[opacity=1, below] {$C_4$} (1,-0.5);
\draw[color=black, fill=yellow, fill opacity=0.2](1, 0.5) -- ++(72:1) -- ++(144:1) -- ++(216:1) -- ++(288:1);
\draw (0.5, 1) node[opacity=1, above] {$C_1$};
\draw[color=black, fill=white, fill opacity=0.2](0, -0.5) -- (-1.5, 1.5) -- ++(0,-1) node[opacity=1, above=1.5ex, right=-.5ex] {$C_2$} -- (0, 0.5);
\draw[color=black, fill=black, fill opacity=0.2](0, -0.5) -- ++(-1.5, -1) -- ++(90:1) node[opacity=1, below right] {$C_3$} -- ++(1.5, 1);
\draw[color=black, fill=blue, fill opacity=0.2](0, -0.5) rectangle node[opacity=1, right=-1ex] {$C_0$} (1,0.5);
\end{tikzpicture}
\arrow[from=ld, bend left, "\alpha_2:j\leftrightarrow k"]
\\
&
\begin{tikzpicture}[baseline=-.5ex]
\draw[fill] (0, -0.5) circle (2pt) node[below=2ex, left=-.5ex] {$a$};
\draw[fill] (1, -0.5) circle (2pt) node[right] {$d$};
\draw[fill] (0, 0.5) circle (2pt) node[above=2ex, left] {$i$};
\draw[fill] (1, 0.5) circle (2pt) node[right] {$e$};
\draw[fill] (0, -1.5) circle (2pt) node[below left] {$b$};
\draw[fill] (1, -1.5) circle (2pt) node[below right] {$c$};
\draw[fill](1, 0.5) ++(72:1) circle (2pt) node[right] {$f$} ++(144:1) circle(2pt) node[above] {$g$} ++(216:1) circle(2pt) node[left] {$h$};
\draw[fill] (-1.5,1.5) circle (2pt) node[above left] {$j$};
\draw[fill] (-1.5,0.5) circle (2pt) node[left] {$k$};
\draw[fill] (-1.5,-0.5) circle (2pt) node[left] {$\ell$};
\draw[fill] (-1.5,-1.5) circle (2pt) node[below left] {$m$};
\draw (0.5,0.5) node[above=-.5ex] {$e_1$} (0,0) node[right=-1ex] {$e_2$} (0.5,-0.5) node[below=-.5ex] {$e_3$};
\draw[color=black, fill=red, fill opacity=0.2](0, -1.5) rectangle node[opacity=1, below] {$C_4$} (1,-0.5);
\draw[color=black, fill=yellow, fill opacity=0.2](1, 0.5) -- ++(72:1) -- ++(144:1) -- ++(216:1) -- ++(288:1);
\draw (0.5, 1) node[opacity=1, above] {$C_1$};
\draw[color=black, fill=white, fill opacity=0.2](0, 0.5) -- ++(-1.5, 1) -- ++(-90:1) node[opacity=1, above right] {$C_2$} -- ++(1.5, -1);
\draw[color=black, fill=black, fill opacity=0.2](0, -0.5) -- ++(-1.5, -1) -- ++(90:1) node[opacity=1, below right] {$C_3$} -- ++(1.5, 1);
\draw[color=black, fill=blue, fill opacity=0.2](0, -0.5) rectangle node[opacity=1, right=-1ex] {$C_0$} (1,0.5);
\end{tikzpicture}
\arrow[ul, bend left, "{\bar\varepsilon_1=\overline{(e_1,C_1)}}"]
\arrow[ur, bend right, "{\bar\varepsilon_2=\overline{(e_2,C_2)}}"']
\arrow[dl, bend right, "{\bar\varepsilon_4=\overline{(e_3,C_4)}}"']
\arrow[dr, bend left, "{\bar\varepsilon_3=\overline{(e_2,C_3)}}"]
& \\
\begin{tikzpicture}[baseline=-.5ex]
\draw[fill] (0, -0.5) circle (2pt) node[below=2ex, left=-.5ex] {$a$};
\draw[fill] (1, -0.5) circle (2pt) node[right] {$d$};
\draw[fill] (0, 0.5) circle (2pt) node[above=2ex, left] {$i$};
\draw[fill] (1, 0.5) circle (2pt) node[right] {$e$};
\draw[fill] (0, -1.5) circle (2pt) node[below left] {$b$};
\draw[fill] (1, -1.5) circle (2pt) node[below right] {$c$};
\draw[fill](1, 0.5) ++(72:1) circle (2pt) node[right] {$f$} ++(144:1) circle(2pt) node[above] {$g$} ++(216:1) circle(2pt) node[left] {$h$};
\draw[fill] (-1.5,1.5) circle (2pt) node[above left] {$j$};
\draw[fill] (-1.5,0.5) circle (2pt) node[left] {$k$};
\draw[fill] (-1.5,-0.5) circle (2pt) node[left] {$\ell$};
\draw[fill] (-1.5,-1.5) circle (2pt) node[below left] {$m$};
\draw (0.5,0.5) node[above=-.5ex] {$e_1$} (0,0) node[right=-1ex] {$e_2$} (0.5,-0.5) node[below=-.5ex] {$e_3$};
\draw[color=black, fill=red, fill opacity=0.2](0, -.5) -- (1,-1.5) -- node[midway, opacity=1, above] {$C_4$} (0,-1.5) -- (1,-0.5);
\draw[color=black, fill=yellow, fill opacity=0.2](1, 0.5) -- ++(72:1) -- ++(144:1) -- ++(216:1) -- ++(288:1);
\draw (0.5, 1) node[opacity=1, above] {$C_1$};
\draw[color=black, fill=white, fill opacity=0.2](0, 0.5) -- ++(-1.5, 1) -- ++(-90:1) node[opacity=1, above right] {$C_2$} -- ++(1.5, -1);
\draw[color=black, fill=black, fill opacity=0.2](0, -0.5) -- ++(-1.5, -1) -- ++(90:1) node[opacity=1, below right] {$C_3$} -- ++(1.5, 1);
\draw[color=black, fill=blue, fill opacity=0.2](0, -0.5) rectangle node[opacity=1, right=-1ex] {$C_0$} (1,0.5);
\end{tikzpicture}
\arrow[from=ru, bend left, "\alpha_4:b\leftrightarrow c"]
& & 
\begin{tikzpicture}[baseline=-.5ex]
\draw[fill] (0, -0.5) circle (2pt) node[below=2ex, left=-.5ex] {$a$};
\draw[fill] (1, -0.5) circle (2pt) node[right] {$d$};
\draw[fill] (0, 0.5) circle (2pt) node[above=2ex, left] {$i$};
\draw[fill] (1, 0.5) circle (2pt) node[right] {$e$};
\draw[fill] (0, -1.5) circle (2pt) node[below left] {$b$};
\draw[fill] (1, -1.5) circle (2pt) node[below right] {$c$};
\draw[fill](1, 0.5) ++(72:1) circle (2pt) node[right] {$f$} ++(144:1) circle(2pt) node[above] {$g$} ++(216:1) circle(2pt) node[left] {$h$};
\draw[fill] (-1.5,1.5) circle (2pt) node[above left] {$j$};
\draw[fill] (-1.5,0.5) circle (2pt) node[left] {$k$};
\draw[fill] (-1.5,-0.5) circle (2pt) node[left] {$\ell$};
\draw[fill] (-1.5,-1.5) circle (2pt) node[below left] {$m$};
\draw (0.5,0.5) node[above=-.5ex] {$e_1$} (0,0) node[right=-1ex] {$e_2$} (0.5,-0.5) node[below=-.5ex] {$e_3$};
\draw[color=black, fill=red, fill opacity=0.2](0, -1.5) rectangle node[opacity=1, below] {$C_4$} (1,-0.5);
\draw[color=black, fill=yellow, fill opacity=0.2](1, 0.5) -- ++(72:1) -- ++(144:1) -- ++(216:1) -- ++(288:1);
\draw (0.5, 1) node[opacity=1, above] {$C_1$};
\draw[color=black, fill=white, fill opacity=0.2](0, 0.5) -- ++(-1.5, 1) -- ++(-90:1) node[opacity=1, above right] {$C_2$} -- ++(1.5, -1);
\draw[color=black, fill=black, fill opacity=0.2](0, 0.5) -- (-1.5, -1.5) -- ++(0,1) node[opacity=1, below=1.5ex, right=-.5ex] {$C_3$} -- (0, -0.5);
\draw[color=black, fill=blue, fill opacity=0.2](0, -0.5) rectangle node[opacity=1, right=-1ex] {$C_0$} (1,0.5);
\end{tikzpicture}
\arrow[from=lu, bend right, "\alpha_4:\ell\leftrightarrow m"']
\end{tikzcd}
\]
\caption{Examples of elements of $\Iso(\Gamma)$}
\label{figure:example of Iso}
\end{figure}

One can check that for each $0\le i\le j\le 4$, 
\begin{equation}\label{equation:relation of Iso}
\alpha_i\alpha_j =\begin{cases}
\operatorname{Id} & i=j;\\
\alpha_3\alpha_0 & i=0, j=2;\\
\alpha_2\alpha_0 & i=0, j=3;\\
\alpha_j\alpha_i & \text{otherwise}
\end{cases}
\end{equation}
and therefore we have an isomorphism
\[
\Iso(\Gamma)\cong \mathbb{Z}_2\times\mathbb{Z}_2\times
\left(
(\mathbb{Z}_2\times \mathbb{Z}_2)\rtimes \mathbb{Z}_2
\right),
\]
where each factor from the left is generated by $\alpha_1, \alpha_4, \alpha_2, \alpha_3$ and $\alpha_0$.
\end{example}

\begin{corollary}
For a discretely rigid graph $\Gamma$, we have $\Iso(\Gamma)\cong \Aut(\Gamma)$.
\end{corollary}
\begin{proof}
The hypothesis implies that for each $\alpha\in\Iso(\Gamma)$, we have $\alpha(\Gamma)=\Gamma$, i.e., $\alpha\in\Aut(\Gamma)$ and therefore $\Iso(\Gamma)\cong\Aut(\Gamma)$.
\end{proof}

\subsection{The category \texorpdfstring{$\mathscr{G}$}{} of \CLTTF graphs}

From now on, we mean edge-twists by edge-twists along outward edges in the chunk tree unless mentioned otherwise.

\begin{definition}
Let $\tilde{\mathscr{G}}$ be the category of \CLTTF graphs defined as follows:
\begin{itemize}
\item The objects are edge-separated \CLTTF graphs.
\item The hom-set is freely generated by graph isomorphisms and edge-twists.
\end{itemize}
\end{definition}

In other words, for any morphism $f\in\hom(\Gamma,\Delta)$ is a composition
\[
f:\Gamma=\Gamma_0\xrightarrow{f_1}\Gamma_1\xrightarrow{f_2}\cdots\xrightarrow{f_n}\Gamma_n=\Delta,
\]
where $f_i:\Gamma_{i-1}\rightarrow \Gamma_i$ is either a graph isomorphism or an edge-twist.

Notice that the hom-set of the category $\tilde{\mathscr{G}}$ is \emph{freely} generated.
In particular, any $\varepsilon=(e,C)\in E^{\out}(\Ch_{\llbracket\Gamma\rrbracket})$ with an even-labeled separating edge $e\subset \Gamma$ induces an endomorphism $\bar\varepsilon:\Gamma\to\Gamma$ but it will never be regarded as the identity.
Furthermore, if $e$ is odd-labeled, then there is an edge-twist morphism $\bar\varepsilon:\Gamma\to\Delta$ and so by regarding $\varepsilon$ as an edge in $E^\out(\Ch_{\llbracket\Gamma\rrbracket})$, it defines an edge-twist $\bar\varepsilon:\Delta\to \Gamma$ again.
However, in $\tilde{\mathscr{G}}$, the composition $\bar\varepsilon^2$ is \emph{not} the identity.
Therefore, we will denote each edge-twist in the category $\tilde{\mathscr{G}}$ by $\varepsilon$ instead of $\bar\varepsilon$ in order to avoid the confusion as above.

Now let $\mathscr{E}$ be the set of morphisms generated by edge-twists. Then by localizing $\tilde{\mathscr{G}}$ with respect to $\mathscr{E}$, we obtain the category
\[
\bar{\mathscr{G}}=\tilde{\mathscr{G}}[\mathscr{E}^{-1}].
\]
In other words, in the category $\bar{\mathscr{G}}$, we have the formal inverse $\varepsilon^{-1}\in\Hom_{\bar{\mathscr{G}}}(\Delta,\Gamma)$ of each edge-twist $\varepsilon\in\Hom_{\bar{\mathscr{G}}}(\Gamma,\Delta)$.
Hence any morphism $f\in\Hom_{\bar{\mathscr{G}}}(\Gamma,\Delta)$ is a composition 
\[
f:\Gamma=\Gamma_0\xrightarrow{f_1}\Gamma_1\xrightarrow{f_2}\cdots\xrightarrow{f_n}\Gamma_n=\Delta,
\]
where $f_i:\Gamma_{i-1}\rightarrow \Gamma_i$ is either 
\begin{enumerate}
\item a graph isomorphism $\alpha_i$,
\item an edge-twist $\varepsilon_i$, or
\item a formal inverse $\varepsilon_i^{-1}$ of an edge-twist $\varepsilon_i$.
\end{enumerate}

We also define an equivalence relation on the hom-set of $\bar{\mathscr{G}}$ generated by the following three types of relations:
\begin{enumerate}
\item for two graph isomorphisms $\alpha:\Gamma\to \Gamma'$ and $\beta:\Gamma'\to \Gamma''$ whose composition is $\gamma:\Gamma\to \Gamma''$,
\[
\beta\alpha\sim \gamma.
\]
\item for each pull-back (or push-forward) diagram of graph isomorphisms and edge-twists
\[
\begin{tikzcd}
\Gamma \arrow[r, "\alpha"] \arrow[d, "\varepsilon"'] & \Delta \arrow[d, "\varepsilon'"]\\
\Gamma' \arrow[r, "\alpha'"] & \Delta',
\end{tikzcd}
\]
there is a relation
\begin{equation}\label{equation:commutativity 1}
\varepsilon' \alpha \sim \alpha' \varepsilon.
\end{equation}
\item for each pull-back (or push-forward) diagram of edge-twists
\[
\begin{tikzcd}
\Gamma \arrow[r, "\varepsilon_1"] \arrow[d, "\varepsilon_2"'] & \Gamma_1 \arrow[d, "\varepsilon_2"]\\
\Gamma_2 \arrow[r, "\varepsilon_1"] & \Gamma',
\end{tikzcd}
\]
there is a relation
\begin{equation}\label{equation:commutativity 2}
\varepsilon_2 \varepsilon_1 \sim \varepsilon_1 \varepsilon_2.
\end{equation}
\end{enumerate}

\begin{definition}
The category $\mathscr{G}$ is defined to be the quotient category
\[
\mathscr{G} = \bar{\mathscr{G}}/\sim = \left(\tilde{\mathscr{G}}[\mathscr{E}^{-1}]\right)/\sim.
\]
\end{definition}

Let $\Gamma=(V,E,m)\in \mathscr{G}$ be a \CLTTF graph.
By the relations \eqref{equation:commutativity 1} and \eqref{equation:commutativity 2}, any isomorphism $f\in\Hom_{\mathscr{G}}(\Gamma,\Delta)$ is a composition
\begin{align}\label{equation:normal form}
f&=\mathcal{E}\alpha,&
\mathcal{E}&= \mathcal{E}(\eta)=\prod_{\varepsilon\in E^{\out}(\Ch_{\llbracket\Gamma\rrbracket})} \varepsilon^{\eta(\varepsilon)},
\end{align}
where $\alpha$ is a graph isomorphism and $\eta:E^{\out}(\Ch_{\llbracket\Gamma\rrbracket})\to \mathbb{Z}$ is a function.

\begin{definition}[Even edge-twists]
We say that a composiiton $\mathcal{E}=\mathcal{E}(\eta)$ of edge-twists in $\mathscr{G}$ is \emph{even} if $\eta(\varepsilon)$ is even for every $\varepsilon\in E^{\out}_\odd(\Ch_{\llbracket\Gamma\rrbracket})$.

A subset $\Dehn_{\mathscr{G}}(\Gamma)$ of $\Aut_{\mathscr{G}}(\Gamma)$ is defined as the set of even compositions of edge-twists.
\end{definition}

\begin{lemma}\label{lemma:even edge-twists}
Let $\mathcal{E}:\Gamma\to\Delta$ be a composition of edge-twists in $\mathscr{G}$.
Then $\Gamma=\Delta$ if and only if $\mathcal{E}\in \Dehn_{\mathscr{G}}(\Gamma)$.
\end{lemma}
\begin{proof}
Assume that $\mathcal{E}=\mathcal{E}(\eta)$ for some $\eta:E^{\out}(\Ch_{\llbracket\Gamma\rrbracket})\to\mathbb{Z}$.
If $\mathcal{E}$ is not even, then there exists $\varepsilon\in E^{\out}(\Ch_{\llbracket\Gamma\rrbracket})$ with odd $m_{\llbracket\Gamma\rrbracket}(\varepsilon)$ so that $\eta(\varepsilon)$ is odd.
Then the resulting graph never be the same as $\Gamma$ as mentioned earlier.
Therefore $\mathcal{E}$ should be even.

Conversely, any even $\mathcal{E}$ obviously gives us an automorphism $\mathcal{E}:\Gamma\to\Gamma$.
\end{proof}

\begin{corollary}\label{corollary:even edge-twists}
The set $\Dehn_{\mathscr{G}}(\Gamma)$ is a normal subgroup of $\Aut_{\mathscr{G}}(\Gamma)$ and isomorphic to the free abelian group $\mathbb{Z}^{\#(E^\out(\Ch_{\llbracket\Gamma\rrbracket}))}$.
\end{corollary}
\begin{proof}
By the above lemma, even edge-twists form a group, which is free abelian by the relation \eqref{equation:commutativity 2} and normal in $\Aut_{\mathscr{G}}(\Gamma)$ since the conjugate of an even edge-twist by a graph automorphism is again even.

Finally, the group of even edge-twists is isomorphic to a free abelian group generated by the set
\begin{align*} 
\{\varepsilon\mid \varepsilon\in E^{\out}_\even(\Ch_{\llbracket\Gamma\rrbracket})\}
&\cup
\{\varepsilon^2\mid \varepsilon\in E^{\out}_\odd(\Ch_{\llbracket\Gamma\rrbracket})\},
\end{align*}
which has one-to-one correspondence with $E^{\out}(\Ch_{\llbracket\Gamma\rrbracket})$ and we are done.
\end{proof}

\begin{proposition}\label{proposition:graph isomorphism and edge-twists}
Let $\alpha,\mathcal{E} \in \Hom_{\mathscr{G}}(\Gamma,\Delta)$ such that $\alpha$ is a graph isomorphism and $\mathcal{E}$ is a composition of edge-twists and their inverses.
Suppose that $\alpha=\mathcal{E}$ in $\Hom_{\mathscr{G}}(\Gamma,\Delta)$. Then $\Gamma=\Delta$ and both $\alpha$ and $\mathcal{E}$ are the identities.
\end{proposition}
This proposition is evident since the only relations in \eqref{equation:commutativity 1} and \eqref{equation:commutativity 2} do not cancel a graph isomorphism with a composition of edge-twists.
However, we will give a concrete proof later.

Under the aid of Proposition~\ref{proposition:graph isomorphism and edge-twists}, we have the following theorem.
\begin{theorem}\label{theorem:normal form in G}
For each isomorphism $f\in\Hom_{\mathscr{G}}(\Gamma,\Delta)$, there is a unique pair of a graph isomorphism $\alpha$ and a composition $\mathcal{E}$ of edge-twists or inverses such that
\[
f=\mathcal{E}\alpha.
\]
\end{theorem}
\begin{proof}
Suppose that $f$ has two such expressions $f=\mathcal{E}\alpha= \mathcal{E}'\alpha'$.
%Then it suffices to show that $\alpha=\alpha'$.
By pre- and post-compositions of $\alpha'^{-1}$ and $\mathcal{E}^{-1}$, we have $\alpha\alpha'^{-1} = \mathcal{E}^{-1}\mathcal{E}'$,
which should be the identity by Proposition~\ref{proposition:graph isomorphism and edge-twists} and so $\alpha=\alpha'$ and $\mathcal{E}=\mathcal{E}'$ as desired.
\end{proof}

As an immediate consequence, we have the following corollary.
\begin{corollary}\label{corollary:rigid automorphism in G}
Let $\Gamma$ be a discretely rigid \CLTTF graph. Then the automorphism group $\Aut_{\mathscr{G}}(\Gamma)$ is the semidirect product of the free abelian group generated by edge-twists and the automorphism group of $\Gamma$.
\[
\Aut_{\mathscr{G}}(\Gamma)\cong
\Dehn_{\mathscr{G}}(\Gamma)\rtimes \Aut(\Gamma)
\cong \mathbb{Z}^{\#(E^{\out}(\Ch_{\Gamma}))}\rtimes \Aut(\Gamma)
\]
\end{corollary}
\begin{proof}
Let $f=\mathcal{E}\alpha\in \Aut_{\mathscr{G}}(\Gamma)$ with a graph isomorphism $\alpha:\Gamma\to\Delta$ and a composition of edge-twists $\mathcal{E}:\Delta\to\Gamma$ so that $\Gamma\sim\Delta$ and $\Gamma\cong\Delta$.
Since $\Gamma$ is discretely rigid, $\Gamma=\Delta$ and so both $\alpha$ and $\mathcal{E}$ are automorphisms.
Therefore $\Aut_{\mathscr{G}}(\Gamma)$ is generated by $\Aut(\Gamma)$ and $\Dehn_{\mathscr{G}}(\Gamma)$ by Lemma~\ref{lemma:even edge-twists}.

Finally, by Corollary~\ref{corollary:even edge-twists}, Proposition~\ref{proposition:graph isomorphism and edge-twists} and Theorem~\ref{theorem:normal form in G}, we are done.
\end{proof}

\begin{example}[Special automorphism $\Phi$]\label{example:special automorphism}
Let $*_\Gamma$ be the central vertex of $\Ch_\Gamma$.
If $*_\Gamma$ is a chunk, then $\Phi\in\Aut_{{\mathscr{G}}}(\Gamma)$ will be defined to be the identity.

Suppose that $*_\Gamma=e=\{s,t\}$ is a separating edge. Let
\[
\{\varepsilon_1,\dots, \varepsilon_N\mid \varepsilon_i=(e, C_i)\}\subset E^\out(\Ch_\Gamma)
\]
be the subset of edges adjacent to $*_\Gamma$ of $\Ch_\Gamma$.
We define the composition 
\begin{equation}\label{equation:edge-twist of the center}
\mathcal{E}_{*_\Gamma} = \varepsilon_1\varepsilon_2\cdots\varepsilon_N.
\end{equation}
of all edge-twists $\varepsilon_1,\dots, \varepsilon_N$.

If the label $m(e)$ is even, then $\Delta=\Gamma$ and so $\mathcal{E}$ is an automorphism and let $\Phi=\mathcal{E}_{*_\Gamma}$.

Otherwise, notice that $\Delta$ is obtained by interchanging the roles of vertices $s$ and $t$, and isomorphic to $\Gamma$.
The precise graph isomorphism $\alpha_{*_\Gamma}:\Gamma\to\Delta$ is given by
\[
\alpha_{*_\Gamma}(v)=\begin{cases}
v & v\not\in\{s,t\};\\
t & v=s;\\
s & v=t.
\end{cases}
\]
Then we define $\Phi$ to be the composition $\mathcal{E}_{*_\Gamma}\alpha_{*_\Gamma}:\Gamma\to \Gamma$.

In summary, the special automorphism $\Phi$ is defined as
\begin{equation}\label{equation:special automorphism}
\Phi\coloneqq\begin{cases}
\operatorname{Id} & *_\Gamma\text{ is a chunk};\\
\mathcal{E}_{*_\Gamma} & *_\Gamma\text{ is an even-labeled separating edge};\\
\mathcal{E}_{*_\Gamma}\alpha_{*_\Gamma} & *_\Gamma\text{ is an odd-labeled separating edge}.
\end{cases}
\end{equation}
\end{example}

\begin{remark}\label{remark:rigid means central chunk}
Observe that if $*_\Gamma$ is an odd-labeled separating edge, then $\Gamma$ can not be discretely rigid.
Conversely, for any discretely rigid \CLTTF graph $\Gamma$, the central vertex $*_\Gamma$ is either a chunk or an even-labeled separating edge.
\end{remark}

\section{\CLTTF Artin groups}\label{section:CLTTF Artin groups}

\subsection{\CLTTF Artin groups and their isomorphisms}
Let $\Gamma=(V,E,m)$ be a \CLTTF graph. An {\em Artin group} $A_\Gamma$ with a defining graph $\Gamma$ is given by the group presentation 
\begin{equation}\label{equation:Artin group presentation}
A_\Gamma=\langle V \mid 
(s,t;m(e)) = (t,s; m(e))\text{ for each }e=\{s,t\}\in E\rangle,
\end{equation}
where $(s,t;m)$ is the alternating product of generators $s$ and $t$ of length $m$.
For example,
\begin{align*}
(s,t;1) &= s,& 
(s,t;2) &= st,&
(s,t;3) &= sts,& &\dots\\
(t,s;1) &= t,& 
(t,s;2) &= ts,&
(t,s;3) &= tst,& &\dots
\end{align*}

For each $e=\{s,t\}\in E$, let us denote the subgroup $G(e)$ generated by $\{s,t\}$.
Then the element $x_e=(s,t;m(e))\in G(e)$ preserves the set $\{s,t\}$ of generators under the conjugation. That is,
\[
x_e^{-1} \{s,t\} x_e = \{s,t\}
\]
and we call $x_e$ the \emph{quasi-center} of $G(e)$ or simply the \emph{quasi-center} for $e$.
On the other hand, the conjugation by $x_e$ preserves each generator $s$ and $t$ if and only if $m(e)$ is even. Therefore the element $z_e$ defined as 
\[
z_e = \begin{cases}
x_e^2 & m(e)\text{ is odd};\\
x_e & m(e)\text{ is even},
\end{cases}
\]
generates the center of $G(e)$.

According to the Crisp's result in \cite{Cr05}, there are four types of elementary isomorphisms which generate every isomorphism between \CLTTF Artin groups as follows:

\begin{enumerate}
\item A graph isomorphism $\alpha_{\#}:A_\Gamma\rightarrow A_\Delta$ defined as $\alpha_{\#}(v)=\alpha(v)$ for each $v\in V$, where $(\alpha:\Gamma\rightarrow\Delta)\in \mathfrak{S}_V$ is a graph isomorphism.
\item The global inversion $\iota:A_\Gamma \rightarrow A_\Gamma$ defined as $\iota(v)=v^{-1}.
$ for each $v\in V$.
\item An inner automorphisms $g_{\#}:A_\Gamma \rightarrow A_\Gamma$ for some $g\in A_\Gamma$ defined as $g_{\#}(v)=g^{-1}vg$ for each $v\in V$.
\item A partial conjugation $\varepsilon_\#:A_\Gamma\to A_\Delta$ for each decomposition $\varepsilon=(\Gamma_1, e,\Gamma_2)$ defined as
\[
\varepsilon_\#(v)=\begin{cases}
v & v\in V_1(\varepsilon);\\
x_e^{-1}v x_e & v\not\in V_1(\varepsilon),
\end{cases}
\]
where the graph $\Delta$ is obtained by edge-twists with respect to the decomposition $\varepsilon$.
\end{enumerate}

\begin{remark}
The above classification is slightly different from that described in \cite{Cr05}.
Indeed, all graphs are \emph{up to isomorphism} in \cite{Cr05} and so all graph isomorphisms above should be translated into graph automorphisms. Actually, this can be done by fixing a reference graph isomorphism $\Gamma\to\Delta$ for each $\Delta\cong\Gamma$.
\end{remark}

\begin{remark}
When $\Gamma$ has a leaf, then the leaf can be inverted separately, called the \emph{leaf inversion}. However, by Assumption~\ref{assumption:edge-separated}, there are no leaves in $\Gamma$.
\end{remark}

\begin{definition}[Rigidity of \CLTTF Artin groups]
A \CLTTF Artin group $A_\Gamma$ is said to be \emph{rigid} if it has a unique defining graph $\Gamma$ up to isomorphism.
\end{definition}
\begin{theorem}[\cite{BMcMN02}]
A \CLTTF Artin group $A_\Gamma$ is rigid if and only if so is $\Gamma$.
\end{theorem}

\subsection{The category \texorpdfstring{$\mathscr{A}$}{} of \CLTTF Artin groups}
Let us define the category $\mathscr{A}$ of \CLTTF Artin groups, whose objects and morphisms are as follows:
\begin{enumerate}
\item Objects are Artin group $A_\Gamma$ for all \CLTTF graphs $\Gamma$ given by the group presentation as described in \eqref{equation:Artin group presentation}.
\item Morphisms are compositions of partial conjugations and graph isomorphisms.
\end{enumerate}

Now let us consider the functor $\tilde{\mathscr{F}}:\tilde{\mathscr{G}}\to{\mathscr{A}}$ as follows:
for each \CLTTF graph $\Gamma$, we assign the Artin group $A_\Gamma$
\[
\tilde{\mathscr{F}}(\Gamma)=A_\Gamma
\]
given by the group presentation as mentioned at the beginning.
For each graph isomorphism $\alpha$ and edge-twist $\varepsilon$ in $\Hom_{\tilde{\mathscr{G}}}(\Gamma,\Delta)$, we assign a graph isomorphism and a partial conjugation
\[
\tilde{\mathscr{F}}(\alpha)=\alpha_\#\quad\text{ and }\quad
\tilde{\mathscr{F}}(\varepsilon)=\varepsilon_\#,
\]
respectively.
Then since the morphisms in $\tilde{\mathscr{G}}$ are freely generated by graph isomorphisms and edge-twists, the functor $\tilde{\mathscr{F}}$ is well-defined.

\begin{proposition}\label{proposition:existence}
The functor $\tilde{\mathscr{F}}:\tilde{\mathscr{G}}\to\mathscr{A}$ factors through the localization $\bar{\mathscr{G}}$ and the quotient category $\mathscr{G}$. Namely, there exist unique functors up to natural isomorphisms
\begin{align*}
\bar{\mathscr{F}}:\bar{\mathscr{G}}\to\mathscr{A}\quad\text{ and }\quad
\mathscr{F}:\mathscr{G}\to\mathscr{A},
\end{align*}
which fit into the following commutative diagram:
\[
\begin{tikzcd}[column sep=4pc]
\tilde{\mathscr{G}} \arrow[rd, bend left, "\tilde{\mathscr{F}}"]\arrow[d]\\
\bar{\mathscr{G}}=\tilde{\mathscr{G}}[\mathscr{E}^{-1}]\arrow[r, "\exists\bar{\mathscr{F}}"]
\arrow[d, ->>] & \mathscr{A}\\
\mathscr{G}=\bar{\mathscr{G}}/\sim\arrow[ru, bend right, "\exists\mathscr{F}"']
\end{tikzcd}
\]
\end{proposition}
\begin{proof}
The existence and the uniqueness of the functor
\[
\bar{\mathscr{F}}:\bar{\mathscr{G}}=\tilde{\mathscr{G}}[\mathscr{E}^{-1}]\to \mathscr{A}
\]
come from the universal property of the localized category since each edge-twist maps to a partial conjugation which is an isomorphism in $\mathscr{A}$.

Since $\mathscr{G}=\bar{\mathscr{G}}/\sim$ is the quotient category, by the universal property of the quotient category, it suffices to prove that for pull-back diagrams
\begin{align*}
&\begin{tikzcd}[ampersand replacement=\&]
\Gamma \arrow[r, "\alpha"] \arrow[d, "\varepsilon"'] \& \Delta \arrow[d, "\varepsilon'"]\\
\Gamma' \arrow[r, "\alpha"] \& \Delta',
\end{tikzcd}&
&\text{ and }&
\begin{tikzcd}[ampersand replacement=\&]
\Gamma \arrow[r, "\varepsilon_1"] \arrow[d, "\varepsilon_2"'] \& \Gamma_1 \arrow[d, "\varepsilon_2'"]\\
\Gamma_2 \arrow[r, "\varepsilon_1'"] \& \Gamma',
\end{tikzcd}
\end{align*}
the compositions of induced maps are identical in $\mathscr{A}$. Namely,
\begin{align*}
\varepsilon'_\# \alpha_\# &= \alpha_\#\varepsilon_\#,&
&\text{ and }&
{\varepsilon_2'}_\# {\varepsilon_1}_\#&={\varepsilon_1'}_\#{\varepsilon_2}_\#
\end{align*}
in $\Isom_{{\mathscr{A}}}(A_\Gamma,A_\Delta)$.

Let $\alpha, \alpha'$ and $\varepsilon=(e,C), \varepsilon'=(e',C')$ be graph isomorphisms and edge-twists that fit into a pull-back diagram.
We denote by $x_e\in A_\Gamma$ and $x_{e'}\in A_{\Delta}$ the quasi-centers for $e$ and $e'$, respectively. Let $V_2(\varepsilon)$ and  $V_2(\varepsilon')$ be the sets of vertices of $\Gamma_2(\varepsilon)$ and $\Gamma_2(\varepsilon')$, respectively.
By definition of the pull-back or push-forward, two subsets are canonically identified via $\alpha$.

Then the maps $\varepsilon'_\#\alpha_\#$ and $\varepsilon'_\#\alpha_\#$ are defined as follows: for each $v\in V_\Gamma$,
\begin{align*}
(\varepsilon'_\#\alpha_\#)(v)&=\begin{cases}
\alpha(v) & \alpha(v)\in V_1(\varepsilon');\\
x_{e'}^{-1}\alpha(v)x_{e'} & \alpha(v)\not\in V_1(\varepsilon').
\end{cases}\\
(\alpha_\#\varepsilon_\#)(v)&=\begin{cases}
\alpha(v) & v\in V_1(\varepsilon);\\
\alpha_\#(x_{e}^{-1}vx_{e}) & v\not\in V_1(\varepsilon).
\end{cases}
\end{align*}

We observe that since the restriction $\alpha|:V_1(\varepsilon)\to V_1(\varepsilon')$ is a bijection, we have
\[
\alpha(v)\not\in V_1(\varepsilon')\Longleftrightarrow v\not\in V_1(\varepsilon)
\]
and since $\alpha_\#(x_e)=x_{e'}$,
\begin{align*}
\alpha_\#(x_{e}^{-1}vx_{e})&= \alpha_\#(x_{e})^{-1}\alpha(v)\alpha_\#(x_e)= x_{e'}^{-1}\alpha(v)x_{e'}
\end{align*}
as desired.

Let $\varepsilon_1=(e_1,C_1),\varepsilon_1'=(e_1',C_1'), \varepsilon_2=(e_2,C_2)$ and $\varepsilon_2'=(e_2',C_2')$ be edge-twists that fit into a pull-back diagram.
Note that $\varepsilon_1=\varepsilon_1'$ and $\varepsilon_2=\varepsilon_2'$ in $\Ch_{\llbracket\Gamma\rrbracket}$ and so $V_i(\varepsilon_j)=V_i(\varepsilon_j')$ for each $i,j=1,2$. Moreover, we have identifications $x_{e_1}=x_{e_1'}$ and $x_{e_2}=x_{e_2'}$ as words of $V$.

As seen in Remark~\ref{remark:scopes are disjoint or nested}, $V\setminus V_1(\varepsilon_1)$ and $V\setminus V_1(\varepsilon_2)$ are either disjoint or nested.
If $(V\setminus V_1(\varepsilon_1))\cap (V\setminus V_1(\varepsilon_2))=\varnothing$, or equivalently, $V_1(\varepsilon_1)\cup V_1(\varepsilon_2)=V$, then for each $v\in V$,
\begin{align*}
{\varepsilon_2'}_\#{\varepsilon_1}_\#(v)&=
{\varepsilon_1'}_\#{\varepsilon_2}_\#(v)
=\begin{cases}
v & v\in V_1(\varepsilon_1)\cap V_1(\varepsilon_2);\\
x_{e_1}^{-1}vx_{e_1} & v\in V_1(\varepsilon_2)\setminus V_1(\varepsilon_1);\\
x_{e_2}^{-1}vx_{e_2} & v\in V_1(\varepsilon_1)\setminus V_1(\varepsilon_2).
\end{cases}
\end{align*}

On the other hand, if $V\setminus V_1(\varepsilon_1)$ and $V\setminus V_1(\varepsilon_2)$ are nested, then we may assume that $(V\setminus V_1(\varepsilon_2))\subset (V\setminus V_1(\varepsilon_1))$, or equivalently, $V_1(\varepsilon_1)\subset V_1(\varepsilon_2)$.
Hence, ${\varepsilon_2'}_\#$ preserves vertices of $e_1$ and so $x_{e_1}$ as well.
Therefore for each $v\in V$,
\begin{align*}
({\varepsilon_2'}_\#{\varepsilon_1}_\#)(v)
&=\begin{cases}
v & v\in V_1({\varepsilon_1});\\
{\varepsilon_2'}_\#\left(x_{e_1}^{-1}vx_{e_1}\right) & v\not\in V_1({\varepsilon_1}),
\end{cases}\\
&=\begin{cases}
v & v\in V_1({\varepsilon_1});\\
x_{e_1}^{-1}vx_{e_1} & v\in V_1(\varepsilon_2)\setminus V_1(\varepsilon_1);\\
x_{e_1}^{-1}{\varepsilon_2'}_\#(v)x_{\varepsilon_1} & v\not\in V_1(\varepsilon_2),
\end{cases}\\
&=\begin{cases}
v & v\in V_1(\varepsilon_1);\\
x_{e_1}^{-1}vx_{e_1} & v\in V_1(\varepsilon_2)\setminus V_1(\varepsilon_1);\\
x_{e_1}^{-1}x_{e_2}^{-1}vx_{e_2}x_{e_1} & v\not\in V_1(\varepsilon_2).
\end{cases}
\end{align*}

On the other hand, we also have
\begin{align*}
({\varepsilon_1'}_\#{\varepsilon_2}_\#)(v)
&=\begin{cases}
v & v\in V_1(\varepsilon_2);\\
{\varepsilon_1'}_\#\left(x_{e_2}^{-1}vx_{e_2}\right) & v\not\in V_1(\varepsilon_2),
\end{cases}\\
&=\begin{cases}
v & v\in V_1(\varepsilon_1);\\
{\varepsilon_1'}_\#(v) & v\in V_1(\varepsilon_2)\setminus V_1(\varepsilon_1);\\
{\varepsilon_1'}_\#\left(x_{e_2}^{-1}vx_{e_2}\right) & v\not\in V_1(\varepsilon_2),
\end{cases}\\
&=\begin{cases}
v & v\in V_1(\varepsilon_1);\\
x_{e_1}^{-1}vx_{e_1} & v\in V_1(\varepsilon_2)\setminus V_1(\varepsilon_1);\\
\left(x_{e_1}^{-1}x_{e_2}^{-1}x_{e_1}\right)
\left(x_{e_1}^{-1}vx_{e_1}\right)
\left(x_{e_1}^{-1}x_{e_2}x_{e_1}\right) & v\not\in V_1(\varepsilon_2).
\end{cases}
\end{align*}
Therefore,
\[
({\varepsilon_1'}_\#{\varepsilon_2}_\#)(v)=\begin{cases}
v & v\in V_1(\varepsilon_1);\\
x_{e_1}^{-1}vx_{e_1} & v\in V_1(\varepsilon_2)\setminus V_1(\varepsilon_1);\\
x_{e_1}^{-1}x_{e_2}^{-1}vx_{e_2}x_{e_1} & v\not\in V_1(\varepsilon_2),
\end{cases}
\]
which completes the proof.
\end{proof}

\begin{corollary}\label{corollary:fullness}
The functor $\mathscr{F}:\mathscr{G}\to\mathscr{A}$ is full.
\end{corollary}
\begin{proof}
Every graph isomorphsm and partial conjugation comes essentially from a graph isomorphism and an edge-twist or its formal inverse, which is again a morphism in $\mathscr{G}$ by definition. Hence the induced functor $\bar{\mathscr{F}}:\bar{\mathscr{G}}\to\mathscr{A}$ is full, and so is the functor $\mathscr{F}$ since $\mathscr{G}$ is the quotient category of $\bar{\mathscr{G}}$.
\end{proof}

Recall Proposition~\ref{proposition:graph isomorphism and edge-twists}, which claims that the only identity map can be both a graph isomorphism and an edge-twist in $\mathscr{G}$.
Indeed, we insist a bit stronger statement below which proves Proposition~\ref{proposition:graph isomorphism and edge-twists} as a direct consequence.

\begin{proposition}\label{proposition:graph isomorphism and edge-twists in A}
Let $\alpha,\mathcal{E} \in \Hom_{\mathscr{G}}(\Gamma,\Delta)$ such that $\alpha$ is a graph isomorphism and $\mathcal{E}$ is a composition of edge-twists and their inverses.
If $\mathscr{F}(\alpha)=\mathscr{F}(\mathcal{E})$ in $\Hom_{\mathscr{A}}(A_\Gamma,A_\Delta)$. Then $\Gamma=\Delta$ and both $\alpha$ and $\mathcal{E}$ are the identities.
\end{proposition}
\begin{proof}
Let us assume that $\mathcal{E}=\mathcal{E}(\eta)$ for some $\eta:E^\out(\Ch_\Gamma)\to \mathbb{Z}$.
Unless $\alpha$ is the identity, the induced automorphism $\mathscr{F}(\alpha)$ can not be the identity. Hence, it suffices to prove that $\mathcal{E}$ is the identity.

We use the induction on $\operatorname{Diam}(\Ch_\Gamma)$, which is always even.
If $\operatorname{Diam}(\Ch_\Gamma)=0$, then there are no edges in $E^\out(\Ch_\Gamma)$ and therefore $\mathcal{E}$ must be trivial.
If $\operatorname{Diam}(\Ch_\Gamma)=2$, then the center $*_\Gamma$ must be a separating edge $e$ since the leaves in $\Ch_\Gamma$ are chunks.
As above, all edge-twists induce partial conjugations that fix both $s$ and $t$. 
Suppose that $\eta(\varepsilon)=a\neq 0$ for some $\varepsilon=(e, C)\in E^\out(\Ch_\Gamma)$.
Since $C$ is triangle-free and edge-separated, we can pick an edge $f=\{v,w\}\in E(C)$ such that $e\cap f=\varnothing$.
Then we have
\[
\mathscr{F}(\alpha)(\{v,w\}) = \{\alpha(v),\alpha(w)\}= \{x_e^{-a} v x_e^a, x_e^{-a} w x_e^a\}=\mathscr{F}(\mathcal{E})(\{v,w\}),
\]
which generate subgroups $G(\alpha(f))$ and $x_e^{-a}G(f)x_e$.
Unless $f=\alpha(f)$, these two subgroups can not be the same since $G(f)$ and $G(\alpha(f))$ are not conjugate to each other.
Hence we have $f=\alpha(f)$ and so $x_e^a$ is in the normalizer of $G(f)$, which is a power of $x_f$.
This is a contradiction since $G(e)\cap G(f)$ is trivial and so $a$ must be $0$.

Suppose that the assertion holds for every $\Gamma$ with $\operatorname{Diam}(\Ch_\Gamma)\le 2N$. We assume that \[\operatorname{Diam}(\Ch_\Gamma)=2N+4.\]

Let $\Gamma'\subset \Gamma$ and $\Delta'\subset\Delta$ be the subgraphs which are unions of chunks in $\Gamma$ and $\Delta$ within a distance $(N+1)$ (or $N$\footnote{Since the vertices at the distance $(N+1)$ from $*_\Gamma$ correspond to separating edges which are already contained in chunks at the distance $N$ from $*_\Gamma$.}) from the center $*_\Gamma$ and $*_\Delta$, respectively.
Then since both $\alpha$ and $\mathcal{E}$ induce isomorphisms on $(\Ch_\Gamma,*_\Gamma)\to (\Ch_\Delta,*_\Delta)$, their restrictions
\[
\alpha|:\Gamma'\to\Delta'\quad\text{ and }\quad
\mathcal{E}|:\Gamma'\to\Delta'
\]
are well-defined so that $\alpha|$ is a graph isomorphism and $\mathcal{E}|$ is a composition of edge-twists again.
Furthermore, they induce the same maps so that $\mathscr{F}(\alpha|) = \mathscr{F}(\mathcal{E})$ in $\Hom_{\mathscr{A}}(A_{\Gamma'},A_{\Delta'})$ and therefore both $\alpha|$ and $\mathcal{E}|$ must be the identity by the induction hypothesis.

Hence the only possibility for $\mathcal{E}$ is a composition of edge-twists involving chunks which are farthest from the center $*_\Gamma$.
Suppose that there exists an edge $\varepsilon=(e,C)\in E^\out(\Ch_\Gamma)$ with $\eta(\varepsilon)=a\neq 0$ involving a farthest chunk $C$.
Then the exactly same argument as above yields a contradiction and therefore $\mathcal{E}$ must be trivial.
\end{proof}

\begin{proof}[Proof of Proposition~\ref{proposition:graph isomorphism and edge-twists}]
Let $\alpha,\mathcal{E}$ be two morphisms satisfying the hypothesis.
If $\alpha=\mathcal{E}$, then $\mathscr{F}(\alpha)=\mathscr{F}(\mathcal{E})$ and so both $\alpha$ and $\mathcal{E}$ are the identities by Proposition~\ref{proposition:graph isomorphism and edge-twists in A}.
\end{proof}

Moreover, Proposition~\ref{proposition:graph isomorphism and edge-twists in A} also implies the faithfulness of the functor $\mathscr{F}$ as follows:
\begin{corollary}\label{corollary:faithfulness}
Let $f,g\in\Hom_{\mathscr{G}}(\Gamma,\Delta)$.
If $\mathscr{F}(f)=\mathscr{F}(g)\in\Hom_{\mathscr{A}}(A_\Gamma,A_\Delta)$, then $f=g$.

In other words, the functor $\mathscr{F}:\mathscr{G}\to\mathscr{A}$ is faithful.
\end{corollary}
\begin{proof}
Let $f, g\in\Hom_{\mathscr{G}}(\Gamma,\Delta)$.
Then as observed in \eqref{equation:normal form}, $f$ and $g$ can be expressed as compositions
\begin{align*}
f&=\mathcal{E}\alpha,&
\mathcal{E}&=\mathcal{E}(\eta),&%\varepsilon_n^{\eta_n}\cdots\varepsilon_1^{\eta_1},\\
g&=\mathcal{E}'\alpha',&
\mathcal{E}'&=\mathcal{E}(\eta')%\varepsilon_m'^{\eta'_m}\cdots\varepsilon_1'^{\eta_1'}.
\end{align*}
for some $\eta,\eta':E^\out(\Ch_{\llbracket\Gamma\rrbracket})\to\mathbb{Z}$.
Then since $\mathscr{F}(f)=\mathscr{F}(g)$, by pre-composition of $\mathscr{F}(\alpha)^{-1}$ and post-composition of $\mathscr{F}(\mathcal{E}')$, we have
\begin{align*}
\mathscr{F}(\mathcal{E}'^{-1} f\alpha^{-1})
&=\mathscr{F}(\mathcal{E}')^{-1} \mathscr{F}(f)\mathscr{F}(\alpha)^{-1}
=\mathscr{F}(\mathcal{E}')^{-1} \mathscr{F}(g)\mathscr{F}(\alpha)^{-1}
=\mathscr{F}(\mathcal{E}'^{-1} g \alpha^{-1}).
\end{align*}
However, the left hand side is the induced map of edge-twists
$\mathscr{F}(\mathcal{E}'^{-1} f\alpha^{-1})=\mathscr{F}(\mathcal{E}'^{-1}\mathcal{E})$ while
the right hand side is the induced map of graph isomorphisms
$\mathscr{F}(\mathcal{E}'^{-1} g \alpha^{-1})= \mathscr{F}(\alpha'\alpha^{-1})$.
Then by Proposition~\ref{proposition:graph isomorphism and edge-twists in A}, we must have
\[
\mathcal{E}'^{-1}\mathcal{E} = \alpha'\alpha^{-1}=\operatorname{Id}\in\Hom_{\mathscr{G}}(\Gamma,\Gamma),
\]
and therefore $\alpha'=\alpha$ and $\mathcal{E}'=\mathcal{E}$, which implies that $f=g$.
\end{proof}

In summary, we have the following theorem.
\begin{theorem}\label{theorem:equivalence of categories}
The induced functor $\mathscr{F}:\mathscr{G}\to\mathscr{A}$ is an equivalence of categories.
\end{theorem}
\begin{proof}
In order to show that $\mathscr{F}$ is an equivalence, we will show (i) the essential surjectivity, and (ii) the fully-faithfulness.

\noindent (i) By definition, every object in $\mathscr{A}$ is an Artin group presentation for a \CLTTF graph and every morphism in ${\mathscr{A}}$ is an isomorphism. Hence the essential surjectivity is obvious.

\noindent (ii) The fully-faithfulness comes from Corollaries~\ref{corollary:fullness} and \ref{corollary:faithfulness}, and we are done.
\end{proof}

Before closing this section, we will prove the equivalence of categories between $\mathscr{G}$ and the subcategory of the groupoid defined in \cite{Cr05}.
Let $\mathcal{I}so$ be the category of edge-separated $\CLTTF$ graphs \emph{up to isomorphism} whose morphisms are the set of all group isomorphisms $A_\Gamma\rightarrow A_\Gamma'$.

Then by \cite[Theorem~1]{Cr05}, the morphisms in $\mathcal{I}so$ are generated by graph automorphisms, leaf and global inversions, inner automorphisms, and partial conjugations.

We define the subcategory $\mathcal{I}so_0$ of $\mathcal{I}so$ whose morphisms are generated by graph automorphisms and partial conjugations.
Then we will show an equivalence between two categories $\mathscr{G}$ and $\mathcal{I}so$.

For each $\Gamma\in\mathscr{G}$.
Let us denote $[\Gamma]$ the graph isomorphism class of $\Gamma$ in $\mathscr{G}$.
\[
[\Gamma]\coloneqq\{
\Delta\in\mathscr{G}\mid \Delta\cong\Gamma\}
\]
Then as a set, $[\mathscr{G}]$ is defined to be the set of isomorphism classes, or equivalently, the set of \CLTTF graphs up to isomorphism.

For each isomorphism class $[\Gamma]$, we fix a representative $\Gamma_0\in[\Gamma]$.
Moreover, for each $\Delta\in[\Gamma]$, we also fix a graph isomorphism
$\alpha_\Delta:\Gamma_0\to \Delta$.
Obviously, $\alpha_\Delta=\operatorname{Id}$ if and only if $\Delta$ represents its isomorphism class.

Now we define a functor $[\cdot]:\mathscr{G}\to\mathcal{I}so_0$ as follows:
For each \CLTTF graph $\Gamma$,
\[
[\cdot]:\Gamma\mapsto [\Gamma],
\]
where $[\Gamma]$ is the graph isomorphism class of $\Gamma$.

For each graph isomorphism $\alpha:\Gamma\to\Delta$, we have a graph automorphism
$\alpha_{\Delta}^{-1}\alpha\alpha_\Gamma:\Gamma_0\to\Gamma_0$, which induces an isomorphism
\[
[\alpha]\coloneqq(\alpha_{\Delta}^{-1}\alpha\alpha_\Gamma)_\#:A_{\Gamma_0}\to A_{\Gamma_0}.
\]
Here $\Gamma_0$ is the chosen representive of $[\Gamma]=[\Delta]$.

For each edge-twist $\varepsilon:\Gamma\to\Delta$, we have a composition
$\alpha_\Delta^{-1}\varepsilon\alpha_\Gamma
%=\alpha_\Delta^{-1}\alpha_\Gamma \alpha_\Gamma^*(\varepsilon)
:\Gamma_0\to \Delta_0$,
which induces an isomorphism
\[
[\varepsilon]\coloneqq
(\alpha_\Delta)_\#^{-1}\varepsilon_\#(\alpha_\Gamma)_\#
%(\alpha_\Delta^{-1}\alpha_\Gamma)_\# \alpha_\Gamma^*(\varepsilon)_\#
:A_{\Gamma_0}\to A_{\Delta_0}.
\]
Here $\Gamma_0$ and $\Delta_0$ are the chosen representative of $[\Gamma]$ and $[\Delta]$, respectively.

\begin{theorem}
The functor $[\cdot]:\mathscr{G}\to\mathcal{I}so_0$ is an equivalence of categories.
\end{theorem}
\begin{proof}
By definition of $\mathcal{I}so_0$, Theorem~1 in \cite{Cr05} and Theorem~\ref{theorem:equivalence of categories}, the functor $[\cdot]$ is well-defined, surjective, and full.
The faithfulness follows obviously from Proposition~\ref{proposition:graph isomorphism and edge-twists in A} as well.
\end{proof}

\section{The automorphism group \texorpdfstring{$\Aut(A_\Gamma)$}{}}
\label{section:automorphism groups}
In this section, we will provide finite presentations for both $\Aut(A_\Gamma)$ and $\Out(A_\Gamma)$.
To this end, we first analyze the automorphism group $\Aut(A_\Gamma)$, which is generated by $\Aut_{\mathscr{A}}(A_\Gamma)$, the inner automorphism group $\Inn(A_\Gamma)$ and the global inversion $\iota:A_\Gamma\to A_\Gamma$, and define a pairing between elements in $\Iso(\Gamma)$ called a \emph{twisted intersection product}.

\subsection{Preliminaries}
\subsubsection{Positive automorphisms}
For each $f\in \Aut(A_\Gamma)$, let us consider the induced map $H_1(f)$ on the abelianization $H_1(A_\Gamma)$
\[
H_1(f):H_1(A_\Gamma)\to H_1(A_\Gamma),
\]
which is either a permutation of vertices or a composition of the global inversion and a permutation of vertices.
Namely, there exists $\sgn(f)=\pm1$ such that for each $v\in V$, 
\[
H_1(f)([v])=\sgn(f)[w]\in H_1(A_\Gamma)
\]
for some $w\in V_\Gamma$. 
We say that $f$ is \emph{positive} if $\sgn(f)=1$.
Then $\sgn$ defines a group homomorphism
\[
\sgn:\Aut(A_\Gamma)\to \mathbb{Z}_2,
\]
whose kernel is the subgroup of positive automorphisms and denoted by $\Aut_+(A_\Gamma)\coloneqq\ker(\sgn)$. Then
\[
\Inn(A_\Gamma)\lhd\Aut_+(A_\Gamma)\lhd \Aut(A_\Gamma)
\]
and so we define the group $\Out_+(A_\Gamma)$ as the quotient
\[
\Out_+(A_\Gamma)\coloneqq\Aut_+(A_\Gamma)\big/\Inn(A_\Gamma).
\]

Obviously, the global inversion $\left(\iota:A_\Gamma\to A_\Gamma\right)\in \Aut(A_\Gamma)$ is not contained in $\Aut_+(A_\Gamma)$ and acts on inner automorphisms, partial conjugations and graph isomorphisms by conjugation.
More precisely, for each isomorphism $\alpha_\#, \varepsilon_\#:A_\Gamma\to A_\Delta$ coming from a graph isomorphism $\alpha$ and an edge-twist $\varepsilon$, or inner automorphism $g_\#:A_\Gamma\to A_\Gamma$, we have
\begin{align*}
\alpha_\#\iota &= \iota\alpha_\#,&
\varepsilon_\#^{-1}\iota &= \iota\varepsilon_\#,&
\bar{g}_\#^{-1}\iota &= \iota g_\#,
\end{align*}
where $\bar{g} = v_k\cdots v_1$ is the reverse of $g=v_1\cdots v_k$ with $v_i\in V$.
Therefore the following lemmas are immediate consequences.

\begin{lemma}\label{lemma:big commutative diagram}
There is a commutative diagram with exact rows and columns
\[
\begin{tikzcd}
& 1\arrow[d] & 1\arrow[d]\\
& \Inn(A_\Gamma)\arrow[d]\arrow[r,equal] & \Inn(A_\Gamma)\arrow[d] &  \\
1\arrow[r]& \Aut_+(A_\Gamma)\arrow[d]\arrow[r]& \Aut(A_\Gamma)\arrow[d]\arrow[r, "\sgn"]& \mathbb{Z}_2\arrow[d,equal]\arrow[r]& 1\\
1\arrow[r]& \Out_+(A_\Gamma)\arrow[d]\arrow[r]& \Out(A_\Gamma)\arrow[d]\arrow[r, "\sgn"]& \mathbb{Z}_2\arrow[r]& 1\\
& 1  & 1 
\end{tikzcd}
\]
where the group $\mathbb{Z}_2\cong\langle \iota\mid \iota^2\rangle$ is generated by the global inversion $\iota$.
Therefore the rows split and
\begin{align*}
\Aut(A_\Gamma)&\cong \Aut_+(A_\Gamma) \rtimes \mathbb{Z}_2,&
\Out(A_\Gamma)&\cong \Out_+(A_\Gamma) \rtimes \mathbb{Z}_2.
\end{align*}
\end{lemma}

\begin{lemma}\label{lemma:positive automorphisms}
An automorphism $f$ is in $\Aut_+(A_\Gamma)$ if and only if it is a composition of inner automorphisms, partial conjugations and graph isomorphisms.
\end{lemma}

In particular, we have the subgroup $\Aut_{\mathscr{A}}(A_\Gamma)\subset \Aut_+(A_\Gamma)$ consisting of compositions of partial conjugations and graph isomorphisms.

\subsubsection{The special automorphism}
Now recall the special automorphism $\Phi\in\Aut_{\mathscr{G}}(\Gamma)$ defined in Example~\ref{example:special automorphism}.
We claim that the induced map $\Phi_\#=\mathscr{F}(\Phi)\in \Aut_{\mathscr{A}}(A_\Gamma)$ is either the identity or the inner automorphism as follows:
\[
\Phi_\# = \begin{cases}
\operatorname{Id} & \text{ if }*_\Gamma\text{ is a chunk};\\
{x_e}_\# & \text{ if }*_\Gamma\text{ is a separating edge $e$}.
\end{cases}
\]

If $*_\Gamma$ is a chunk, then $\Phi\in\Aut_{\mathscr{G}}(\Gamma)$ is the identity and so is $\Phi_\#$.
Otherwise, assume that $*_\Gamma$ is a separating edge $e=\{s,t\}$.
By definition of $\Phi$, either
\[
\Phi=\mathcal{E}_{*_\Gamma}\quad\text{ or }\quad
\Phi=\mathcal{E}_{*_\Gamma}\alpha_{*_\Gamma}
\]
according to the parity of the label $m(e)$.
Namely, if $m(e)$ is even, then
\[
\Phi_\#(v)=\begin{cases}
x_e^{-1}vx_e & v\not\in\{s,t\};\\
s & v=s;\\
t & v=t,
\end{cases}
\]
which is the inner automorphism ${x_e}_\#$ since $s=x_e^{-1}sx_e$ and $t=x_e^{-1}tx_e$.

If $m(e)$ is odd, then since $\Phi$ is a composition with a graph isomorphism $\alpha$ that interchanges $s$ and $t$, we have
\[
\Phi_\#(v)=\begin{cases}
x_e^{-1}vx_e & v\not\in\{s,t\};\\
t & v=s;\\
s & v=t
\end{cases}
\]
and therefore $\Phi_\#={x_e}_\#$ again since $s=x_e^{-1}tx_e$ and $t=x_e^{-1}sx_e$.

For a sake of convenience, the subgroup generated by $\Phi_\#$ will be denoted by $Z_\Gamma$.
Unless $\Phi_\#$ is trivial, it is of infinite order since the Artin group $A_\Gamma$ is centerless.
\[
Z_\Gamma=\langle\Phi_\#\rangle\cong\begin{cases}
1 & *_\Gamma\text{ is a chunk};\\
\mathbb{Z} & *_\Gamma\text{ is a separating edge}.
\end{cases}
\]

\begin{proposition}\label{proposition:inner and edge-twist}
For each $A_\Gamma\in\mathscr{A}$, 
\[
\Inn(A_\Gamma)\cap \Aut_{\mathcal{A}}(A_\Gamma) = Z_\Gamma.
\]
\end{proposition}
\begin{proof}
By the above discussion, the subgroup $Z_\Gamma\subset \Inn(A_\Gamma)\cap \Aut_{\mathscr{A}}(A_\Gamma)$.

Suppose that $\phi\in \Inn(A_\Gamma)\cap \Aut_{\mathscr{A}}(A_\Gamma)$. Then by Theorems~\ref{theorem:normal form in G} and \ref{theorem:equivalence of categories}, 
\[
\phi=g_\#=\mathcal{E}_\#\alpha_\#
\]
for some $g\in A_\Gamma$, graph isomorphism $\alpha:\Gamma\to\Delta$ and composition of edge-twists $\mathcal{E}:\Delta\to\Gamma$.

Let $*_\Gamma$ be the central vertex of the chunk tree $\Ch_\Gamma$.
If $*_\Gamma$ is a chunk $C$, then $D\coloneqq\alpha(C)\subset\Delta$ is isomorphic to $C$.
Moreover, for any edge-twist $\varepsilon\in E^{\out}(\Ch_\Delta)$, the set $V_2({\varepsilon})$ contains no vertices in $D$.
Therefore for each $v\in V(C)$,
\[
\phi(v)=g^{-1}vg=\mathcal{E}_\#(\alpha(v))=\alpha(v)\in V(D)\subset V,
\]
which implies that the full subgraph defined by $V(D)$ in $\Gamma$ is isomorphic to $C$.
In other words, the graph isomorphism $\alpha$ on $C$ is a graph automorphism and therefore if we restrict $\phi$ to the Artin group $A_C$, then we have
\[
\phi|_{A_C} = (\alpha|_C)_*:A_C\to A_C.
\]

Since the graph automorphism $\alpha|_C$ is of finite order, we have
\[
\operatorname{Id} = (\alpha|_C)_\#^N = (g_\#|_{A_C})^N = (g^N)_\#|_{A_C}
\]
for some $N\ge 1$. This means that $g^N$ is contained in the centralizer $C_{A_\Gamma}(A_C)$ of $A_C$ in $A_\Gamma$, which is trivial.
Since $A_\Gamma$ is torsion-free, $g$ must be trivial.

If $*_\Gamma$ is a separating edge $e=\{s,t\}$, then $\alpha(\{s,t\})=\{s,t\}$ is the central separating edge in $\Ch_\Delta$. 
As above, both vertices $s$ and $t$ are not contained in $V_2(\varepsilon)$ of any edge-twist $\varepsilon\in E^\out(\Ch_\Delta)$.
Therefore, for $v\in \{s,t\}$, we have
\[
\phi(v) =g^{-1}vg= \mathcal{E}_\#(\alpha(v)) = \alpha(v)\in \{s,t\}\subset V,
\]
and either
\[
\begin{cases}
\alpha(s) = s;\\
\alpha(t) = t,
\end{cases}\quad\text{ or }\quad
\begin{cases}
\alpha(s) = t;\\
\alpha(t) = s.
\end{cases}
\]

Since $g$ normalizes $G(e)$, it is a power of $x_e$ as proved in \cite{Godelle03} and we are done.
\end{proof}

\begin{corollary}
The group $Z_\Gamma$ is contained in the center of $\Aut_{\mathscr{A}}(A_\Gamma)$.
\end{corollary}
\begin{proof}
Since there is nothing to prove when $*_\Gamma$ is a chunk, we assume that $*_\Gamma$ is a separating edge $e=\{s,t\}$.

For each $\phi=\mathcal{E}_\#\alpha_\#\in \Aut_{\mathscr{A}}(A_\Gamma)$, it suffices to prove that $\Phi_\#\phi = \phi\Phi_\#$.
As seen in the proof of Proposition~\ref{proposition:inner and edge-twist}, for each $v\in\{s,t\}$, we have $\phi(v)=\alpha(v)\in V$, where $\alpha$ on $e=\{s,t\}$ is a graph automorphism.
Therefore $\phi(x_e)=\alpha_\#(x_e)=x_e$, and so for any $v\in V$,
\begin{align*}
(\Phi_\#\phi)(v) = x_e^{-1}\phi(v)x_e = \phi(x_e^{-1}vx_e) = (\phi\Phi_\#)(v),
\end{align*}
which completes the proof.
\end{proof}

In particular, the group $Z_\Gamma$ is a normal subgroup of $\Aut_{\mathscr{A}}(A_\Gamma)$.
Hence, there is a commutative diagram with exact rows as follows:
\begin{equation}\label{equation:commutative diagram}
\begin{tikzcd}
1\arrow[r] & Z_\Gamma \arrow[r]\arrow[d,hookrightarrow] & \Aut_{\mathscr{A}}(A_\Gamma)\arrow[r]\arrow[d,hookrightarrow] & \Aut_{\mathscr{A}}(A_\Gamma)\big/Z_\Gamma\arrow[r]\arrow[d,"\bar\Psi"] & 1\\
1\arrow[r] & \Inn(A_\Gamma) \arrow[r] & \Aut_+(A_\Gamma)\arrow[r] & \Out_+(A_\Gamma)\arrow[r] & 1
\end{tikzcd}
\end{equation}
The right vertical arrow $\bar\Psi$ is the induced map of the inclusion $\Aut_{\mathscr{A}}(A_\Gamma)\to \Aut_+(A_\Gamma)$.

\begin{theorem}\label{theorem:outer automorphism group}
There is an isomorphism
\[
\bar\Psi:\Aut_{\mathscr{A}}(A_\Gamma)\big/ Z_\Gamma\cong \Out_+(A_\Gamma).
\]
Therefore, by Lemma~\ref{lemma:big commutative diagram},
\[
\Out(A_\Gamma)\cong \left(\Aut_{\mathscr{A}}(A_\Gamma)\big/Z_\Gamma\right)\rtimes \mathbb{Z}_2.
\]
\end{theorem}
\begin{proof}
By Lemma~\ref{lemma:positive automorphisms}, any automorphism in $\Aut_+(A_\Gamma)$ is a composition of inner automorphisms, partial conjugations and graph isomorphisms.
Therefore, for each $[\phi]\in\Out_+(A_\Gamma)$, we have a representative $\phi\in\Aut_+(A_\Gamma)$ which is a composition of partial conjugations and graph isomorphisms.
However, by definition of $\Aut_{\mathscr{A}}(A_\Gamma)$, the map $\phi$ is also contained in $\Aut_{\mathscr{A}}(A_\Gamma)$ and therefore $\bar\Psi$ is surjective.

Suppose that $\ker(\bar\Psi)$ is nontrivial. Then there exists $\phi\in \Inn(A_\Gamma)\cap \Aut_{\mathscr{A}}(A_\Gamma)=Z_\Gamma$, which is trivial in $\Aut_{\mathscr{A}}(A_\Gamma)\big/Z_\Gamma$ as desired.
\end{proof}

Unfortunately, the row exact sequences in \eqref{equation:commutative diagram} do not split in general.
However, when $*_\Gamma$ is a chunk, then the group $Z_\Gamma$ is trivial and therefore 
\[
\Out_+(A_\Gamma)\cong\Aut_{\mathscr{A}}(A_\Gamma)\subset\Aut_+(A_\Gamma).
\]

\begin{corollary}\label{corollary:automorphism}
If $*_\Gamma$ is a chunk, then 
\begin{align*}
\Aut(A_\Gamma)&\cong \Inn(A_\Gamma)\rtimes \Out(A_\Gamma), &
%\left(\Aut_{\mathscr{A}}(A_\Gamma)\rtimes \mathbb{Z}_2\right),\\
\Out(A_\Gamma)&\cong \Aut_{\mathscr{A}}(A_\Gamma)\rtimes \mathbb{Z}_2.
\end{align*}
\end{corollary}

In particular, if $\Gamma$ is furthermore discretely rigid, then we have the following corollary.
\begin{corollary}
Let $\Gamma$ be a discretely rigid \CLTTF graph such that $*_\Gamma$ is a chunk. Then 
\begin{align*}
\Aut(A_\Gamma)&\cong \Inn(A_\Gamma)\rtimes \left(\left(\mathbb{Z}^{\#(E^\out(\Ch_\Gamma))}\rtimes \Aut(\Gamma)\right)\rtimes \mathbb{Z}_2\right),\\
\Out(A_\Gamma)&\cong \left(\mathbb{Z}^{\#(E^\out(\Ch_\Gamma))}\rtimes \Aut(\Gamma)\right)\rtimes \mathbb{Z}_2.
\end{align*}
\end{corollary}
\begin{proof}
This is a combination of Corollaries~\ref{corollary:rigid automorphism in G}, \ref{corollary:automorphism} and Theorem~\ref{theorem:equivalence of categories}.
\end{proof}

\subsubsection{Twisted intersection product}
Let us define the \emph{twisted intersection product} $(\alpha,\beta)=\mathcal{E}(\eta_{\alpha,\beta})$ as a composition of edge-twists in $\mathscr{G}$,
where
\[
\eta_{\alpha,\beta}(\varepsilon)=\begin{cases}
\bar\eta_\alpha(\varepsilon)\cdot\bar\eta_\beta(\varepsilon) & \varepsilon\in E^\out_\odd(\Ch_{\llbracket\Gamma\rrbracket});\\
0 & \text{otherwise}.
\end{cases}
\]

Indeed, this product captures the carry-over of the sum of $\bar\eta_\alpha$ and $\alpha_*(\bar\eta_\beta)$ in binary arithmetic that we may lose since each edge-twist is involutive in $\Iso(\Gamma)$.

Note that the twisted intersection product is not necessarily commutative.
Moreover, it satisfies the following properties.
\begin{lemma}\label{lemma:twisted product of triple}
For $\alpha_1, \alpha_2, \alpha_3\in \Iso(\Gamma)$, we have
\[
(\alpha_1, \alpha_2\alpha_3)\cdot(\alpha_1)_*(\alpha_2, \alpha_3) = (\alpha_1\alpha_2, \alpha_3)\cdot(\alpha_1,\alpha_2).
\]
\end{lemma}
\begin{proof}
For each $i=1,2,3$, let $\alpha_i:\Gamma\to\Delta_i$ and $\bar{\mathcal{E}}_i=\bar{\mathcal{E}}(\bar\eta_i):\Delta_i\to \Gamma$ be the unique composition of edge-twists.
Then for each $\varepsilon\in E^\out_\odd(\Ch_{\llbracket\Gamma\rrbracket})$, both count the carry-over in the sum
\begin{align*}
&\mathrel{\hphantom{=}}\bar\eta_1(\varepsilon) + (\alpha_1)_*(\bar\eta_2(\varepsilon)) +
(\alpha_1\alpha_2)_*(\bar\eta_3(\varepsilon))\\
&=\bar\eta_1(\varepsilon)\cdot (\alpha_1)_*(\bar\eta_2(\varepsilon)) +
\bar\eta_1(\varepsilon)\cdot (\alpha_1\alpha_2)_*(\bar\eta_3(\varepsilon)) +
(\alpha_1)_*(\bar\eta_2(\varepsilon))\cdot (\alpha_1\alpha_2)_*(\bar\eta_3(\varepsilon))\quad (\mod{2}).\qedhere
\end{align*}
\end{proof}

\begin{lemma}
Let $\alpha_0\in \Aut(\Gamma)$ and $\alpha_1,\alpha_2\in\Iso(\Gamma)$. Then the following holds:
\begin{enumerate}
\item $( \alpha_0, \alpha_1) = ( \alpha_1, \alpha_0) = \operatorname{Id}$,
\item $( \alpha_1, \alpha_2\alpha_0) = ( \alpha_1, \alpha_2)$,
\item $( \alpha_0\alpha_1, \alpha_2) = (\alpha_0)_*(\alpha_1, \alpha_2)$,
\item $( \alpha_1\alpha_0, \alpha_2) = (\alpha_1, \alpha_0\alpha_2)$.
\end{enumerate}
\end{lemma}
\begin{proof}
\noindent (1) This is obvious since $\alpha_0:\Gamma\to\Gamma$.

\noindent (2) -- (4) These are immediate corollary of (1) and Lemma~\ref{lemma:twisted product of triple}.
\end{proof}

\begin{example}
Recall the generators $\alpha_0, \dots, \alpha_4$ for $\Iso(\Gamma)$ as described in Example~\ref{example:generators of Iso}.

Let $\alpha$ and $\beta$ be elements in $\Iso(\Gamma)$. By relation in \eqref{equation:relation of Iso}, there exist two sequences $(i_0,\dots, i_4)$ and $(j_0,\dots, j_4)$ in $\{0,1\}$ such that
\begin{align*}
\alpha &= \alpha_1^{i_1}\alpha_2^{i_2}\alpha_3^{i_3}\alpha_4^{i_4}\alpha_0^{i_0},&
\mathcal{E} &= \varepsilon_1^{i_1}\varepsilon_2^{i_2}\varepsilon_3^{i_3}\varepsilon_4^{i_4}:\Gamma\to\alpha(\Gamma),\\
\alpha' &= \alpha_1^{j_1}\alpha_2^{j_2}\alpha_3^{j_3}\alpha_4^{j_4}\alpha_0^{j_0},&
\mathcal{E}' &= \varepsilon_1^{j_1}\varepsilon_2^{j_2}\varepsilon_3^{j_3}\varepsilon_4^{j_4}:\Gamma\to\alpha'(\Gamma).
\end{align*}
Then since $\alpha_1,\dots,\alpha_4$ preserve $\Ch_{\llbracket\Gamma\rrbracket}$ but $\alpha_0$ interchanges $\varepsilon_2$ and $\varepsilon_3$, we have
\begin{align*}
\alpha_*(\mathcal{E}')&=\begin{cases}
\varepsilon_1^{j_1}\varepsilon_2^{j_2}\varepsilon_3^{j_3}\varepsilon_4^{j_4} & i_0=0;\\
\varepsilon_1^{j_1}\varepsilon_2^{j_3}\varepsilon_3^{j_2}\varepsilon_4^{j_4} & i_0=1,
\end{cases}&
( \alpha, \alpha') &=
\begin{cases}
\varepsilon_1^{i_1 j_1}\varepsilon_2^{i_2 j_2}\varepsilon_3^{i_3 j_3}\varepsilon_4^{i_4 j_4} & i_0=0;\\
\varepsilon_1^{i_1 j_1}\varepsilon_2^{i_2 j_3}\varepsilon_3^{i_3 j_2}\varepsilon_4^{i_4 j_4} & i_0=1.
\end{cases}
\end{align*}
\end{example}

\subsection{Group presentation for \texorpdfstring{$\Aut(A_\Gamma)$}{the automorphism group}}
\label{section:presentation}

Let us define the sets
\begin{align*}
S(\Gamma)&=
\left\{\varepsilon~\middle|~ \varepsilon\in E^\out_{\even}(\Ch_\Gamma)\right\}\cup
\left\{\varepsilon^2~\middle|~ \varepsilon\in E^\out_{\odd}(\Ch_\Gamma)\right\}\cup
\left\{\tilde\alpha~\middle|~ \alpha\in \Iso(\Gamma)\right\}\\
R_0(\Gamma)&=\left\{(s,t;m(e))=(t,s;m(e))\mid e=\{s,t\}\in E\right\}\\
&\mathrel{\hphantom{=}}\cup
\{\varepsilon v = (\varepsilon_\#(v))\varepsilon\mid v\in V, \varepsilon\in S(\Gamma)\}\cup
\{\varepsilon^2 v = (\varepsilon^2_\#(v))\varepsilon\mid v\in V, \varepsilon^2\in S(\Gamma)\},\\
&\mathrel{\hphantom{=}}\cup
\{\tilde\alpha v = (\tilde\alpha_\#(v))\tilde\alpha\mid v\in V, \tilde\alpha\in S(\Gamma)\}\cup
\{\iota v = v^{-1}\iota\mid v\in V\}\\
R_1(\Gamma)&=
\left\{\varepsilon \varepsilon'=\varepsilon' \varepsilon ~\middle|~ \varepsilon, \varepsilon'\in E^\out_{\even}(\Ch_\Gamma)\right\}
\cup
\left\{\varepsilon^2 \varepsilon'^2=\varepsilon'^2 \varepsilon^2 ~\middle|~ \varepsilon, \varepsilon'\in E^\out_{\odd}(\Ch_\Gamma)\right\}
\\
&\mathrel{\hphantom{=}}\cup
\left\{\varepsilon \varepsilon'^2 =\varepsilon'^2\varepsilon ~\middle|~ \varepsilon\in E^\out_{\even}(\Ch_\Gamma), \varepsilon'\in E^\out_{\odd}(\Ch_\Gamma)\right\},\\
R_2(\Gamma)&=
\left\{\tilde\alpha \varepsilon = \alpha_*(\varepsilon) \tilde\alpha ~\middle|~ \alpha\in \Iso(\Gamma), \varepsilon\in E^\out_{\even}(\Ch_\Gamma)\right\}\\
&\mathrel{\hphantom{=}}\cup
\left\{\tilde\alpha \varepsilon^2 = \alpha_*(\varepsilon)^2\tilde\alpha ~\middle|~ \alpha\in \Iso(\Gamma), \varepsilon\in E^\out_{\odd}(\Ch_\Gamma)\right\},\\
R_3(\Gamma)&=
\left\{
\tilde\alpha\tilde\beta=(\alpha,\beta)^2\widetilde{\alpha\beta}~\middle|~
\alpha, \beta\in \Iso(\Gamma)\right\}\\
R_4(\Gamma)&=\left\{\iota^2\right\}\cup
\left\{\tilde\alpha \iota=\mathcal{E}_\alpha^2\iota \tilde\alpha
~\middle|~ \alpha\in \Iso(\Gamma)\right\}\\
&\mathrel{\hphantom{=}}\cup
\left\{\varepsilon\iota\varepsilon\iota \mid \varepsilon\in E^\out_{\even}(\Ch_\Gamma)\right\}
\cup
\left\{\varepsilon^{2}\iota\varepsilon^2\iota \mid\varepsilon\in E^\out_{\odd}(\Ch_\Gamma)\right\},\\
\tilde R_\Phi(\Gamma)&=\begin{cases}
\varnothing & *_\Gamma\text{ is a chunk};\\
\{\mathcal{E}_{*_\Gamma}=x_e\} & *_\Gamma\text{ is an even-labeled separating edge }e;\\
\{\tilde\alpha_{*_\Gamma}=x_e\} & *_\Gamma\text{ is an odd-labeled separating edge }e,
\end{cases}
\end{align*}
and
\begin{align*}
R_\Phi(\Gamma)&=\begin{cases}
\varnothing & *_\Gamma\text{ is a chunk};\\
\{\mathcal{E}_{*_\Gamma}\} & *_\Gamma\text{ is an even-labeled separating edge};\\
\{\tilde\alpha_{*_\Gamma}\} & *_\Gamma\text{ is an odd-labeled separating edge}.
\end{cases}
\end{align*}
Recall \eqref{equation:edge-twist of the center} and \eqref{equation:special automorphism} for the definitions of $\mathcal{E}_{*_\Gamma}$ and $\Phi$, respectively.
If $*_\Gamma$ is an even-labeled separating edge, then $\Phi=\mathcal{E}_{*_\Gamma}\in\Dehn_{\mathscr{G}}(\Gamma)$ is a word in $S(\Gamma)$.
Otherwise, if $*_\Gamma$ is an odd-labeled separating edge, then
$\Phi=\mathcal{E}_{*_\Gamma}\alpha_{*_\Gamma}$ will be mapped to $\alpha_{*_\Gamma}$ via the map $\Aut_{\mathscr{G}}(\Gamma)\to\Iso(\Gamma)$.
Hence it corresponds to $\tilde\alpha_{*_\Gamma}\in S(\Gamma)$.

Here, we have the main theorem of the paper.
\begin{theorem}\label{theorem:main theorem}
Let $\Gamma=(V,E,m)$ be a \CLTTF graph. Then the automorphism group $\Aut(A_\Gamma)$ and outer automorphism group admit the following finite group presentations:
\begin{align*}
\Aut(A_\Gamma)&\cong
\left\langle V, S(\Gamma), \iota ~\middle|~ R_0(\Gamma), R_1(\Gamma), R_2(\Gamma), R_3(\Gamma), R_4(\Gamma), \tilde R_\Phi(\Gamma)
\right\rangle,\\
\Out(A_\Gamma)&\cong\left\langle
S(\Gamma), \iota ~\middle|~
R_1(\Gamma), R_2(\Gamma), R_3(\Gamma), R_4(\Gamma), R_\Phi
\right\rangle.
\end{align*}
\end{theorem}

In order to prove this theorem, we first consider a short exact sequence
\begin{equation}\label{equation:short exact sequence of even edge-twists}
\begin{tikzcd}
1\arrow[r] & \Dehn_{\mathscr{G}}(\Gamma) \arrow[r] & \Aut_{\mathscr{G}}(\Gamma) \arrow[->>, r] & \Iso(\Gamma) \arrow[r] & 1,
\end{tikzcd}
\end{equation}
where the quotient map
\[
\begin{tikzcd}[row sep=0pc]
\Aut_{\mathscr{G}}(\Gamma) \arrow[->>, r] & \Iso(\Gamma)\\
\mathcal{E}\alpha \arrow[mapsto, r] & \alpha,
\end{tikzcd}
\]
is just a projection, which is well-defined since $\mathcal{E}:\Delta\to \Gamma$ if $\alpha:\Gamma\to\Delta$ and whose kernel is precisely $\Dehn_{\mathscr{G}}(\Gamma)$.

Now we want to provide a group presentation for $\Aut_{\mathscr{G}}(\Gamma)\cong \Aut_{\mathscr{A}}(A_\Gamma)$ by using the following proposition whose proof is elementary and will be omitted.
\begin{proposition}\label{proposition:presentation for ses}
Let $N$ and $Q$ be groups admits presentations
\begin{align*}
N&=\langle S_N\mid R_N\rangle,&
Q&=\langle S_Q\mid R_Q\rangle,
\end{align*}
which fit into the short exact sequence
\[
\begin{tikzcd}
1\arrow[r]& N\arrow[r, "i"]& G\arrow[r,"\pi"]& Q\arrow[r]& 1.
\end{tikzcd}
\]
Let $s:F(S_Q)\to F(S_N)$ be a group homomorphism between free groups $F(S_Q)$ and $F(S_N)$ on $S_Q$ and $S_N$ which makes the following diagram commutative:
\[
\begin{tikzcd}
F(S_Q) \arrow[r,"s"] \arrow[d,->>,"{[\cdot]}"'] & F(S_N) \arrow[d, ->>,"{[\cdot]}"]\\
Q \arrow[from=r, "\pi"] & G,
\end{tikzcd}
\]
Here the vertical maps are the canonical surfections.
Then $G$ admits a group presentation 
\[
G\cong \langle S_N \cup s(S_Q) \mid R_N \cup R_C \cup \tilde{R}_Q\rangle,
\]
where
\begin{align*}
R_C&=\{s(t) g s(t)^{-1} w^{-1} \mid g\in S_N, t\in S_Q, w\in N, [s(t)g s(t)^{-1}]=i(w)\in G\}\\
\tilde{R}_Q&= \{s(r) h^{-1} \mid r\in R_Q, h\in N, [s(r)]=i(h)\in G\}.
\end{align*}
\end{proposition}

We want to use the short exact sequence in \eqref{equation:short exact sequence of even edge-twists} and the proposition above to obtain a group presentation of $\Aut_{\mathscr{G}}(\Gamma)$.

For each $(\alpha:\Gamma\to\Delta)\in\Iso(\Gamma)$, there is a unique composition of edge-twists $\bar{\mathcal{E}}_\alpha=\bar{\mathcal{E}}(\bar\eta_\alpha):\Delta\to\Gamma$ for some $\bar\eta_\alpha:E^\out_\odd(\Ch_{\llbracket\Gamma\rrbracket})\to\{0,1\}$ by Proposition~\ref{proposition:normal form for graphs}.
Let $\eta_\alpha:E^\out(\Ch_{\llbracket\Gamma\rrbracket})\to\mathbb{Z}$ be a function
defined as
\[
\eta_\alpha(\varepsilon)=\begin{cases}
1 & \varepsilon\in E^\out_\odd(\Ch_{\llbracket\Gamma\rrbracket}), \bar\eta_\alpha(\varepsilon)=1;\\
0 & \text{otherwise}.
\end{cases}
\]
Then we have a lift $\tilde\alpha\coloneqq\mathcal{E}_\alpha\alpha:\Gamma\to\Gamma$ of $\alpha$, where $\mathcal{E}_\alpha\coloneqq\mathcal{E}(\eta_\alpha)$.

For each $g\in \Dehn_{\mathscr{G}}(\Gamma)$, the conjugate of $g$ by $\tilde\alpha$ is then 
\begin{align*}
\tilde\alpha g\tilde\alpha^{-1} &=
(\mathcal{E}_\alpha\alpha) g (\alpha^{-1}\mathcal{E}_\alpha^{-1})
=\mathcal{E}_\alpha\alpha_*(g)\mathcal{E}_\alpha^{-1}
=\alpha_*(g).
\end{align*}

For $\alpha, \beta\in \Iso(\Gamma)$, we have the following:
\[
\tilde\alpha\tilde\beta = \mathcal{E}_\alpha \alpha \mathcal{E}_\beta\beta
=\mathcal{E}_\alpha \alpha_*(\mathcal{E}_\beta) \alpha\beta\quad
\text{ and }\quad
\widetilde{\alpha\beta}=\mathcal{E}_{\alpha\beta}(\alpha\beta),
\]
which coincide in $\Iso(\Gamma)$ and so 
%Since $\alpha\beta=\gamma$ in $\Iso(\Gamma)$, we have $\alpha\beta=\gamma$ in $\mathfrak{S}_V$ and 
\[
\mathcal{E}_\alpha \alpha_*(\mathcal{E}_\beta) \mathcal{E}_{\alpha\beta}^{-1}
\]
is a composition of even edge-twists. Indeed, this is exactly the same as $(\alpha,\beta)^2$ by the meaning of the twisted intersection product as mentioned earlier.
Therefore we have 
\[
\mathcal{E}_\alpha \alpha_*(\mathcal{E}_\beta)\mathcal{E}_{\alpha\beta}^{-1}
=(\alpha,\beta)^2\in\Dehn_{\mathscr{G}}(\Gamma),\quad\text{ or equivalently,}
\quad
\tilde\alpha\tilde\beta = (\alpha,\beta)^2\widetilde{\alpha\beta}.
\]

\begin{proposition}
The groups $\Aut_{\mathscr{A}}(A_\Gamma)$ and $\Aut_{\mathscr{A}}(A_\Gamma)/Z_\Gamma$ admit the finite group presentations
\begin{align*}
\Aut_{\mathscr{A}}(A_\Gamma)&\cong
\langle S(\Gamma) \mid R_1(\Gamma), R_2(\Gamma), R_3(\Gamma)\rangle,\\
\Aut_{\mathscr{A}}(A_\Gamma)/Z_\Gamma&\cong
\langle S(\Gamma) \mid R_1(\Gamma), R_2(\Gamma), R_3(\Gamma), R_\Phi\rangle.
\end{align*}
\end{proposition}
\begin{proof}
We use Proposition~\ref{proposition:presentation for ses} on \eqref{equation:short exact sequence of even edge-twists}.
Since $\Dehn_{\mathscr{G}}(\Gamma)$ is a free abelian group generated by even edge-twists, the group $\Aut_{\mathscr{A}}(A_\Gamma)$ is generated by the set $S(\Gamma)$.
Moreover the sets $R_1, R_2$ and $R_3$ correspond to $R_N, R_C$ and $\tilde R_Q$ in Proposition~\ref{proposition:presentation for ses} and the generator for $Z_\Gamma$ corresponds to the element in $R_\Phi$ if $R_\Phi\neq\varnothing$. Therefore we are done.
\end{proof}

\begin{proof}[Proof of Theorem~\ref{theorem:main theorem}]
We first find the group presentation for $\Out(A_\Gamma)$, which is isomorphic to $(\Aut_{\mathscr{A}}(A_\Gamma)/Z_\Gamma)\rtimes \mathbb{Z}_2$ by Theorem~\ref{theorem:outer automorphism group}.
Therefore we need to justify relations in $R_4(\Gamma)$.

Since $\iota$ is an involution, $\iota^2$ is the identity.
Let $\varepsilon=(e=\{s,t\}, C)$ be an edge in $E^\out(\Ch_\Gamma)$.
For each $v\in V$, 
\begin{align*}
\iota\varepsilon_\#(v) &=
\begin{cases}
\iota(v) & v\in V_1(\varepsilon);\\
\iota(x_e^{-1}vx_e) & v\not\in V_1(\varepsilon).
\end{cases}\\
&=\begin{cases}
v & v\in V_1(\varepsilon);\\
\overline{x_e}^{-1}v^{-1}\overline{x_e} & v\not\in V_1(\varepsilon),
\end{cases}
\end{align*}
where $\overline{x_e}$ is the reverse of $x_e$ and identical to $x_e$ in $A_\Gamma$.
Therefore, 
\[
\iota\varepsilon_\#(v) = (\varepsilon_\#)^{-1}(v^{-1}) = \varepsilon_\#^{-1}\iota(v)
\]
and so the second or third type of relations follows.
We also note that for each $\varepsilon_1, \varepsilon_2\in E^\out(\Ch_\Gamma)$,
\begin{equation}\label{equation:involution and partial conjugation}
\iota(\varepsilon_1)_\#(\varepsilon_2)_\#=
(\varepsilon_1)^{-1}_\#(\varepsilon_2)^{-1}_\#\iota.
\end{equation}

In order to check if the relation $\tilde\alpha \iota=\mathcal{E}_\alpha^2\iota \tilde\alpha$ holds, let us regard $\tilde\alpha$ as the composition $\tilde\alpha = \mathcal{E}_\alpha \alpha$,
where $\alpha:\Gamma\to\Delta$ is a graph isomorphism for some $\Delta$ and $\bar{\mathcal{E}}_\alpha:\Delta\to\Gamma$ is a composition of edge-twists.
Then by \eqref{equation:involution and partial conjugation}
\begin{align*}
(\iota\tilde\alpha)(v) &= \iota(\mathcal{E}_\alpha)_\#\alpha(v)
=(\mathcal{E}_\alpha)_\#^{-1}(\iota\alpha)(v)
=(\mathcal{E}_\alpha)_\#^{-1}(\alpha\iota)(v)
\end{align*}
for each $v\in V$.
Hence $\iota\tilde\alpha=(\mathcal{E}_\alpha)_\#^{-2}\tilde\alpha \iota$, or equivalently, $(\mathcal{E}_\alpha)^2_\#\iota\tilde\alpha = \tilde\alpha\iota$.

Finally, we use Proposition~\ref{proposition:presentation for ses} again to the short exact sequence between automorphism groups.
\[
\begin{tikzcd}
1\arrow[r] & \Inn(A_\Gamma) \arrow[r] & \Aut(A_\Gamma) \arrow[r] & \Out(A_\Gamma) \arrow[r] & 1
\end{tikzcd}
\]

Since $\Inn(A_\Gamma)\cong A_\Gamma$, we may use the group presentation for $A_\Gamma$.
Therefore, the group $\Aut(A_\Gamma)$ is generated by two sets $V$ and $S(\Gamma)$ defined earlier.
The relations are consisting of three types of relations such that (i) the original relations in $A_\Gamma$, (ii) the action of $\Out(A_\Gamma)$ on $A_\Gamma\cong\Inn(A_\Gamma)$, and (iii) lifts of relations in $\Out(A_\Gamma)$.

The generating set is obvious.
The set $R_0$ consists of relations in $A_\Gamma$ and relations corresponding to the action of $\Out(A_\Gamma)$ on $A_\Gamma$, where the latter relations are obvious by definition.
Moreover, all relations but $R_\Phi(\Gamma)$ in $\Out(A_\Gamma)$ hold in $\Aut(\Gamma)$ as well and the can be lifted without any modification. However, the relation $R_\Phi(\Gamma)$ may not hold in $\Aut(\Gamma)$ when $*_\Gamma$ is a separating edge.
In this case, we should identify $\mathcal{E}_{*_\Gamma}$ or $\tilde\alpha_{*_\Gamma}$ with the inner automorphism $(x_e)_\#$, which is just $x_e$ in our presentation. This completes the proof.
\end{proof}

\subsection{Examples}
We will compute (outer) automorphism groups for various \CLTTF graphs.

\subsubsection{Discretely rigid with the central chunk}
As seen earlier, the following \CLTTF graph $\Gamma$ is not rigid but discretely rigid.
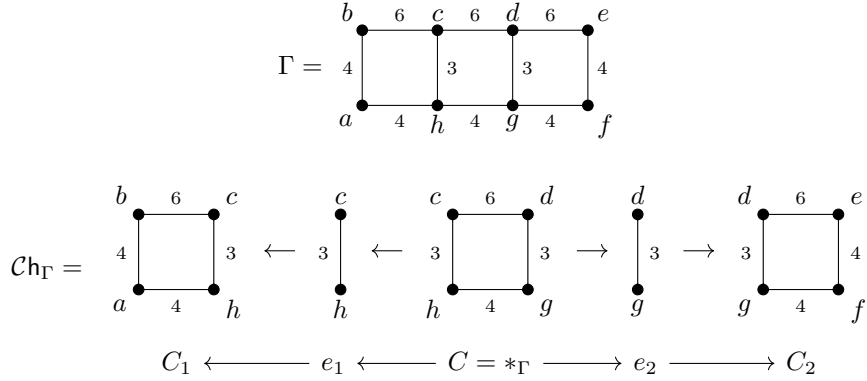
\begin{align*}
\Gamma&=\begin{tikzpicture}[baseline=-.5ex]
\foreach \i in {-1.5,-0.5,0.5,1.5} {
\draw[fill] (\i,-0.5) circle (2pt) (\i, 0.5) circle (2pt);
}
\draw (-1.5,-0.5) node[below left] {$a$} -- node[midway, left] {$\scriptstyle 4$} (-1.5, 0.5) node[above left] {$b$};
\draw (-0.5,-0.5) node[below] {$h$} -- node[midway, right] {$\scriptstyle 3$} (-0.5, 0.5) node[above] {$c$};
\draw (0.5,-0.5) node[below] {$g$} -- node[midway, right] {$\scriptstyle 3$} (0.5, 0.5) node[above] {$d$};
\draw (1.5,-0.5) node[below right] {$f$} -- node[midway, right] {$\scriptstyle 4$} (1.5, 0.5) node[above right] {$e$};
\draw (-1.5,-0.5) -- node[midway, below] {$\scriptstyle 4$} (-0.5, -0.5);
\draw (-1.5,0.5) -- node[midway, above] {$\scriptstyle 6$} (-0.5, 0.5);
\draw (-0.5,-0.5) -- node[midway, below] {$\scriptstyle 4$} (0.5, -0.5);
\draw (-0.5,0.5) -- node[midway, above] {$\scriptstyle 6$} (0.5, 0.5);
\draw (0.5,-0.5) -- node[midway, below] {$\scriptstyle 4$} (1.5, -0.5);
\draw (0.5,0.5) -- node[midway, above] {$\scriptstyle 6$} (1.5, 0.5);
\end{tikzpicture}
\end{align*}
\[
\Ch_\Gamma=
\begin{tikzcd}[ampersand replacement=\&, column sep=1pc, row sep=0pc]
\begin{tikzpicture}[baseline=-.5ex]
\draw[fill] 
(-0.5, -0.5) circle (2pt) node[below left] {$a$} 
-- node[midway, left] {$\scriptstyle 4$} (-0.5, 0.5) circle (2pt)  node[above left] {$b$}
-- node[midway, above] {$\scriptstyle 6$} (0.5, 0.5) circle (2pt) node[above right] {$c$}
-- node[midway, right] {$\scriptstyle 3$} (0.5, -0.5) circle (2pt)  node[below right] {$h$}
-- node[midway, below] {$\scriptstyle 4$} (-0.5, -0.5);
\end{tikzpicture}
\arrow[from=r] \&
\begin{tikzpicture}[baseline=-.5ex]
\draw[fill]
(0, -0.5) circle (2pt) node[below] {$h$} -- node[midway, left] {$\scriptstyle3$} (0, 0.5) circle (2pt) node[above] {$c$};
\end{tikzpicture}
\arrow[from=r] \&
\begin{tikzpicture}[baseline=-.5ex]
\draw[fill] 
(-0.5, -0.5) circle (2pt) node[below left] {$h$} 
-- node[midway, left] {$\scriptstyle 3$} (-0.5, 0.5) circle (2pt)  node[above left] {$c$}
-- node[midway, above] {$\scriptstyle 6$} (0.5, 0.5) circle (2pt) node[above right] {$d$}
-- node[midway, right] {$\scriptstyle 3$} (0.5, -0.5) circle (2pt)  node[below right] {$g$}
-- node[midway, below] {$\scriptstyle 4$} (-0.5, -0.5);
\end{tikzpicture}
\arrow[r] \&
\begin{tikzpicture}[baseline=-.5ex]
\draw[fill]
(0, -0.5) circle (2pt) node[below] {$g$} -- node[midway, right] {$\scriptstyle3$} (0, 0.5) circle (2pt) node[above] {$d$};
\end{tikzpicture}
\arrow[r] \&
\begin{tikzpicture}[baseline=-.5ex]
\draw[fill] 
(-0.5, -0.5) circle (2pt) node[below left] {$g$} 
-- node[midway, left] {$\scriptstyle 3$} (-0.5, 0.5) circle (2pt)  node[above left] {$d$}
-- node[midway, above] {$\scriptstyle 6$} (0.5, 0.5) circle (2pt) node[above right] {$e$}
-- node[midway, right] {$\scriptstyle 4$} (0.5, -0.5) circle (2pt)  node[below right] {$f$}
-- node[midway, below] {$\scriptstyle 4$} (-0.5, -0.5);
\end{tikzpicture}\\
C_1 \arrow[from=r] \& e_1 \arrow[from=r] \& C =*_\Gamma \arrow[r] \& e_2 \arrow[r] \& C_2
\end{tikzcd}
\]
Then in the chunk tree $\Ch_\Gamma$, the central vertex $*_\Gamma$ is a chunk $C$ with $V(C)=\{c,d,g,h\}$ and there are two outward edges
\[
\varepsilon_1 = (e_1, C_1)\quad\text{ and }\quad
\varepsilon_2 = (e_2, C_2),
\]
where $C_1$ and $C_2$ are induced subgraphs of $\Gamma$ with $V(C_1) = \{a,b,c,h\}$ and $V(C_2)=\{d,e,f,g\}$.
Moreover
$\Aut(\Gamma)=\langle \alpha \mid \alpha^2\rangle$,
where $\alpha:\Gamma\to\Gamma$ is the obvious horizontal reflection that interchanges edge-twists $\varepsilon_1$ and $\varepsilon_2$.
Therefore by Corollary~\ref{corollary:automorphism},
\begin{align*}
\Aut(A_\Gamma)&\cong \Inn(A_\Gamma)\rtimes \Aut_{\mathscr{A}}(A_\Gamma)
\cong A_\Gamma\rtimes \left(\left(\langle \varepsilon_1, \varepsilon_2 \rangle\rtimes \langle \alpha\mid \alpha^2 \rangle\right)\rtimes \langle \iota\mid \iota^2 \rangle \right).
%\left(\left(\mathbb{Z}^{2}\rtimes \mathbb{Z}_2\right)\rtimes \mathbb{Z}_2\right).
\end{align*}

The followings are precise relations:
\begin{itemize}
\item Inner automorphism group relation and the action of $\Out(A_\Gamma)$ on $\Inn(A_\Gamma)$
\begin{align*}
R_0&=\left\{(s,t;m(e))=(t,s;m(e))\mid e=\{s,t\}\in E\right\}\\
&\mathrel{\hphantom{=}}\cup
\{\alpha v = (\alpha(v))\alpha\mid v\in V\}\cup
\{\varepsilon_i v = (\varepsilon_i(v))\varepsilon_i\mid v\in V, i=1,2\}\cup
\{\iota v = v^{-1}\iota\mid v\in V\}
\end{align*}

\item Commutative relations between $\varepsilon_1$ and $\varepsilon_2$
\[
R_1=\{\varepsilon_1\varepsilon_2=\varepsilon_2\varepsilon_1\}
\]

\item The action of $\alpha$ on $\langle\varepsilon_1, \varepsilon_2\rangle$ and involutivity 
\begin{align*}
R_2&=\{\alpha \varepsilon_1= \varepsilon_2 \alpha\},&
R_3&=\{\alpha^2\}
\end{align*}

\item The action of $\iota$ on $\Aut_{\mathscr{A}}(\Gamma)$
\begin{align*}
R_4=\{\iota^2, \iota \alpha = \alpha \iota\}\cup
\{\iota\varepsilon_i = \varepsilon_i^{-1}\iota\mid i=1,2\}
\end{align*}
\end{itemize}

Therefore we have the following group presentations:
\begin{align*}
\Aut(A_\Gamma)&=\left\langle
V, \varepsilon_1, \varepsilon_2, \alpha, \iota \mid R_0, R_1, R_2, R_3, R_4
\right\rangle,\\
\Out(A_\Gamma)&=\left\langle
\varepsilon_1, \varepsilon_2, \alpha, \iota \mid R_1, R_2, R_3, R_4
\right\rangle.
\end{align*}

\subsubsection{Discretely non-rigid with the central separating edge}
The \CLTTF graph $\Gamma$ below is rigid but not discretely rigid with the central separating edge $e=\{c,f\}$
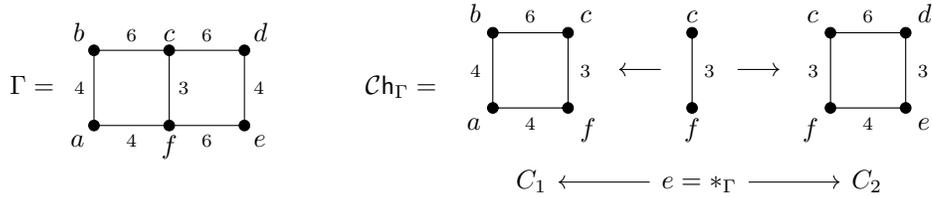
\begin{align*}
\Gamma&=\begin{tikzpicture}[baseline=-.5ex]
\foreach \i in {-1,0,1} {
\draw[fill] (\i,-0.5) circle (2pt) (\i, 0.5) circle (2pt);
}
\draw (-1,-0.5) node[below left] {$a$} -- node[midway, left] {$\scriptstyle 4$} (-1, 0.5) node[above left] {$b$};
\draw (0,-0.5) node[below] {$f$} -- node[midway, right] {$\scriptstyle 3$} (0, 0.5) node[above] {$c$};
\draw (1,-0.5) node[below right] {$e$} -- node[midway, right] {$\scriptstyle 4$} (1, 0.5) node[above right] {$d$};
\draw (-1,-0.5) -- node[midway, below] {$\scriptstyle 4$} (0, -0.5);
\draw (-1,0.5) -- node[midway, above] {$\scriptstyle 6$} (0, 0.5);
\draw (0,-0.5) -- node[midway, below] {$\scriptstyle 6$} (1, -0.5);
\draw (0,0.5) -- node[midway, above] {$\scriptstyle 6$} (1, 0.5);
\end{tikzpicture}&
\Ch_\Gamma&=
\begin{tikzcd}[ampersand replacement=\&, column sep=1pc, row sep=0pc]
\begin{tikzpicture}[baseline=-.5ex]
\draw[fill] 
(-0.5, -0.5) circle (2pt) node[below left] {$a$} 
-- node[midway, left] {$\scriptstyle 4$} (-0.5, 0.5) circle (2pt)  node[above left] {$b$}
-- node[midway, above] {$\scriptstyle 6$} (0.5, 0.5) circle (2pt) node[above right] {$c$}
-- node[midway, right] {$\scriptstyle 3$} (0.5, -0.5) circle (2pt)  node[below right] {$f$}
-- node[midway, below] {$\scriptstyle 4$} (-0.5, -0.5);
\end{tikzpicture}
\arrow[from=r] \&
\begin{tikzpicture}[baseline=-.5ex]
\draw[fill]
(0, -0.5) circle (2pt) node[below] {$f$} -- node[midway, right] {$\scriptstyle 3$}
(0, 0.5) circle (2pt) node[above] {$c$};
\end{tikzpicture}
\arrow[r] \&
\begin{tikzpicture}[baseline=-.5ex]
\draw[fill] 
(-0.5, -0.5) circle (2pt) node[below left] {$f$} 
-- node[midway, left] {$\scriptstyle 3$} (-0.5, 0.5) circle (2pt)  node[above left] {$c$}
-- node[midway, above] {$\scriptstyle 6$} (0.5, 0.5) circle (2pt) node[above right] {$d$}
-- node[midway, right] {$\scriptstyle 3$} (0.5, -0.5) circle (2pt)  node[below right] {$e$}
-- node[midway, below] {$\scriptstyle 4$} (-0.5, -0.5);
\end{tikzpicture}\\
C_1 \arrow[from=r] \& e=*_\Gamma \arrow[r] \& C_2
\end{tikzcd}
\end{align*}
There are two outward edges in $E^{\out}(\Ch_\Gamma)$
\[
\varepsilon_1=(e_1, C_1)\quad\text{ and }\quad
\varepsilon_2=(e_2, C_2)
\]
and so we have four \CLTTF graphs edge-twist equivalent to $\Gamma$
\[
\Gamma_{0}\coloneqq\Gamma,\qquad
\Gamma_{1}\coloneqq\bar\varepsilon_1(\Gamma),\qquad
\Gamma_{2}\coloneqq\bar\varepsilon_2(\Gamma),\quad\text{ and }\quad
\Gamma_{3}\coloneqq\bar\varepsilon_1\bar\varepsilon_2(\Gamma),
\]
which are all isomorphic to $\Gamma$.
Moreover, since the group of graph automorphisms is trivial, each edge-twist equivalent graph $\Gamma_{i}$ has the unique graph isomorphism $\alpha_{i}$, where
\begin{align*}
\alpha_{0}&=\operatorname{Id}_V,&
\alpha_{1}(v)&=
\begin{cases}
f & v=c;\\
e & v=d;\\
d & v=e;\\
c & v=f;\\
v & v\in \{a,b\},
\end{cases}&
\alpha_{2}(v)&=
\begin{cases}
e & v=d;\\
d & v=e;\\
v & v\neq d,e,
\end{cases}&
\alpha_{3}&=\alpha_{1}\alpha_{2}.
%\begin{cases}
%f & v=c;\\
%c & v=f;\\
%v & v\neq c,f.
%\end{cases}
\end{align*}
Therefore 
\[
\Iso(\Gamma)=\left\langle
\alpha_{1}, \alpha_{2} \mid \alpha_{1}\alpha_{2}=\alpha_{2}\alpha_{1},
\alpha_{1}^2, \alpha_{2}^2
\right\rangle
\cong \mathbb{Z}_2\times \mathbb{Z}_2.
\]

Since both $\alpha_1$ and $\alpha_2$ act trivially on $\Ch_\Gamma$, we have
\[
(\alpha_i)_*(\varepsilon_j) = \varepsilon_j,\quad\text{ for all }\quad
1\le i,j\le 2
\]
and so the twisted intersection is then defined as follows: for $i=1,2$,
\begin{align*}
(\alpha_{i}, \alpha_{i}) &= (\alpha_i, \alpha_3) = (\alpha_3, \alpha_i) = \varepsilon_i,&
(\alpha_1, \alpha_2) &= (\alpha_2, \alpha_1) = \operatorname{Id},&
(\alpha_3, \alpha_3) &=\varepsilon_1\varepsilon_2.
\end{align*}

Now we compute the group $\Aut_{\mathscr{A}}(A_\Gamma)$, which is generated by
\[
\varepsilon_1^2, \varepsilon_2^2, \tilde\alpha_1, \tilde\alpha_2, \tilde\alpha_3.
\]
The sets of relations are as follows:
\begin{align*}
R_1 &= \{\varepsilon_1^2\varepsilon_2^2 = \varepsilon_2^2\varepsilon_1^2\},\\
R_2 &= \{\tilde\alpha_i \varepsilon_j^2 =\varepsilon_j^2 \tilde\alpha_i\mid 1\le i\le 3, 1\le j\le 2\},\\
R_3 &= \{
\tilde\alpha_i^2 = \varepsilon_i^2 \mid 1\le i\le 2
\}\cup\{\tilde\alpha_3^2 = \varepsilon_1^2\varepsilon_2^2\}\\
&\mathrel{\hphantom{=}}\cup
\{\tilde\alpha_1\tilde\alpha_2 = \tilde\alpha_2\tilde\alpha_1 = \tilde\alpha_3,
\tilde\alpha_1\tilde\alpha_3 = \tilde\alpha_3\tilde\alpha_1 = \varepsilon_1^2\tilde\alpha_2,
\tilde\alpha_2\tilde\alpha_3 = \tilde\alpha_3\tilde\alpha_2 = \varepsilon_1^2\tilde\alpha_1\}
\end{align*}
Hence one can easily see that
\[
\Aut_{\mathscr{A}}(\Gamma)=\left\langle
\varepsilon_1^2, \varepsilon_2^2, \tilde\alpha_1, \tilde\alpha_2, \tilde\alpha_3\mid
R_1, R_2, R_3
\right\rangle
=\left\langle\tilde\alpha_1, \tilde\alpha_2\mid 
\tilde\alpha_1\tilde\alpha_2=\tilde\alpha_2\tilde\alpha_1
\right\rangle\cong \mathbb{Z}^2.
\]

Moreover, since the central element $*_\Gamma$ is an odd-labeled edge, we have the nontrivial subgroup $Z_\Gamma$ generated by $\mathcal{E}_{*_\Gamma}\alpha_{*_\Gamma}$, where
\[
\mathcal{E}_{*_\Gamma} = \varepsilon_1\varepsilon_2\quad\text{ and }\quad
\alpha_{*_\Gamma} = \alpha_1\alpha_2.
\]
That is, in the above group presentation, the generator for $Z_\Gamma$ corresponds to $\tilde\alpha_1\tilde\alpha_2$ and 

\begin{align*}
\Aut_{\mathscr{A}}(A_\Gamma)/Z_\Gamma&\cong
\left\langle
\varepsilon_1^2, \varepsilon_2^2, \tilde\alpha_1, \tilde\alpha_2, \tilde\alpha_3\mid
R_1, R_2, R_3, R_\Phi
\right\rangle\\
&\cong\left\langle\tilde\alpha_1, \tilde\alpha_2
\mid
\tilde\alpha_1\tilde\alpha_2=\tilde\alpha_2\tilde\alpha_1, \tilde\alpha_1\tilde\alpha_2
\right\rangle\cong \mathbb{Z},
\end{align*}
where $R_\Phi=\{\varepsilon_1\varepsilon_2\}$.
We have the following commutative diagram with exact rows:
\begin{equation}\label{equation:example exact sequence}
\begin{tikzcd}
1\arrow[r] & Z_\Gamma\arrow[d,"\cong"] \arrow[r] &
\Aut_{\mathscr{A}}(A_\Gamma)\arrow[d,"\cong"] 
\arrow[r] & \Aut_{\mathscr{A}}(A_\Gamma)/Z_\Gamma\arrow[d,"\cong"] \arrow[r] & 1\\
1\arrow[r] & \langle \tilde\alpha_1\tilde\alpha_2\rangle \arrow[r] & \langle \tilde\alpha_1, \tilde\alpha_2\rangle \arrow[r, "\tilde\alpha_2\mapsto \tilde\alpha_1^{-1}"] &
\langle \tilde\alpha_1\rangle \arrow[r] & 1
\end{tikzcd}
\end{equation}
Therefore by Theorem~\ref{theorem:outer automorphism group}, we have
\begin{align*}
\Out(A_\Gamma)\cong(\Aut_{\mathscr{A}}(A_\Gamma)/Z_\Gamma)\rtimes \mathbb{Z}_2
\cong
\langle\tilde\alpha_1, \iota\mid \iota\tilde\alpha_1\iota=\tilde\alpha_1^{-1}\rangle
\cong
\mathbb{Z}\rtimes \mathbb{Z}_2,
\end{align*}
where the last relation comes from 
\[
\tilde\alpha_1\iota=\varepsilon_1^2\iota\tilde\alpha_1
\Longleftrightarrow
\tilde\alpha_1\iota=\tilde\alpha_1^2\iota\tilde\alpha_1
\Longleftrightarrow
\iota\tilde\alpha_1\iota=\tilde\alpha_1^{-1}.
\]

Moreover, the below row in \eqref{equation:example exact sequence} splits and so we may regard $\Out(A_\Gamma)$ as a subgroup of $\Aut(A_\Gamma)$ so that $\Aut(A_\Gamma)\cong \Inn(A_\Gamma) \rtimes \Out(A_\Gamma)$.
Hence the automorphism group $\Aut(A_\Gamma)$ admits the following presentation:
\[
\Aut(A_\Gamma)\cong
\left\langle
V, \tilde\alpha_1, \iota \mid R_0,
\iota \tilde\alpha_1 \iota = \tilde\alpha_1^{-1}
\right\rangle,
\]
where
\[
R_0=\{(s,t;m(e))=(t,s;m(e))\mid e=\{s,t\}\in E\}\cup\{\tilde\alpha_1 v = (\tilde\alpha_1)_\#(v) \tilde\alpha_1, \iota v = v^{-1}\iota\}
\]
and
\[
(\tilde\alpha_1)_\#(v) =\begin{cases}
f & v=c;\\
e & v=d;\\
d & v=e;\\
c & v=f;\\
(cfc)^{-1}v(cfc) & v\in \{a,b\}.
\end{cases}
\]

\subsubsection{Discretely non-rigid with the central chunk}
The \CLTTF graph below is discretely non-rigid and the center $*_\Gamma$ in the chunk tree is a chunk $C_0$ as seen earlier.
\[
\Gamma=\begin{tikzpicture}[baseline=-.5ex]
\draw[fill] (0, -0.5) circle (2pt) node[below=2ex, left=-.5ex] {$a$};
\draw[fill] (1, -0.5) circle (2pt) node[right] {$d$};
\draw[fill] (0, 0.5) circle (2pt) node[above=2ex, left] {$i$};
\draw[fill] (1, 0.5) circle (2pt) node[right] {$e$};
\draw[fill] (0, -1.5) circle (2pt) node[below left] {$b$};
\draw[fill] (1, -1.5) circle (2pt) node[below right] {$c$};
\draw[fill](1, 0.5) ++(72:1) circle (2pt) node[right] {$f$} ++(144:1) circle(2pt) node[above] {$g$} ++(216:1) circle(2pt) node[left] {$h$};
\draw[fill] (-1.5,1.5) circle (2pt) node[above left] {$j$};
\draw[fill] (-1.5,0.5) circle (2pt) node[left] {$k$};
\draw[fill] (-1.5,-0.5) circle (2pt) node[left] {$\ell$};
\draw[fill] (-1.5,-1.5) circle (2pt) node[below left] {$m$};
\draw (0.5,0.5) node[above=-.5ex] {$e_1$} (0,0) node[right=-1ex] {$e_2$} (0.5,-0.5) node[below=-.5ex] {$e_3$};
\draw[color=black, fill=red, fill opacity=0.2](0, -1.5) rectangle node[opacity=1, below] {$C_4$} (1,-0.5);
\draw[color=black, fill=yellow, fill opacity=0.2](1, 0.5) -- ++(72:1) -- ++(144:1) -- ++(216:1) -- ++(288:1);
\draw (0.5, 1) node[opacity=1, above] {$C_1$};
\draw[color=black, fill=white, fill opacity=0.2](0, 0.5) -- ++(-1.5, 1) -- ++(-90:1) node[opacity=1, above right] {$C_2$} -- ++(1.5, -1);
\draw[color=black, fill=black, fill opacity=0.2](0, -0.5) -- ++(-1.5, -1) -- ++(90:1) node[opacity=1, below right] {$C_3$} -- ++(1.5, 1);
\draw[color=black, fill=blue, fill opacity=0.2](0, -0.5) rectangle node[opacity=1, right=-1ex] {$C_0$} (1,0.5);
\end{tikzpicture}
\]

Therefore by the corollary~\ref{corollary:automorphism},
$\Aut(A_\Gamma)\cong \Inn(A_\Gamma)\rtimes \left(\Aut_{\mathscr{A}}(A_\Gamma)\rtimes \mathbb{Z}_2\right)$.
Recall the generating set $\{\alpha_0, \dots, \alpha_4\}$ for $\Iso(\Gamma)$ and edge-twists ${\bar\varepsilon_1},\cdots,{\bar\varepsilon_4}$ as described in Example~\ref{example:generators of Iso} so that $\bar\varepsilon_i \alpha_i :\Gamma\to \Gamma$ for each $1\le i\le 4$.

The set $S$ of generators for $\Aut_{\mathscr{A}}(A_\Gamma)$ is then
\[
S=\{\tilde\alpha_0, \dots, \tilde\alpha_4, \varepsilon_1^2, \dots, \varepsilon_4^2\}
\]
and there are three types of relations
\begin{align*}
R_1 &= \{ \varepsilon_i^2 \varepsilon_j^2 = \varepsilon_j^2\varepsilon_i^2\mid 1\le i,j\le 4 \},\\
R_2 &= \{\tilde\alpha_i \varepsilon_j^2 = \varepsilon_j^2 \tilde\alpha_i \mid 1\le i, j\le 4\}
\cup
\{\tilde\alpha_0 \varepsilon_i^2 = \varepsilon_i^2 \tilde\alpha_0 \mid i=1,4\}
\cup
\{
\tilde\alpha_0 \varepsilon_1^2 = \varepsilon_2^2 \tilde\alpha_0,
\tilde\alpha_0 \varepsilon_2^2 = \varepsilon_1^2 \tilde\alpha_0
\},\\
R_3 &= \{\tilde\alpha_0^2\}\cup\{\tilde\alpha_i^2=\varepsilon_i^2\mid 1\le i\le 4\}
\cup
\{\tilde\alpha_i\tilde\alpha_j = \tilde\alpha_j\tilde\alpha_i \mid 1\le i,j\le 4\}\\
&\mathrel{\hphantom{=}}\cup
\{\tilde\alpha_0\tilde\alpha_i = \tilde\alpha_i\tilde\alpha_0 \mid i=1,4\}\cup
\{\tilde\alpha_0\tilde\alpha_2 = \tilde\alpha_3\tilde\alpha_0,
\tilde\alpha_0\tilde\alpha_3 = \tilde\alpha_2\tilde\alpha_0\}
\end{align*}
so that
\[
\Aut_{\mathscr{A}}(A_\Gamma)\cong
\langle \tilde\alpha_0, \dots, \tilde\alpha_4, \varepsilon_1^2, \dots, \varepsilon_4^2\mid 
R_1, R_2, R_3 
\rangle.
\]

Now the action of $\iota$ gives us relations in $\Out(A_\Gamma)$ as follows:
\begin{align*}
R_4&=\{\iota^2\}\cup\{\iota\tilde\alpha_0=\tilde\alpha_0\iota\}\cup
\{\tilde\alpha_i\iota=\varepsilon_i^2\iota\tilde\alpha_i\mid 1\le i\le 4\}\cup
\{\varepsilon_i\iota\varepsilon_i\iota\mid 1\le i\le 4\}
\end{align*}
and we have a group presentation 
\[
\Out(A_\Gamma)=\left\langle
\tilde\alpha_0, \dots, \tilde\alpha_4, \varepsilon_1^2, \dots, \varepsilon_4^2, \iota~\middle|~
R_1, R_2, R_3, R_4\right\rangle.
\]

One can reduce this presentation so that
\begin{align*}
\Aut_{\mathscr{A}}(A_\Gamma)&\cong
\left(
\langle \tilde\alpha_1\rangle \times
\langle \tilde\alpha_2\rangle \times
\langle \tilde\alpha_3\rangle \times
\langle \tilde\alpha_4\rangle
\right)\rtimes \langle \tilde\alpha_0\mid \tilde\alpha_0^2\rangle
\cong
\mathbb{Z}^4\rtimes \mathbb{Z}_2\\
\Out(A_\Gamma)&\cong
\left(
\langle \tilde\alpha_1\rangle \times
\langle \tilde\alpha_2\rangle \times
\langle \tilde\alpha_3\rangle \times
\langle \tilde\alpha_4\rangle
\right)
\rtimes 
\left(
\langle \tilde\alpha_0\mid \tilde\alpha_0^2\rangle\times
\langle \iota\mid \iota^2\rangle
\right)\cong
\mathbb{Z}^4\rtimes \mathbb{Z}_2^2
\end{align*}
Here $\iota$ acts on $\tilde\alpha_0$ trivially and on $\tilde\alpha_i$ for $1\le i\le 4$ as 
\[
\iota\tilde\alpha_i\iota=\tilde\alpha_i^{-1}.
\]

Finally, the group presentation for $\Aut(A_\Gamma)$ is given as
\begin{align*}
\Aut(A_\Gamma)&\cong
\left\langle
V, \tilde\alpha_0,\dots,\tilde\alpha_4, \varepsilon_1^2,\dots,\varepsilon_4^2, \iota
~\middle|~
R_0, R_1, R_2, R_3, R_4
\right\rangle\\
&\cong
A_\Gamma\rtimes\left(\mathbb{Z}^4\rtimes\mathbb{Z}_2^2\right),
\end{align*}
where
\begin{align*}
R_0&=\{(s,t;m(e))=(t,s;m(e))\mid e=\{s,t\}\in E\}\\
&\mathrel{\hphantom{=}}\cup\{\tilde\alpha_i v = (\tilde\alpha_i)_\#(v)\tilde\alpha_i\mid 0\le i\le 4, v\in V\}\\
&\mathrel{\hphantom{=}}\cup\{\varepsilon_i^2 v = (\varepsilon_i)_\#^2(v)\varepsilon_i^2\mid 1\le i\le 4, v\in V\}\\
&\mathrel{\hphantom{=}}\cup\{\iota v = v^{-1}\iota\mid v\in V\}.
\end{align*}

\begin{remark}
The above presentation can be reduced further but we omit the detail.
\end{remark}


\begin{thebibliography}{7}
\bibitem{BMcMN02}
N.~Brady, J.~McCammond, B.~M$\ddot{u}$hlherr, and W.~Neumann, {\em Rigidity of Coxeter groups and Artin groups}, Geometriae Dedicata {\bf 94} (2002) 91--109.
\bibitem{Cr05}
J.~Crisp, {\em Automorphisms and abstract commensurators of 2-dimensional Artin groups}, Geom. Topol. {\bf 9} (2005) 1381--1441.
\bibitem{Day09}
M.~B.~Day, {\em Peak reduction and finite presentations for automorphism groups of right- angled Artin groups}, Geom. Topol. {\bf 13} (2009) 817--855.
\bibitem{Day14}
M.~B.~Day, {\em Full-featured peak reduction in right-angled Artin groups}, Algebr. Geom. Topol. {\bf 14} (2014) 1677--1743.
\bibitem{Droms87}
C.~Droms, {\em Isomorphisms of graph groups}, P. Am. Math. Soc. {\bf 100}  (1987) 407--408.
\bibitem{Godelle03}
E.~Godelle, {\em Parabolic subgroups of Artin groups of type FC}, Pac. J. Math. {\bf 208} (2003) 243--254.
\bibitem{Godelle07}
E.~Godelle, {\em Artin-Tits groups with CAT(0) Deligne complex}, J. Pure Appl. Algebra {\bf 208} (2007) 39--52.
\end{thebibliography}
\end{document}